\documentclass{amsart}
\usepackage[utf8]{inputenc}
\usepackage{latexsym,amssymb,amsthm,amsmath}
\usepackage{graphicx}
\usepackage{tikz-cd}
\usepackage{mathtools}
\usepackage{hyperref}
\usepackage{cleveref}

\usepackage[backend=biber,style=alphabetic,doi=false,isbn=false,url=false,eprint=false]{biblatex}
\numberwithin{equation}{section}

\theoremstyle{plain}
\newtheorem{thm}{Theorem}[section]
\newtheorem{lemma}[thm]{Lemma}
\newtheorem{cor}[thm]{Corollary}
\newtheorem{prop}[thm]{Proposition}
\newtheorem*{thm*}{Theorem}

\theoremstyle{definition}
\newtheorem{defn}[thm]{Definition}
\newtheorem{xmpl}[thm]{Example}
\newtheorem{rmk}[thm]{Remark}

\crefname{thm}{Theorem}{Theorems}
\Crefname{thm}{Theorem}{Theorems}
\crefname{lemma}{Lemma}{Lemmas}
\Crefname{lemma}{Lemma}{Lemmas}
\crefname{prop}{Proposition}{Propositions}
\Crefname{prop}{Proposition}{Propositions}
\crefname{cor}{Corollary}{Corollaries}
\Crefname{cor}{Corollary}{Corollaries}
\crefname{defn}{Definition}{Definitions}
\Crefname{defn}{Definition}{Definitions}
\crefname{xmpl}{Example}{Examples}
\Crefname{xmpl}{Example}{Examples}
\crefname{rmk}{Remark}{Remarks}
\Crefname{rmk}{Remark}{Remarks}

\newcommand{\multa}{\mathcal{MA}}

\DeclareMathOperator{\id}{id}

\DeclareMathOperator{\Hilb}{\mathsf{Hilb}\textsf{ -}}
\DeclareMathOperator{\Hilbrep}{\mathsf{Hilb_{rep}}\textsf{ -}}
\DeclareMathOperator{\KHilb}{\mathcal{K}\mathsf{Hilb}\textsf{ -}}
\DeclareMathOperator{\KHilbB}{\mathcal{K}\mathsf{Hilb}\textsf{-}\mathcal{B}}

\DeclareMathOperator{\sKHilb}{\mathcal{K}\mathsf{Hilb}\textsf{ -}}
\DeclareMathOperator{\fgProj}{\mathsf{fgProj}\textsf{ -}}
\DeclareMathOperator{\HilbB}{\mathsf{Hilb}\textsf{-}\mathcal{B}}
\DeclareMathOperator{\Ob}{\mathsf{Ob}}

\DeclareMathOperator{\Mat}{\mathrm{Mat -}}

\DeclareMathOperator{\Span}{\mathrm{span}}
\DeclareMathOperator{\Hcat}{\mathsf{C^*-HCat}}
\DeclareMathOperator{\Hcatrep}{\mathsf{C^*-HCat_{rep}}}
\DeclareMathOperator{\CCatndg}{\mathsf{C^*-Cat_{ndg}}}

\DeclareMathOperator{\Bimod}{\mathsf{C^*-Bimod}}
\DeclareMathOperator{\A}{\mathcal{A}}
\DeclareMathOperator{\C}{\mathcal{C}}
\DeclareMathOperator{\B}{\mathcal{B}}
\DeclareMathOperator{\D}{\mathcal{D}}
\DeclareMathOperator{\MA}{\mathcal{MA}}
\DeclareMathOperator{\MB}{\mathcal{MB}}
\DeclareMathOperator{\siota}{\iota_{\mathcal{A}}}
\DeclareMathOperator{\CMax}{C_{\textrm{max}}^*}
\title{An Eilenberg-Watts Theorem and Morita localization for Hilbert modules over $C^*$-categories}
\author{Arthur Pander Maat}
\date{\today}

\begin{document}

\maketitle
The $C^*$-category of Hilbert modules over a given $C^*$-category was first defined in \cite{MitchenerCcat}. We prove an Eilenberg-Watts theorem for Hilbert modules over $C^*$-categories, inspired by techniques from \cite{BlecherModules}. We follow by characterizing equivalences of categories of Hilbert modules as being given by tensoring with imprimitivity bimodules, and employ our results to build a reflective localization of the category of locally small $C^*$-categories at the Morita equivalences.
\tableofcontents
\section{Introduction}

A pervasive theme in algebra is to consider some appropriately defined category of modules over a fixed type of object, and ask how we might see that two objects share equivalent categories of modules, i.e. that they are \textit{Morita equivalent}. 
This philosophy of Morita equivalence gets its name from the work of Kiiti Morita who first solved this problem in the case of rings \cite{MoritaDuality}. We refer the reader to \cite{MeyerMorita} for characterizations of Morita equivalences in the cases of rings, $C^*$- and $W^*$-algebras, and topological groupoids.

A common step in solving this problem is a result that states that certain functors from the category of $A$-modules to that of $B$-modules are determined by their precomposition with some embedding of $A$ into its own module category, or more precisely, are always equivalent to tensoring with the $A-B$ bimodule given by this composition. This sort of result is usually called an \textit{Eilenberg-Watts theorem} after simultaneous discoveries by Eilenberg, Watts and Zisman in the ring case: see \cite{WattsIntrinsic} for Watts' contribution. Blecher proved an Eilenberg-Watts theorem for Hilbert modules over $C^*$-algebras in \cite{BlecherModules}. For an exposition of several more such results, including in homotopical algebra and enriched categories, see \cite{EW}. With an Eilenberg-Watts theorem in hand, the problem of Morita invariance reduces to determining \textit{which} bimodules give equivalences upon tensoring.

Finally, a particularly well-behaved modern set-up is a \textit{Morita homotopy category} of these objects, which has isomorphisms between exactly those objects that are Morita equivalent. We might ask moreover that there is a universal functor from the usual category to the Morita homotopy category which inverts those morphisms between objects which give Morita equivalences (i.e. is a \textit{localization} at the Morita equivalences), and hope that there is a concrete description of the hom-sets of the category. This has been achieved for example in the case of unital $C^*$-categories (\cite{IvoGoncalo}), dg-categories (see \cite{TabuadaAdditive}, \cite{ToenMorita}), as well as in the case of $(\infty,1)$-categories (\cite{MoritaHtopyInfty}). One immediately notices that this only seems to be feasible when our objects are \textit{themselves certain kinds of categories}: one explanation for this is that in this scenario, one can invert Yoneda-type functors from each category into its module category.

This paper solves all three of the above problems (an Eilenberg-Watts theorem, a characterization of Morita equivalent objects, and a Morita homotopy theory) in the case of locally small, not necessarily unital $C^*$-categories, which are semicategories enriched over complex Banach spaces with a well-behaved involution on the hom-spaces, built to model the topology and involution on spaces of operators between Hilbert spaces. Proving results on their module theory is considerably harder than the unital case treated in \cite{IvoGoncalo}, since two Morita equivalent unital $C^*$-categories must be Morita equivalent as additive categories\footnote{To be precise, they must have the same closure under subobjects and direct sums: in \cite{IvoGoncalo} this was taken to be the \textit{definition} of Morita equivalence, and we show in \Cref{unitalmorita} that our definition coincides with theirs in the unital case.}, mirroring the fact that unital $C^*$-algebras are Morita equivalent if and only if they're Morita equivalent as rings. For general $C^*$-categories there is no hope of such a result, and we must instead explore the theory of Hilbert modules over $C^*$-categories.

 $C^*$-categories were first defined in \cite{GLR}, and we will see they share many essential properties with $C^*$-algebras: still there are many subtleties in $C^*$-categories that do not arise in the theory of $C^*$-algebras. Section 2 is devoted to clarifying these subtleties, and in particular defines for a $C^*$-category $\A$ its additive closure $\A_{\oplus}$ (\Cref{sumclosuredefn}) and multiplier category $\MA$ (\Cref{multiplierdefn}), as well as clarifying their universal properties and how they interact (\Cref{multiplierprop}, \Cref{multclosure}).

Section 3 focuses on our notion of module over a $C^*$-category $\A$, namely right Hilbert $\A$-modules. We define adjointable and compact operators between these, prove a Yoneda lemma (\Cref{Yoneda}), and in particular we identify morphisms in $\mathcal{A}$ with compact operators between representable modules over $\mathcal{A}$. We also characterize a `strong' topology of pointwise convergence on operators (\Cref{strongtopdefn}) and prove that the multiplier category of the compact operators returns the bounded operators (\Cref{MKA}). We finish by generalizing a result of Blecher's (\cite[Theorem 3.1]{BlecherModules}) which is crucial in his proof of the $C^*$-algebra Eilenberg-Watts theorem, namely that a Hilbert module can asymptotically be viewed as a direct summand in a net of finitely generated free modules: this is the content of \Cref{approxff}.

Section 4 brings the promised results on functors between module categories, beginning with the Eilenberg-Watts theorem:

\begin{thm*}[\Cref{EW}]
    Suppose that $\mathcal{A}$ and $\mathcal{B}$ are $C^*$-categories, and that  $F:\Hilb\mathcal{A}\rightarrow\Hilb\mathcal{B}$ is a unital $C^*$-functor which is strongly continuous on bounded subsets. Let $E:\mathcal{A}\rightarrow\Hilb\mathcal{B}$ be the  bimodule obtained by precomposing $F$ with the Yoneda embedding $\iota_\mathcal{A}:\mathcal{A}\rightarrow\Hilb\mathcal{A}$. Then there exists a unitary isomorphism $F\cong-\bar{\otimes}_\mathcal{A}E$ of functors.
\end{thm*}

Here the strong topology is the topology of pointwise convergence mentioned earlier (see \Cref{strongtopdefn}) and $\bar{\otimes}_\mathcal{A}$ is the (internal) tensor product of Hilbert bimodules, see \Cref{tensordefn}.

We build on this result to characterize which $C^*$-categories are Morita equivalent, and find that the answer is completely analogous to the $C^*$-algebra case:

\begin{thm*}[\Cref{moritathm}]
If $\mathcal{A}$ and $\mathcal{B}$ are $C^*$-categories, the following are equivalent:
\begin{itemize}
    \item There is a $C^*$-functor $F:\Hilb\mathcal{A}\rightarrow\Hilb\mathcal{B}$ which is an equivalence and strongly continuous on bounded subsets.
    \item There exists a right Hilbert $\mathcal{A}-\mathcal{B}$ bimodule  $E:\mathcal{A}\rightarrow\Hilb\mathcal{B}$ which is left small, full, non-degenerate, and an isometry onto the compact operators between modules in its image.
\item There exists a bi-Hilbert $\mathcal{A}-\mathcal{B}$ bimodule $E$ whose inner products are both full and satisfy $_\mathcal{A}\langle e,f\rangle\cdot g=e\cdot\langle f,g\rangle_\mathcal{B}$ for all elements $e\in E(x)(y), f\in E(x')(y)$ and $g\in E(x')(y')$.
\end{itemize} 
\end{thm*}
Here `left small' is a technical set-theoretic condition and `full' signifies, as in the algebra case, that the inner products on the bimodule span a dense ideal of $\mathcal{B}$ (\Cref{fulldefn}). The definition of a bi-Hilbert bimodule is another direct generalization of the algebra case (see \Cref{Bihilbdefn}).

This result allows us to considerably relax the size and unitality restrictions on a result of Joachim's \cite[Proposition 3.2]{JoachimKHom}, which says that a unital $C^*$-category with countably many objects is always Morita equivalent to a $C^*$-algebra, and show that in fact \textit{every $C^*$-category} is Morita equivalent to a $C^*$-algebra, modulo some set-theoretical concerns (see \Cref{catalgequiv} and the remarks below).

Section 5 contains the promised categorical framework for our results. First, we use the Eilenberg-Watts theorem and \Cref{catalgequiv} to construct biequivalences between four bicategories of $C^*$-categories and $C^*$-algebras (\Cref{4bicatequiv}). Finally, after taking some set-theoretical caution by limiting ourselves to \textit{small} Hilbert modules (those generated by a set of elements), we exhibit the localization of the category of locally small $C^*$-categories at the Morita equivalences:

\begin{thm*}[\Cref{MoritaLocalization}]
    Let $\mathbf{C^*Cat}$ be the 1-category whose objects are locally small $C^*$-categories, and whose morphisms from $\A$ to $\B$ are natural isomorphism classes of non-degenerate functors from $\A$ to the multiplier category $\mathcal{MB}$ of $\B$. Let $\sKHilb: \mathbf{C^*Cat}\rightarrow \mathbf{C^*Cat}$ be the endofunctor that takes every $C^*$-category to its $C^*$-category of Hilbert modules and compact operators, and every non-degenerate functor to the tensor product with the bimodule obtained by composing with the embedding $\mathcal{MB}\rightarrow\Hilb\B$. Let $\mathbf{KHilb}$ be the full subcategory with objects those in the image of $\sKHilb$.

    Then $\sKHilb$ exhibits $\mathbf{KHilb}$ as the reflective localization of $\mathbf{C^*Cat}$ at the Morita equivalences, and two $C^*$-categories become isomorphic in $\mathbf{KHilb}$ if and only if they are Morita equivalent.
\end{thm*}

The original results in this paper were obtained as part of my PhD thesis \cite{APMThesis}, supervised by Ivo Dell'Ambrogio and Behrang Noohi. We also thank Benjamin Duenzinger whose notion of small modules was the key to unlock the final result.  While this preprint was in its final stages of preparation, the author was made aware of a PhD thesis (\cite{Ferrier}) containing some overlapping results, defended in late 2022 and made public in April 2023. We have marked original results which also appear there.
\section{$C^\ast$-categories}
We begin with an investigation into the theory of $C^*$-categories. Readers familiar with the theory may skim or skip most of this section, but we include some results that have appeared in other publications since not all of them have proofs, and in order to fix notations.

Some notes on terminology: in this article, we will use the term `category' to refer to a semicategory (a category that doesn't necessarily have units), and `unital category' to refer to what's usually called a category. Similarly a `functor' will signify a semifunctor (i.e. a functorial mapping of semicategories that does not necessarily preserve units, even if they exist) and a functor between unital categories that preserves units will be called a `unital functor'. 

Instead of sticking to small or locally small $C^*$-categories, we work in this section with categories (termed $\A,\B,\C$, etc.) enriched over some bicomplete closed symmetric monoidal category $\mathcal{U}$ of Banach spaces\footnote{An enriching category with these properties is often called a \emph{Benabou cosmos}.}. We do this so that our results carry over to $C^*$-categories such as the category of $C^*$-functors $\B^{\A}$ whose hom-spaces will technically live in an enlargement of $\mathcal{U}$: see \cite[Section 2.6]{KellyEnriched} for details on this construction, and \cite[Chapter 1.1]{APMThesis} for an even more careful treatment of this problem in the instance of $C^*$-categories. Having defined the constructions and proven the results we want for $C^*$-categories of `arbitrary size', from Section 3 on we assume $\A$, $\B$, $\C$ are locally small. 


\subsection{$C^*$-categories and finite direct sums}
We build up to a definition of a $C^*$-category.

\begin{defn}
A \emph{Banach category} is a category $\mathcal{C}$ enriched in a chosen universe $\mathcal{U}$ of complex Banach spaces: so in particular, the composition is bilinear and for all composable $f,g$, we have $\|f\circ g\|\leq\|f\|\|g\|$: we will refer to the latter property as \emph{contractivity}. 
\end{defn}

\begin{defn}
A Banach $\ast$\emph{-category} is a Banach category $\mathcal{C}$ equipped with a conjugate linear involution $(-)^\ast:\mathcal{C}\rightarrow\mathcal{C}^\text{op}$ that fixes objects and squares to the identity functor: that is, for all morphisms $f$, we have $(f^\ast)^\ast=f$.
\end{defn}

\begin{defn}
A $C^*$\emph{-category} is a Banach $\ast$-category $\mathcal{A}$ such that for every morphism $a\in\A(x,y)$: \begin{itemize}
    \item the morphism $a^{\ast}a\in\mathcal{A}(x,x)$ satisfies the $C^*$-identity $\|a^{\ast}a\|=\|a\|^2$, 
    \item and we have the inclusion $$\textrm{Spec}(a^\ast a):=\Bigg\{\lambda\in\mathbb{C}:\begin{array}{cc} & a^*a-\lambda\emph{ is not invertible in the} \\& \emph{minimal}\footnotemark \emph{ unitization } \mathcal{A}(x,x)^+\end{array}\Bigg\} \subseteq\mathbb{R}^{\geq0}.$$ 
\end{itemize}
\end{defn}
\footnotetext{See e.g. \cite[Definition 2.1.6]{Wegge} for details of this construction.}




We give some examples to provide a flavour of the theory.
\begin{xmpl}
$C^*$-algebras are precisely those Banach $\ast$-algebras that occur as the hom-space of a single-object $C^*$-category.
\end{xmpl}
It is with this justification that $C^*$-categories can be regarded as `multi-object' $C^*$-algebras.
\begin{xmpl}
The category $\mathsf{Hilb}$ of Hilbert spaces and bounded operators between them is a $C^*$-category with the familiar operator norm, and involution given by taking adjoint operators.
\end{xmpl}

\begin{xmpl}[{\cite[Definition 5.10]{MitchenerCcat}}]
    If $\mathcal{G}$ is a discrete groupoid, then there is a \emph{maximal groupoid $C^*$-category} $C_{\textrm{max}}^*(\mathcal{G})$ with the same objects as $\mathcal{G}$, whose hom-spaces  $\CMax(\mathcal{G})(x,y)$ are given by taking the free complex vector space on $\mathcal{G}(x,y)$ and completing in the norm given by $\|a\|=\sup_{F:\mathcal{G}\rightarrow\mathsf{Hilb}}F(a)$ (where $F$ varies over unitary representations of $\mathcal{G}$). The involution is given by $(\lambda\cdot g)^*=\overline{\lambda}\cdot g^{-1}$.
\end{xmpl}

There is also a `reduced groupoid $C^*$-category' defined along similar lines in the same publication. For a more general perspective on the maximal groupoid construction, we refer the reader to the series of adjunctions presented on \cite[p. 66]{BunkeVorlesung}, which assemble to give a `free $C^*$-category' on any category with involution. 

A groupoid $C^*$-category in the case where $\mathcal{G}$ is \textit{not discrete} can be defined using the notion of a topological $C^*$-category, i.e. one with a topology on the objects: we do not investigate this construction here and refer the interested reader to \cite[Section 5.1]{SullivanThesis}.

Another important example of a $C^*$-category is as follows:

\begin{xmpl}
For any $C^*$-algebra $B$, the category $\Hilb B$ of right Hilbert $B$-modules and bounded adjointable operators is a $C^*$-category.
\end{xmpl}

This result is a special case of \cite[Proposition 9.4]{MitchenerCcat}. We refer the reader to \cite{LanceHilbert} for a thorough exposition of the theory of Hilbert modules over $C^*$-algebras.

It is in fact the case that much as every  $C^*$-algebra can be embedded into the algebra of all bounded operators on some Hilbert space, every \textit{small} $C^*$-category can be embedded isometrically into the $C^*$-category $\mathsf{Hilb}$ of Hilbert spaces and bounded operators: see \cite[Theorem 6.12]{MitchenerCcat}. The proof involves a direct analog of the Gelfand-Naimark-Segal construction. We won't use this result much here, and stick with an abstract rather than a concrete notion of $C^*$-categories.

Functors between $C^*$-categories get an adapted definition: 
\begin{defn}
If $\mathcal{A}$ and $\mathcal{B}$ are $C^*$-categories, a $C^*$\emph{-functor} from $\A$ to $\B$ is a $\mathbb{C}$-linear functor $F:\A\rightarrow\B$ that intertwines the involutions on $\mathcal{A}$ and $\mathcal{B}$.
\end{defn}
This is called a linear $*$-functor in parts of the literature, but we prefer `$C^*$-functor' for brevity. An example of a functor between $C^*$-categories that is not a $C^*$-functor is the involution: note that when $\A$ is a $C^*$-category, then $\A^{op}$ is also a $C^*$-category, with the same vector space structure, involution and norm: however the involution $(-)^*:\A\rightarrow\A^{op}$ is \emph{not} a $C^*$-functor since it is not linear but conjugate linear. 

One might ask whether there isn't a continuity requirement missing here, but as in the algebraic case, this is determined by the other structure:

\begin{prop}[{\cite[Proposition 2.14]{MitchenerCcat}}]\label{starfuncisbounded}
If $F:\mathcal{A}\rightarrow\mathcal{B}$ is a $C^*$-functor, its associated maps on hom-sets $F:\mathcal{A}(x,y)\rightarrow\mathcal{B}(F(x),F(y))$ are norm-decreasing, so \emph{a fortiori} continuous. 

Furthermore, $C^*$-functors that give injective maps on hom-spaces are isometric.
\end{prop}


\begin{cor}[{\cite[Corollary 2.16]{MitchenerCcat}}]\label{normunique}
The norm on any $C^*$-category is unique: that is, if $\|\cdot\|_1$ and $\|\cdot\|_2$ are two norms on a linear $\ast$-category $\mathcal{C}$ both turning it into a $C^*$-category, then actually $\|\cdot\|_1=\|\cdot\|_2$.
\end{cor}

We define several types of subcategory of a $C^*$-category:

\begin{defn}\label{subcatideal}
If $\mathcal{A}$ is a $C^*$-category, a $\mathbb{C}$-linear subcategory $\mathcal{B}\subseteq\mathcal{A}$ is a \emph{sub-$C^*$-category} if it is norm and involution closed. $\mathcal{B}$ is a \emph{$C^*$-ideal} (or simply an \emph{ideal}) if in addition for all composable $f\in\mathcal{A},g\in\mathcal{B},h\in\mathcal{A}$ we have $fg\in\mathcal{B},gh\in\mathcal{B}$. An ideal $\mathcal{B}$ is \emph{essential} if the following statement is true for all $a\in\A$: if $ab=0$ for all $b\in\B$ such that this composition exist, then $a=0$.
\end{defn}

Note that since there is at least one morphism (the zero morphism) between any two objects in a $C^*$-category, an ideal of $\A$ necessarily contains all objects of $\A$, i.e. it is necessarily a \textit{wide} subcategory.

Note also that using the involution we can deduce an obvious opposite criterion for essentiality: $\mathcal{B}$ is essential if and only if when $a\in\mathcal{A}$ is such that $ba=0$ for all composable $b\in\mathcal{B}$, we in fact have $a=0$.
\begin{prop}\label{essentialcorollary}
If $\B\subseteq\A$ is an essential ideal and $\C\subseteq\A$ is any ideal such that $\C\cap\B=0$, then $\C=0$.
\end{prop}
\begin{proof}
Suppose $\mathcal{C}$ is another ideal of $\A$ such that $\C\cap\B=0$. Then certainly for any composable pair $ c\in\mathcal{C}, b\in\mathcal{B}$, we have $cb\in\mathcal{C}\cap\mathcal{B}$ by ideal properties, so $cb=0$, and the stated property gives $c=0$ for all $c\in\mathcal{C}$. So $\mathcal{B}$ has the stated property.
\end{proof}

\begin{xmpl}\label{minimalunit}
Every $C^*$-category $\mathcal{A}$ is an essential ideal in the $C^*$-category $\mathcal{A}^+$, which has hom-sets
\begin{equation*}
\begin{split}
    \mathcal{A}^+(x,x)&= \mathcal{A}(x,x)\oplus\mathbb{C}\emph{, the minimal unitization}\emph{ of } \mathcal{A}(x,x),\\ &\emph{\hspace{3mm} whenever } \mathcal{A}(x,x)\emph{ is non-unital,} \\ \mathcal{A}^+(x,y) &= \mathcal{A}(x,y) \emph{ for } x\neq y \emph{ or when } x=y \emph{ and } \mathcal{A}(x,x) \emph{ is unital,}
\end{split}
\end{equation*}
and composition and involution defined in the obvious way.
\end{xmpl}
It is elementary to prove that any complex linear ($\ast$-)functor from $\A$ into a unital ($\ast$-)category extends to a complex linear ($\ast$-)functor on $\A^{+}$.



\begin{xmpl}
The category $\mathsf{Hilb}$ of Hilbert spaces and bounded operators has as an essential ideal the subcategory $\mathcal{K}\mathsf{Hilb}$ of Hilbert spaces and compact operators.
\end{xmpl}

Every $C^*$-ideal gives a new topology on  the $C^*$-category containing it; note that by a topology on a category we will always mean a topology on the \emph{hom-sets} of the category, such that composition is continuous.  

\begin{defn}\label{reltop}
For a $C^*$-ideal $\mathcal{B}\subseteq\mathcal{A}$, the $\mathcal{B}$\emph{-relative topology} on $\mathcal{A}(x,y)$ is generated by the seminorms of the form
\begin{align*}
g\mapsto\|fg\|: f\in\mathcal{B}(y,z)\\
g\mapsto\|gh\|:h\in\mathcal{B}(w,x)
\end{align*}
\end{defn}

It follows from the ideal property that multiplication is continuous in any such topology, and it follows from the contractivity of composition that a net which converges in the norm on $\mathcal{A}$ must converge in the $\mathcal{B}$-relative topology for all $\mathcal{B}\subseteq\mathcal{A}$.
Note that even when $\mathcal{B}=\mathcal{A}$, the $\mathcal{A}$-relative topology on $\mathcal{A}$ isn't necessarily the same as the norm topology, unless $\mathcal{A}$ is unital. 

We follow by defining an `internal hom' for $C^*$-categories:
\begin{prop}\label{functorcat}
If $\mathcal{A}$ and $\mathcal{B}$ are $C^*$-categories, there is a $C^*$-category $\mathcal{B}^\mathcal{A}$ whose objects are $C^*$-functors $F:\mathcal{A}\rightarrow\mathcal{B}$ and whose morphisms $\eta:F\rightarrow F'$ are the natural transformations $\eta:F\Rightarrow F'$ such that $\|\eta\|:=\sup_{x\in\Ob\mathcal{A}}\|\eta_x\|<\infty$. The involution is taken objectwise, i.e. by setting $(\eta^*)_x=\eta_x^*$. 
\end{prop}
\begin{proof}
Note first that the involution is well-defined: if $\eta$ is a natural transformation, then for any morphism $a\in\mathcal{A}(x,x')$ we have $$\eta^*_{x'}\circ F(a)=(F(a^*)\circ\eta_{x'})^*=(\eta_x\circ F'(a^*))^*=F'(a)\circ\eta_{x}^*$$ by naturality of $\eta$, so we can apply the involution at every object and get another natural transformation $\eta^*$. Contractivity of composition in $\B^{\A}$ follows directly from that on $\mathcal{B}$. Completeness on the hom-spaces follows as a Cauchy sequence of natural transformations must by definition be Cauchy in every component, and converge to another natural transformation by contractivity.

To prove the $C^*$-identity, note simply that $$\|\eta^*\eta\|:=\sup_{x\in\Ob\mathcal{A}}\|\eta_x^*\eta\|=\sup_{x\in\Ob\mathcal{A}}\|\eta_x\|^2=(\sup_{x\in\Ob\mathcal{A}}\|\eta_x\|)^2=\|\eta\|^2.$$

To show finally that $\eta^*\eta$ has positive spectrum, suppose $\eta^*\eta-\lambda$ is not invertible in $\mathcal{B}^\mathcal{A}(F,F)^+$. Then at least one of the morphisms $\eta_x^*\eta_x-\lambda\in\mathcal{B}(F(x),F(x))^+$ is not invertible, or else the array of their inverses would easily be seen to constitute an inverse natural transformation to $\eta$. But the morphisms $\eta_x^*\eta_x$ each have positive spectrum, so $\lambda$ must be positive and real. 
\end{proof}






It is useful to study morphisms whose involution is also a (right or left) inverse.

\begin{defn}
In a $C^*$-category $\mathcal{A}$, a $\ast$\emph{-monomorphism} or \textit{isometry} is a morphism $a\in\mathcal{A}(x,y)$ such that $a^\ast a=\id_x$. A \emph{unitary isomorphism} is a morphism $a\in\mathcal{A}(x,y)$ such that $a^*a=\id_y$ and $aa^*=\id_x$: in other words, such that $a$ and $a^*$ are both isometries.
\end{defn}
\begin{xmpl}
If $B$ is a $C^*$-algebra, a morphism in the $C^*$-category of right Hilbert $B$-modules is a unitary equivalence if and only if it is a bijection and preserves all inner products.
\end{xmpl}
These special isomorphisms allow us to define an appropriate type of equivalence of $C^*$-categories:
\begin{defn}\label{equivdefn}
If $\mathcal{A}$ and $\mathcal{B}$ are two unital $C^*$-categories, $F:\mathcal{A}\rightarrow\mathcal{B}$ and $G:\mathcal{B}\rightarrow\mathcal{A}$ are $C^*$-functors, and  $\eta:FG\Rightarrow\id_\mathcal{B}$ and $\epsilon:\id_\mathcal{A}\Rightarrow GF$ are natural isomorphisms, 
this data is called a \emph{unitary equivalence} between $\mathcal{A}$ and $\mathcal{B}$ if $\eta$ and $\epsilon$ are both not just isomorphisms but unitary isomorphisms at every object.
\end{defn}

The existence of a unitary isomorphism between two objects is not, however, a stronger condition than that of isomorphism:

\begin{lemma}[{\cite[Proposition 2.6]{zbMATH06107958}}]\label{isoisunitary}
If $a\in\mathcal{A}(x,y)$ is an isomorphism, then there is a unitary isomorphism from $x$ to $y$ given by $a(a^*a)^{-\frac{1}{2}}$.
\end{lemma}
\begin{cor}
    Two unital $C^*$-categories are unitarily equivalent if and only if they are equivalent.
\end{cor}
\begin{proof}
    Simply note that in \Cref{equivdefn}, the transformation $\eta$ is an isomorphism in $\mathcal{A}^\mathcal{A}$, which is a $C^*$-category by \Cref{functorcat}. Hence $\eta$ can be modified to a unitary isomorphism, which will clearly give a unitary isomorphism at every object. One can proceed similarly for $\epsilon$. 
\end{proof}

We turn now to the matter of direct sums in $C^*$-categories. A standard result about additive categories is that finite products and coproducts coincide. In a $C^*$-category we ask in addition that their structure maps are related by the involution:
\begin{defn}
Given a $C^*$-category $\mathcal{A}$ and a finite (possibly repeating) list of objects $x_1,\dots,x_n$ of $\mathcal{A}$ which each have units, we say $\bigoplus_{i=1}^n x_i\in\mathsf{Ob}\mathcal{A}$ is a \emph{direct sum}\footnotemark\hspace{0.5mm}  of $x_1,\dots,x_n$ with structure maps $\iota_i:x_i\rightarrow\bigoplus_{i=1}^n x_i$ if \begin{itemize}
    \item All $\iota_i$ are isometries.
    \item $\sum_{i=1}^n\iota_i\iota_i^*=\id_{\bigoplus_{i=1}^n x_i}$.
\end{itemize} 
\footnotetext{As this is a more specific definition than a direct sum in a general additive category, it might be more appropriate to name this a $*$-direct sum, but since the more general sums do not appear in this article we stick with `direct sum' for economy. Readers versed in dagger categories may be interested to know that this notion of direct sum is an example of a dagger limit as defined in \cite{Heunen}.}
\end{defn}

It follows from the definition that direct sums are unique up to unitary isomorphism. It is possible to add direct sums of objects `synthetically': 


\begin{prop}\label{sumclosuredefn}
If $\mathcal{A}$ is a $C^*$-category, there is a $C^*$-category $\mathcal{A}_{\oplus}$, termed the \emph{additive hull} or \emph{additive closure} of $\mathcal{A}$, where: \begin{itemize}
    \item Objects of $\mathcal{A}_{\oplus}$ are finite lists of $\mathcal{A}$-objects.
    \item $\mathcal{A}_{\oplus}(\{x_1,..,x_n\},\{y_1,..,y_k\}):=[\hom_\mathcal{A}(x_i,y_j)]_{ij}$, that is $k\times n$ matrix arrays of morphisms $x_i\rightarrow y_j$, which compose by matrix multiplication.
    \item The norm on $\mathcal{A}_{\oplus}(\{x_1,..,x_n\},\{y_1,..,y_k\})$ defined by $$\|[f_{ij}]\|:=\emph{sup}\{\|[f_{ij}][b_i]\|:b_i\in\mathcal{A}(w,x_i),w\in\mathsf{Ob}(\mathcal{A}), \|[b_i]\| \leq 1  \}$$where the norms of columns of morphisms are calculated in Hermitian fashion: $\|[b_i]\|:=\sqrt{\|\sum_i b^*_ib_i\|}$.
    \item For any $f=[f_{ij}]_{i,j}\in\mathcal{A}_{\oplus}(\{x_1,..,x_n\},\{y_1,..,y_k\})$ set $f^\ast:=[f^\ast_{ij}]_{j,i}$.
\end{itemize}
\end{prop}
We omit the proof since this construction was also characterized, for instance, in \cite{AntounVoigt}. Our construction differs slightly from existing ones in the sense that it is `intrinsic', i.e. doesn't depend on an embedding into a category of Hilbert spaces, but it's equivalent to older constructions since the hom-spaces, composition and involution are manifestly the same and the $C^*$-norm is unique by \Cref{normunique}.


For all $\mathcal{A}$, there is an evident faithful functor $\mathcal{A}\xhookrightarrow{}\mathcal{A_{\oplus}}$: it sends each object to its corresponding one-object list and each morphism to a $1\times1$ matrix. It has the following property:
\begin{lemma}\label{addextendfunc}
For every $C^*$-category $\mathcal{A}$, every $\mathbb{C}$-linear functor $F$ from $\mathcal{A}$ into a $\mathbb{C}$-linear category $\mathcal{B}$ that admits all finite direct sums factors through the embedding $\mathcal{A}\xhookrightarrow{}\mathcal{A}_{\oplus}$ to a unique extension $F_{\oplus}:\mathcal{A}_{\oplus}\rightarrow\mathcal{B}$. If $\mathcal{B}$ is a $*$-category and $F$ is a $*$-functor, then so is $F_{\oplus}$.
\end{lemma}

\begin{proof}
There's an obvious extension $F_{\oplus}$ on objects that sends a list of $\mathcal{A}$-objects $\{x_1,\dots,x_n\}=\mathbf{x}$ to $F(x_1)\oplus\cdots\oplus F(x_n)$. As to the morphisms from $\mathbf{x}$ to $\mathbf{y}=\{y_1,..,y_k\}$, we define the action $$F_{\oplus}:\mathcal{A}_{\oplus}(\mathbf{x},\mathbf{y})\rightarrow \mathcal{B}(F(x_1)\oplus\cdots\oplus F(x_n),F(y_1)\oplus\cdots\oplus F(y_k))$$  by $[f_{ij}]\mapsto\sum_{i,j}\kappa_j\kappa^*_j\circ F(f_{ij})\circ\iota_i\iota_i^*$, where $\iota_i$ and $\kappa_j$ are the structure maps of the two direct sums in $\mathcal{B}$. This map is clearly linear and intertwines the involutions. The uniqueness and the final statement are left to the reader.
\end{proof}

 It follows from \Cref{sumclosuredefn} that in particular, for every finite list $\{x_1,..,x_n\}$ there exists a $C^*$-algebra $$M_{x_1,..,x_n}(\mathcal{A}):=\mathcal{A}_{\oplus}(\{x_1,..,x_n\},\{x_1,..,x_n\}).$$ It follows immediately from the definition of the matrix norm that these algebras embed isometrically into each other along list inclusions, so we can form a direct limit of all these algebras:
 
 

\begin{defn}\label{matrixalg}
    For a $C^*$-category, its \emph{matrix algebra} $\Mat\mathcal{A}$ is the $C^*$-algebra obtained as the direct limit\footnotemark\hspace{0pt} of the $C^*$-algebras $M_{x_1,..,x_n}(\mathcal{A})$ along all (not necessarily order-preserving) inclusions $\{x_1,..,x_n\}\xhookrightarrow{}\{y_1,\dots y_{n+k}\}$ of \textit{non-repeating finite lists}.
    \end{defn}
\footnotetext{See e.g. \cite[Section 6.1]{Murphy} for an exposition of the theory of direct limits of $C^*$-algebras.}
We note that $\Mat\A$ is unital if and only if $\A$ has a finite set of objects, all with unital endomorphism algebras. Since the inclusions between finite-dimensional matrix algebras are isometric, the inclusions $M_{x_1,..,x_n}(\mathcal{A})\xhookrightarrow{}\Mat\mathcal{A}$ are isometric too: it is elementary to prove that $\Mat\mathcal{A}$ is isomorphic as a $C^*$-algebra to the Banach space direct sum $\bigoplus_{x,y\in\Ob\mathcal{A}}\mathcal{A}(x,y)$, where this space has multiplication and involution given by `coordinate-wise' composition and involution, and the obvious $\ell^2$-norm. This algebra is described as $A_\mathcal{A}$ in \cite{JoachimKHom} and named the `category algebra' in \cite{Ferrier}.
\begin{xmpl}
    If $\A=\CMax(\mathcal{G})$ is the maximal groupoid $C^*$-category on a discrete groupoid $\mathcal{G}$, then $\Mat\A$ is the classical \emph{full groupoid $C^*$-algebra} of $\mathcal{G}$. 
\end{xmpl}

We will return to these algebras in Section 4.4.

We can use the $2\times2$ matrix $C^*$-algebras $M_{x,y}(\mathcal{A})$ to generalize a lemma about $C^*$-algebras to the context of $C^*$-categories.

\begin{lemma}\label{factorlemma}
For any $u\in\mathcal{A}(x,y)$, there exist morphisms  $v\in\mathcal{A}(x,y),w\in\mathcal{A}(x,x)$ such that $u=vw$, as well as morphisms $s\in\mathcal{A}(y,y),t\in\mathcal{A}(x,y)$ such that $u=st$.
\end{lemma}
\begin{proof}
Consider $u$ as $\begin{bmatrix}
0 & 0 \\
u & 0
\end{bmatrix} \in M_{x,y}(\mathcal{A})$: then a standard result\footnotemark\vspace{0pt} about $C^*$-algebras states that there's an element $v=\begin{bmatrix}
v_{11} & v_{12} \\
v_{21} & v_{22}
\end{bmatrix}\in M_{x,y}(\mathcal{A})$ such that $u=v(u^*u)^{\frac{1}{4}}$. The bottom-left corner of this matrix equation reads $u=v_{21}(u^*u)^{\frac{1}{4}}$ where $v_{21}\in\mathcal{A}(x,y)$. Hence setting $v:=v_{21},w:=(u^*u)^{\frac{1}{4}}$ proves the first part of the lemma. Applying this procedure instead to $u^*\in\mathcal{A}(y,x)$ proves the second part.\end{proof}
\footnotetext{See e.g. \cite[1.4.6]{Pedersen} where we set $\alpha=\frac{1}{2}$.}


We deduce a useful fact about the behaviour of approximate units in $C^*$-categories:

\begin{cor}\label{approxunit}
If $(e_\lambda)$ is an approximate unit for the $C^*$-algebra $\mathcal{A}(x,x)$, then for all $a\in\mathcal{A}(x,y)$ we have $\|ae_\lambda-a\|\rightarrow0$ and for all $b\in\mathcal{A}(w,x)$ we have $\|e_\lambda b-b\|\rightarrow0$.
\end{cor}
\begin{proof}
Factorizing $a=vw$ as in \Cref{factorlemma}, we see $$\|ae_\lambda-a\|=\|vwe_\lambda-vw\|\leq\|v\|\|we_\lambda-w\|\xrightarrow[]{\lambda}0$$ The other case follows immediately from the first by setting $y=w$ and applying the involution.
\end{proof}

We end by treating briefly a sort of Cauchy completion for unital $C^*$-categories:

\begin{defn}
    If $\A$ is a $C^*$-category, a \emph{projection} is a morphism $p\in\A(x,x)$ such that $p=p^*$ and $p^2=p$. We say $p$ \emph{splits} if for some $y,z\in\Ob\A$ there is an isomorphism $u:x\xrightarrow[]{\cong} y\oplus z$ such that $upu^{-1}=\begin{bmatrix}
\id_y & 0 \\
0 & 0
\end{bmatrix}$, in which case $y$ is termed the \emph{image} of $p$. We say $\A$ is \emph{idempotent complete} if every projection splits.
\end{defn}

\begin{prop}\label{idempotentcompletion}
    For any unital $C^*$-category $\A$ there is a unital $C^*$-category $\A^{\natural}$ termed the \emph{idempotent completion} of $\A$,  whose objects are pairs $(x,p)$ where $p\in\A(x,x)$ is a projection, and whose morphisms from $(x,p)$ to $(y,q)$ are those morphisms $a$ in $\A(x,y)$ such that $qa=a=pa$.

    The $C^*$-category $\A^{\natural}$ is idempotent complete, contains the isometric image of $\A$ by sending $x$ to $(x,\id_x)$, and if $\B$ is an idempotent complete $C^*$-category, any unital $C^*$-functor $\A\rightarrow\B$ extends to a unital $C^*$-functor $\A^{\natural}\rightarrow\B$.
\end{prop}
We refer the reader to \cite[Section 2.4]{IvoGoncalo} for a proof of this proposition and more details on this construction.
\subsection{The multiplier category and non-degenerate functors}

In this subsection we introduce a generalization of the multiplier algebra construction to $C^*$-categories. This `multiplier category' was first defined in \cite{KandelakiMultiplier} using different techniques than the ones we use here. It was first formulated in a form alike to the one presented here in \cite[Section 2]{VasselliMultipliers}.

We begin by defining the morphisms in the multiplier category.

\begin{defn}\label{singlevarmultmorph}
If $\mathcal{A}$ is a $C^*$-category and $x,y\in\Ob\mathcal{A}$, a \emph{multiplier morphism} from $x$ to $y$ is a pair of maps $(L,R)$, where:
\begin{itemize}
    \item 
$L:\mathcal{A}(x,x)\rightarrow\mathcal{A}(x,y)$ is a map of right $\mathcal{A}(x,x)$-modules.
\item $R:\mathcal{A}(y,y)\rightarrow\mathcal{A}(x,y)$ is a map of left $\mathcal{A}(y,y)$-modules.
\item For all $f\in\mathcal{A}(x,x)$, $g\in\mathcal{A}(y,y)$, we have $R(g)\circ f=g\circ L(f)$.
\end{itemize}
\end{defn}
It follows from standard arguments using approximate units that $L$ and $R$ are automatically complex linear and bounded: see \cite[Proposition 2.2.8]{Wegge} for the $C^*$-algebra case. 

It will be useful for some purposes to have an alternative definition of multiplier morphisms.

\begin{lemma}\label{altmultiplier}
A multiplier morphism from $x$ to $y$ can equivalently be defined as two arrays of bounded maps $\{L_w:\mathcal{A}(w,x)\rightarrow\mathcal{A}(w,y)$\} and $\{R_z:\mathcal{A}(y,z)\rightarrow\mathcal{A}(x,z)\}$, where $w$ and $z$ range over all objects of $\mathcal{A}$, and such that:
\begin{itemize}
    \item For $f\in\mathcal{A}(w,x),h\in\mathcal{A}(w',w)$, we always have $ L_w(f)\circ h=L_{w'}(f\circ h)$.
    \item For $f\in\mathcal{A}(y,z),h\in\mathcal{A}(z,z')$, we always have $ h\circ R_z(f)=R_{z'}(h\circ f)$.
    \item For all $f\in\mathcal{A}(w,x)$, $g\in\mathcal{A}(y,z)$, we have $R_w(g)\circ f=g\circ L_z(f)$.
\end{itemize}
In addition, for all such arrays we have $\sup_{z}\|L_z\|=\|L_x\|=\|R_y\|=\sup_{w}\|R_w\|$.
\end{lemma}

\begin{proof}
    To go from $L:\A(x,x)\rightarrow\A(x,y)$ to an array $\{L_w:w\in\Ob\A\}$, use \Cref{factorlemma} on a morphism $f\in\A(w,x)$ to get $f=e\circ g$ for $e\in\A(x,x)$: then set $L_w(f)=L(e)\circ f'$. To show this is well-defined, suppose $e\circ g=e'\circ g'$ are two such factorizations of $f$: then for all $u\in\A(y,y)$ we have $$u\circ L_x(e)\circ g=R_y(u)\circ e\circ g=R_y(u)\circ e'\circ g'=u\circ L_x(e')\circ g'$$ so letting $u$ vary over an approximate unit we see that $L_x(e)\circ g=L_x(e')\circ g'$. Hence $\{L_w\}$ is well-defined, and it is easily shown to satisfy the axioms above since $L_x$ does. We can similarly go from $R$ to an array $\{R_z\}$.
    
    For the inverse map, clearly an array of multipliers as above gives a multiplier morphism by restricting to $w=x$ and $z=y$.  Using the factorization lemma as above then tells us $\{L_w\}$ is determined by $L_x$, and similarly $\{R_z\}$ is determined by $R_z$: this gives a bijection between arrays and single multipliers.

   It is easy to use approximate units and \Cref{factorlemma} to show the equations on the norms: we leave these verifications to the reader.
\end{proof}

\begin{prop}\label{multiplierdefn}
For a $C^*$-category $\mathcal{A}$, there is a $C^*$-category $\mathcal{MA}$,  enriched in the same universe as $\A$ and called the \emph{multiplier category of $\A$}, which has:

\begin{itemize}
\item Objects identical to those of $\mathcal{A}$.
\item The hom-space $\mathcal{MA}(x,y)$ equal to the collection of multiplier morphisms $(L,R)$ from $x$ to $y$.
\item The norm on $\mathcal{MA}(x,y)$ defined by $\|(L,R)\|:=\|L_x\|=\|R_y\|$.
\item The involution $\mathcal{MA}(x,y)\rightarrow\mathcal{MA}(y,x)$ defined on multipliers in their first characterization by sending $(L,R)\in\mathcal{MA}(x,y)$ to $(L^*,R^*)\in\mathcal{MA}(y,x)$ where $$L^*:\begin{array}{l}\mathcal{A}(y,y)\rightarrow\mathcal{A}(y,x)\\g\mapsto R(g^*)^*\end{array}\text{ and }\begin{array}{l}R^*:\mathcal{A}(x,x)\rightarrow\mathcal{A}(y,x)\\f\mapsto L(g^*)^*.\end{array}$$
\item Composition law defined on multipliers in the form from \Cref{altmultiplier} by setting for $(\{L_u\},\{R_z\})\in\multa(x,y)$ and $(\{L'_u\},\{R'_z\})\in\multa(w,x)$: $$(\{L_u\},\{R_z\})\circ(\{L'_u\},\{R'_z\}):=(\{L_u\circ L'_u\},\{R'_z\circ R_z\}).$$
\end{itemize}
\end{prop}

\begin{proof}[Proof Sketch]
As most of this is proven in \cite[Section 3]{bunke2021additive}, we only need to show that if $\A$ is enriched in $\mathcal{U}$, so is $\MA$. But note that using the first definition of multiplier morphisms, we can describe $\MA(x,y)$ as the limit of a functor whose values are (a finite set of) mapping spaces between hom-spaces of $\A$, so as $\mathcal{U}$ is assumed to be closed and complete, $\MA(x,y)$ lives in $\mathcal{U}$.
\end{proof} There is a $C^*$-functor $\kappa_{\A}:\mathcal{A}\xhookrightarrow{}\mathcal{MA}$ sending each morphism $a$ to the pair of multipliers $\kappa_{\A}(a)=(\ell_a,r_a)$ given by post and pre-composition by that morphism, which is easily seen to be an isometric functor embedding $\mathcal{A}$ as an essential ideal in $\mathcal{MA}$. Furthermore, if $\mathcal{A}$ is unital, $\kappa_{\A}$ is also full: for  any $(L,R)\in\mathcal{MA}(x,y)$, we have $(L,R)=\kappa_{\A}(a)$ where $a=L(\id_x)=R(\id_y)$.

The category $\mathcal{MA}$ also has a universal property akin to that of multiplier $C^*$-algebras:

\begin{lemma}\label{multiplierprop}
If $\mathcal{A}\xhookrightarrow{}\mathcal{D}$ is an embedding of $\mathcal{A}$  as an ideal of a unital $C^*$-category, there exists a unique extension $F:\mathcal{D}\rightarrow\mathcal{MA}$ to $\mathcal{D}$ of the embedding $\kappa_{\A}:\mathcal{A}\xhookrightarrow{}\mathcal{MA}$ and this extension is faithful if and only if $\mathcal{A}$ is essential in $\mathcal{D}$. Furthermore, among all $C^*$-categories containing $\mathcal{A}$ as an ideal, $\mathcal{MA}$ is unique up to isomorphism  with this property.
\end{lemma}
\begin{proof}
The $C^*$-functor $F:\mathcal{D}\rightarrow\mathcal{MA}$ is defined as $F(d)=(L_d,R_d)$ where we set $L_d(a)=da$ and $R_d(b)=bd$:  these elements are in $\mathcal{A}$ by the ideal property. $F$ is easily verified to be an extension of the embedding $\kappa_{\A}:\mathcal{A}\xhookrightarrow{}\mathcal{MA}$. 

If $F$ is faithful, then it acts isometrically on hom-spaces, and clearly $\mathcal{A}$ is essential in $\mathcal{D}$ since it is essential in $\mathcal{MA}$. Conversely, suppose $\mathcal{A}$ is essential in $\mathcal{D}$, and furthermore that two elements $d,d'\in\mathcal{D}(x,y)$ map to the same multiplier in $\mathcal{MA}$; then we immediately see $d-d'=0$ by definition of essentiality, and hence we see that the functor $\mathcal{D}\rightarrow\mathcal{MA}$ is faithful.

The uniqueness part is straightforward; if $\mathcal{C}$ is another category with the stated property, we get an embedding $\mathcal{MA}\xhookrightarrow{}\mathcal{C}$ since $\mathcal{A}$ is essential in $\mathcal{MA}$. But we also get a map $\mathcal{C}\rightarrow\mathcal{MA}$, which can easily be checked to be a left and right inverse by calling on the uniqueness.
\end{proof}

\begin{defn}\label{ultrastrongdefn} The $\mathcal{A}$-relative topology on $\mathcal{MA}$ is called the \textit{ultrastrong topology}.
\end{defn}

The intuition behind this name comes from the case where $\mathcal{A}$ is the category of Hilbert spaces and compact operators: in this case $\mathcal{MA}$ is the category of Hilbert spaces and bounded operators, as will be proven below. Our terminology diverges here from  \cite{AntounVoigt} and \cite{bunke2021additive}, where it is labelled the \textit{strict topology}.


We note finally that the multiplier operation $\mathcal{M}(-)$ commutes with the additive closure $(-)_{\oplus}$ defined in subsection 2 of this section.
\begin{lemma}\label{addmult}
There is a canonical isomorphism of $C^*$-categories $(\mathcal{MA})_{\oplus}\cong\mathcal{M}(\mathcal{A}_{\oplus})$.
\end{lemma}
We omit the proof and refer the interested reader to \cite[Proposition 1.3.9]{APMThesis}.

As $C^*$-functors are often between non-unital categories, it is not always possible to ask for them to be unital. This throws up a variety of problems, which have to be fixed by requiring our $C^*$-functors to be `approximately unital'. One way of formulating this requirement is to ask that a $C^*$-functor $\mathcal{A}\rightarrow\mathcal{B}$ preserve approximate units of endomorphism algebras; it is useful, however, to generalize at this point instead to $C^*$-functors $\mathcal{A}\rightarrow\mathcal{MB}$.

\begin{defn}\label{nondegendefn}
A $C^*$-functor $F:\mathcal{A}\rightarrow\mathcal{MB}$ is said to be \emph{non-degenerate} when for every $x,y\in\mathsf{Ob}\mathcal{A}$, the subspace $F\mathcal{A}(y,y)\circ\mathcal{B}(Fx,Fy)$ is norm-dense in $\mathcal{B}(Fx,Fy)$. A $C^*$-functor $F:\mathcal{A}\rightarrow\mathcal{B}$ is called non-degenerate if its composition with the inclusion $\mathcal{B}\xhookrightarrow{}\mathcal{MB}$ is non-degenerate.
\end{defn}

We take this definition from \cite{AntounVoigt}, where non-degenerate functors form the 1-morphisms in the 2-category $C^*$-Lin.

We want to provide several equivalent formulations of non-degeneracy, but first we need a technical lemma to streamline the proof.

\begin{lemma}\label{limsuptrick}
If $V$ and $W$ are Banach spaces, $D$ is a dense subset of $V$, and $(T_\lambda:V\rightarrow W)$ is a uniformly bounded net of operators such that $\|T_\lambda(d)\|\xrightarrow[]{\lambda}0$ for each $d\in D$, then in fact $\|T_\lambda(v)\|\xrightarrow[]{\lambda}0$ for each $v\in V$.
\end{lemma}
\begin{proof}
Let $K>0$ be a uniform bound on the net, and for any $v\in V$ pick a sequence $(d_n)$ converging to $v$. For an arbitrary $\epsilon>0$ pick an integer $n$ such that $\|d_n-v\|\leq\frac{\epsilon}{K}$. Then for each $\lambda$ we have $$\begin{array}{rl}\|T_\lambda(v)\|&\leq \|T_\lambda(v-d_n)\|+\|T_\lambda(d_n)\|\\&\leq\epsilon+\|T_\lambda(d_n)\|.\end{array}$$
Hence taking $\limsup_\lambda$ on both sides we obtain that $\limsup_\lambda\|T_\lambda(v)\|\leq\epsilon$, hence $\lim_\lambda\|T_\lambda(v)\|=0$.
\end{proof}
\begin{cor}\label{limsupid}

Suppose $V$ is a Banach space, $D$ is a dense subspace of $V$, and $(U_\lambda:V\rightarrow V)$ is a uniformly bounded net of operators such that $U_\lambda(d)\xrightarrow[]{\lambda}d$ for each $d\in D$. Then we have in fact $U_\lambda(v) \xrightarrow[]{\lambda}v$ for each $v\in V$.
\end{cor}
\begin{proof}
Simply take $W=V$, $T_\lambda=U_\lambda-\id_V$ and apply \Cref{limsuptrick}.
\end{proof}

We can now show the equivalence of several definitions of a non-degenerate functor:

\begin{thm}\label{nondegenstrongthm}
If $\mathcal{A}$ and $\mathcal{B}$ are $C^*$-categories, the following criteria on a $C^*$-functor $F:\mathcal{A}\rightarrow\mathcal{MB}$ are equivalent:
\begin{enumerate}

\item $F$ is non-degenerate.
\item For every $x,y\in\mathsf{Ob}\mathcal{A}$, the subspace $\mathcal{B}(Fx,Fy)\circ F\mathcal{A}(x,x)$ is norm-dense in $\mathcal{B}(Fx,Fy)$.
\item For every $x,y\in\mathsf{Ob}(\mathcal{A})$, approximate unit $(u_\lambda)_{\lambda\in\Lambda}$ for $\mathcal{A}(y,y)$, and morphism $b\in\mathcal{B}(Fx,Fy)$, we have $F(u_\lambda)\circ b\rightarrow b$.
\item For every $x,y\in\mathsf{Ob}(\mathcal{A})$, approximate unit $(u_\lambda)_{\lambda\in\Lambda}$ for $\mathcal{A}(x,x)$, and morphism $b\in\mathcal{B}(Fx,Fy)$, we have $b\circ F(u_\lambda)\rightarrow b$.
\item $F$ extends (uniquely) to a unital $C^*$-functor $\bar{F}:\mathcal{MA}\rightarrow\mathcal{MB}$, which on norm-bounded subsets is in addition continuous with respect to the ultrastrong topologies on $\mathcal{MA}$ and $\mathcal{MB}$, as defined in \Cref{ultrastrongdefn}.
\end{enumerate}
\end{thm}
\begin{proof}
Criteria 1 and 2, as well 3 and 4 are obviously equivalent using the involution. Criterion 3 clearly implies 1, and using \Cref{limsupid} it is easy to show that 1 implies 3. Criterion 5 clearly implies 3, since approximate units are nets ultrastrongly converging to the identity.

Finally, it is the content of \cite[Proposition 3.16]{bunke2021additive} that criteria 1 and 2 together imply criterion 5. \end{proof}

\begin{rmk}It follows from the final criterion that if $\mathcal{A}$ is unital, a non-degenerate functor $\mathcal{A}\rightarrow\mathcal{MB}$ is simply a unital functor $\A\rightarrow\mathcal{MB}$. To see this, note that the ultrastrong topology on $\MA\cong\A$ is just the norm topology,  any $C^*$-functor is norm-continuous, and norm convergence in $\mathcal{MB}$ 
implies ultrastrong convergence. 
\end{rmk}

\begin{rmk}
It is clear from the third criterion that if $G:\mathcal{A}\rightarrow\mathcal{MB}$ and $F:\mathcal{B}\rightarrow\mathcal{MC}$ are two non-degenerate functors, we can compose their lifts to get $\bar{F}\circ\bar{G}:\mathcal{MA}\rightarrow\mathcal{MC}$. This composition is again strongly continuous on bounded subsets as certainly $\bar{F}$ sends bounded subsets to bounded subsets (it's a $C^*$-functor so norm-decreasing), hence it restricts to a non-degenerate functor $\bar{F}\circ G$.
\end{rmk}
\begin{xmpl}
The embedding $\kappa_{\A}:\mathcal{A}\xhookrightarrow{}\mathcal{MA}$ is non-degenerate: we can see this by combining criterion 3 in \Cref{nondegenstrongthm} with \Cref{approxunit}.
\end{xmpl}
\begin{xmpl}
Any full and faithful $C^*$-functor $F:\mathcal{A}\rightarrow\mathcal{B}$ is non-degenerate; it preserves approximate units of endomorphism algebras so criterion 3 is satisfied. Hence for example, the embedding $\mathcal{A}\xhookrightarrow{}\mathcal{A}_{\oplus}$ is non-degenerate.
\end{xmpl}

This theory allows us to talk about unitary equivalences between non-unital $C^*$-categories:

\begin{defn}\label{multequiv}
A non-degenerate $C^*$-functor $F:\mathcal{A}\rightarrow\mathcal{MB}$ is termed a \emph{(unitary) multiplier equivalence} if there is a $C^*$-functor $G:\mathcal{B}\rightarrow\mathcal{MA}$ such that $\bar{F}$ and $\bar{G}$ are inverse (unitary) equivalences between $\mathcal{MA}$ and $\mathcal{MB}$. A $C^*$-functor $F:\mathcal{A}\rightarrow\mathcal{B}$ is termed a (unitary) multiplier equivalence if its composition with the embedding $\kappa_{\mathcal{B}}:\mathcal{B}\xhookrightarrow{}\mathcal{MB}$ is non-degenerate and a (unitary) multiplier equivalence.
\end{defn}

We end with a discussion of multiplier direct sums. This construction solves several problems at once: the first is that the usual definition of direct sums involves identities on objects, which may not exist in a $C^*$-category.  The second is that it allows us to sum an infinite number of objects using the topological structure of a $C^*$-category. The definition below is taken from \cite{AntounVoigt}, where it is called simply a `direct sum'.

\begin{defn}\label{multiplierdirectsumdefn}
Given a $C^*$-category $\mathcal{A}$ and a multiset of objects $\{x_i:i\in I\}$ of $\mathcal{A}$, we say $\bigoplus_{i\in I} x_i\in\Ob\mathcal{A}$ is a \emph{multiplier direct sum} of $\{x_i:i\in I\}$ with structure maps $\iota_i\in\mathcal{MA}(x_i,\bigoplus_{i\in I} x_i)$ if: \begin{itemize}
    \item All maps $\iota_i$ are isometries.
    \item The net of partial sums $\sum_{j\in J}\iota_j\iota_j^*$ for finite sub-multisets $J\subseteq I$ converges to $\id_{\bigoplus_{i\in I} x_i}$ in the ultrastrong topology, as defined in \Cref{reltop}.
\end{itemize}
\end{defn}
Notice firstly that in the case that $I$ is a finite multiset, the above data (together with a choice of total order on the set) just defines a finite direct sum in $\mathcal{MA}$.


It is easy to see that a non-degenerate functor $F:\mathcal{A}\rightarrow\mathcal{MB}$ preserves multiplier direct sums. We close on two results that relate the additive hull to finite multiplier direct sums.

\begin{lemma}
$\mathcal{A}_{\oplus}$ admits all finite multiplier direct sums.
\end{lemma}

\begin{proof}
The structure map $(L_i,R_i)=\iota_i\in\mathcal{M(A_{\oplus})}(\{x_i\},\{x_1,\dots,x_n\})$ is given on one coordinate by the map $L_i:\mathcal{A}_{\oplus}(\{x_i\},\{x_i\})\rightarrow\mathcal{A}_{\oplus}(\{x_i\},\{x_1,\dots,x_n\})$ that sends a morphism $f$ to the column with $f$ in the $i$-th place and zeros elsewhere. The right multiplier works similarly. It is elementary to verify that these multipliers exhibit $\{x_1,\dots,x_n\}$ as the multiplier direct sum of the $x_i$.
\end{proof}

This allows us to formulate a kind of universal property for $\mathcal{A}_{\oplus}$:

\begin{lemma}\label{multclosure}
If $\mathcal{B}$ is a $C^*$-category closed under finite multiplier direct sums, any non-degenerate $C^*$-functor $F:\mathcal{A}\rightarrow\mathcal{MB}$ extends uniquely to a non-degenerate functor $F_{\oplus}:\mathcal{A}_{\oplus}\rightarrow\mathcal{MB}$.
\end{lemma} 
\begin{proof}

  To define $F_{\oplus}$, consider the unique unital and ultrastrongly continuous  extension $\bar{F}:\mathcal{MA}\rightarrow\mathcal{MB}$. As $\mathcal{MB}$ is closed under finite direct sums, this extends to a $C^*$-functor $\bar{F}_{\oplus}:(\mathcal{MA})_{\oplus}\rightarrow\mathcal{MB}$ by \Cref{addextendfunc}. By \Cref{addmult} we can also write this as a functor $\bar{F}_{\oplus}:\mathcal{M}(\mathcal{A}_{\oplus})\rightarrow\mathcal{MB}$. It is easy to deduce that $\bar{F}_{\oplus}$ is unital and ultrastrongly continuous from the fact that $\bar{F}$ is, hence it restricts to a non-degenerate functor $F_{\oplus}:\mathcal{A}_{\oplus}\rightarrow\mathcal{MB}$.
  \end{proof}

\section{Hilbert modules over $C^*$-categories}


We develop the theory of Hilbert modules over $C^*$-categories, first investigated in \cite{MitchenerCcat} and \cite{JoachimKHom}. 
\subsection{Hilbert modules and bounded adjointable operators}
From now on we assume for simplicity that $\A$, $\B$, $\C$, etc. are \emph{locally small} $C^*$-categories, although virtually all the theory below can be developed over an arbitrary universe $\mathcal{U}$ like the above\footnote{One does this by defining modules as functors $\A\rightarrow\mathcal{U}$ satisfying axioms like the ones listed in \Cref{hilbmoddefn}}. We will denote by $\mathsf{Vect}_{\mathbb{C}}$ the category of (small) complex vector spaces.
\begin{defn}\label{hilbmoddefn}
If $\mathcal{A}$ is a (locally small) $C^*$-category, a \emph{large right Hilbert }$\mathcal{A}$\emph{-module} is a $\mathbb{C}$-linear functor $E:\mathcal{A}^{op}\rightarrow\mathsf{Vect}_{\mathbb{C}}$ equipped with sesquilinear\footnotemark\hspace{0pt}  `inner products' $$\langle-,-\rangle_E:E(y)\times E(x)\rightarrow\mathcal{A}(x,y)$$ such that for all $e\in E(y), f\in E(x)$ and $g\in\mathcal{A}(w,x)$ we have \begin{itemize}
    \item $\langle e,f\rangle=\langle f,e\rangle^*$.
    \item $\langle e,f\rangle\circ g=\langle  e,f\cdot g\rangle$, where $f\cdot g:=E(g)(f)$.
    \item $\langle e,e\rangle$ is a positive element of the $C^*$-algebra $\mathcal{A}(x,x)$.
    \item $\langle e,e\rangle=0$ only if $e=0$.
    \item For each $x\in\Ob\mathcal{A}$, the space $E(x)$ is complete in the norm $$\|e\|:=\sqrt{\|\langle e,e\rangle\|_{\mathcal{A}(x,x)}}.$$
\end{itemize}
\end{defn}
\footnotetext{Linear in the second argument and conjugate linear in the first.}
Note that for any Hilbert module, a sort of conjugate linearity in the first argument follows by combining the first and second axioms: $$\langle f\cdot g, e\rangle= \langle e, f\cdot g\rangle^*=(\langle e, f\rangle\circ g)^*=g^*\circ\langle e,f\rangle^*=g^*\circ\langle f,e\rangle.$$

We will not be concerned for the moment with left Hilbert $\mathcal{A}$-modules, though if desired these can easily be defined as right Hilbert $\mathcal{A}^{\mathrm{op}}$-modules; alternatively, one can modify the above axioms by making $E$ a covariant functor (so that elements of $\mathcal{A}$ act on the left), requiring that the product $\langle-,-\rangle_E$ be linear in the first variable and conjugate linear in the second, and that $\langle a\cdot e,f\rangle=a\circ\langle e,f\rangle$.
\begin{xmpl}
If $\mathcal{A}$ has only one object $x$, the large Hilbert modules over $\mathcal{A}$ are exactly the right Hilbert $C^*$-modules over the $C^*$-algebra $\mathcal{A}(x,x)$. In particular, complex Hilbert spaces are Hilbert modules over the one-object $C^*$-category with hom-space $\mathbb{C}$.
\end{xmpl}

\begin{xmpl}\label{repdefn}
For any $C^*$-category $\mathcal{A}$ and any object $x$ of $\mathcal{A}$, there is a `representable' large Hilbert module $$h_x:\begin{array}{l}\mathcal{A}\rightarrow\mathsf{Vect}_\mathbb{C}\\y\mapsto\mathcal{A}(y,x)\end{array}$$ with inner product defined for all morphisms $a\in h_x(y),b\in h_x(z)$ by $$\langle a,b\rangle_{h_x}:=a^*b\in\mathcal{A}(z,y).$$
\end{xmpl}

Studying \Cref{hilbmoddefn}, one quickly notices that unless $\A$ has a \emph{set} of objects, we cannot expect the union of all component spaces $E(x)$ to form a set: this will not be a problem for our theory as long as we ask that there is some set of elements \emph{generating} $E$, as in the theory of small presheaves for locally small categories (see e.g. \cite{nlab:small_presheaf}). The author first became aware of this helpful notion through Benjamin Duenzinger's master's thesis, and later discovered it had also appeared in Simon Henry's paper \cite{henry2015complete}. A good chunk of our theory can be developed without it (and was, in previous versions of this paper) but as we do not know of many interesting `large' Hilbert modules over locally small $C^*$-categories and as their size would obstruct some of the constructions and results in Sections 4 and 5, we restrict attention to these `small' Hilbert modules.

Below, we use \emph{set} to mean a \emph{small set} (say in ZFC), \emph{collection} to mean a set or class, and $\mathsf{Vect}_{\mathbb{C}}$ to be the category of small complex vector spaces.  
\begin{defn}
For a right Hilbert $\A$-module $E:\A^{op}\rightarrow\mathsf{Vect}_{\mathbb{C}}$ and a collection $S$ of elements $s\in E(x_s)$ (where each $x_s\in\Ob\A$) we let  $\langle S\rangle:\A^{op}\rightarrow\mathsf{Vect}_{\mathbb{C}}$ be the smallest subfunctor of $E$ containing $S$ which is closed under scalar multiplication, addition, and action from morphisms in $\A$, and such that its value on each object is closed in the norm inherited from $E$. Put more succinctly, $\langle S\rangle$ is the smallest sub-Hilbert module of $E$ containing $S$. We say that $S$ \emph{generates} $E$ when $\langle S\rangle=E$.

\end{defn}

\begin{lemma}\label{generatedsetisclosure}
For $E$ and $S$ as above, the module $\langle S\rangle$ is given at $x\in\Ob\A$ by the norm closure of all finite linear combinations $\Sigma_{i=1}^n\mu_is_i+ \Sigma_{j=1}^m s_j\cdot a_j$, where for all $i$ and $j$ we have $s_i\in S\cap E(x),\mu_i\in\mathbb{C},s_j\in S$, and $a_j\in\A(x,x_{s_j})$.
\end{lemma}
\begin{proof}
Call the second module in the above statement $[S]$. Note firstly that any finite linear combination of the type defined should be in $\langle S\rangle$, which is also norm-closed, giving us that $[S]\subseteq\langle S\rangle$. Conversely, clearly $[S]$ is a Hilbert submodule of $E$ containing $S$, giving that $\langle S\rangle\subseteq [S]$, and hence we see $[S]=\langle S\rangle$.
\end{proof}

\begin{defn}
A large right $\A$-Hilbert module $E:\A^{op}\rightarrow\mathsf{Vect}_{\mathbb{C}}$ is termed a \emph{small} Hilbert module if there is a \emph{set} $S$ of elements $s\in E(x_s)$, where each $x_s\in\Ob\A$, such that $S$ \emph{generates} $E$.
\end{defn}
Small Hilbert modules were first defined in \cite[Definition 2.2.21]{henry2015complete}. We easily show that the `large' Hilbert modules we've looked at so far are in fact small:
\begin{lemma}\label{smallhilboversmallcat}
    If $\A$ is a small $C^*$-category then every right Hilbert $\A$-module $E$ is in fact small.
\end{lemma}
\begin{proof}
As $\A$ has a set of objects we can simply let $S$ be the set $\bigcup_{x\in\Ob\A}E(x)$ of all elements of $E$.
\end{proof}
\begin{lemma}
If $\A$ is locally small then for each $x\in\Ob\A$ we have that $h_x$ is a small Hilbert module.
\end{lemma}
\begin{proof}
Take an approximate unit $(u_\lambda)_{\lambda\in\Lambda}$ for $\A(x,x)$ where $\Lambda$ is a set, and note that \Cref{approxunit} implies that each element $a\in h_x(y)=\A(y,x)$ is the norm limit of $ u_\lambda\circ  a=u_\lambda\cdot a$. Hence $S=\{u_\lambda:\lambda\in\Lambda\}\subseteq h_x(x)$ generates $h_x$.
\end{proof}

\begin{rmk}Being satisfied with these two examples, from now on we will drop `small' and take `right Hilbert $\A$-module' to mean `small right Hilbert $\A$-module'.\end{rmk}

The action of a $C^*$-category on a Hilbert module is automatically continuous:

\begin{lemma}\label{continuity}
If $E$ is a right Hilbert $\mathcal{A}$-module, then for all $e\in E(x), a\in\mathcal{A}(w,x)$ the inequality $$\|e\cdot a\|_{E(w)}\leq\|e\|_{E(x)}\|a\|_{\mathcal{A}(w,x)}$$ holds. In particular, the action of $\mathcal{A}$ on $E$ is continuous.
\end{lemma}
\begin{proof}
Note that we have $$\|e\cdot a\|_{E(w)}^2=\|\langle e\cdot a,e\cdot a\rangle\|_{\mathcal{A}(w,w)}=\|a\circ \langle e,e\rangle\circ a^*\|_{\mathcal{A}(w,w)}$$ $$\leq \|a\|_{\mathcal{A}(x,w)} \|\langle e,e\rangle\|_{\mathcal{A}(x,x)}\|a^*\|_{\mathcal{A}(x,w)}=\|a\|^2\|e\|_{E(x)}^2.$$ Hence taking square roots we obtain our result.
\end{proof}
There is also a Cauchy-Schwartz lemma we can prove for these Hilbert modules.
\begin{prop}\label{CauchyS}
If $E$ is a right Hilbert $\mathcal{A}$-module, $e\in E(x), f\in E(y)$, then the following Cauchy-Schwartz inequality holds in the $C^*$-algebra $\mathcal{A}(x,x)$: $$\langle e,f\rangle\circ\langle f,e\rangle\leq\|\langle f,f\rangle\|\langle e,e\rangle. $$
\end{prop}
The proof of this proposition, which we omit here, is more or less a verbatim transcription of the proof of \cite[Proposition 1.1]{LanceHilbert}.




\begin{cor}\label{CSLemma}
If $E$ is a right Hilbert $\mathcal{A}$-module, $e\in E(x), f\in E(y)$, then the following Cauchy-Schwartz inequality holds in $\mathbb{R}$: $$\|\langle f,e\rangle\|\leq\|f\|\|e\|.$$
\end{cor}
\begin{proof}
Applying norms to \Cref{CauchyS}, we get $$\|\langle f,e\rangle\|^2=\|\langle f,e\rangle^*\circ\langle f,e\rangle\|=\|\langle e,f\rangle\circ\langle f,e\rangle\|\leq\|\langle f,f\rangle\|\|\langle e,e\rangle\|=\|f\|^2\|e\|^2.$$
\end{proof}

We establish a few basic density results that will come in handy later: 
\begin{lemma}\label{density}
If $E$ is a right Hilbert $\mathcal{A}$-module, then for each object $x$ of $\mathcal{A}$, the subspace $E(x)\cdot\langle E(x),E(x)\rangle$ is dense in $E(x)$, where  $\langle E(x),E(x)\rangle$ is the linear span of all inner products of elements of $E(x)$.
\end{lemma}
\begin{proof}
Note it follows from the Hilbert module axioms that $E(x)$ is a Hilbert module over the $C^*$-algebra $\A(x,x)$, so this follows from the analogous lemma for $C^*$-algebras, see e.g. \cite[p.5]{LanceHilbert}.



\end{proof}
\begin{cor}\label{unitcor}If $(u_\lambda)$ is an approximate unit for $\mathcal{A}(x,x)$ and $e\in E(x)$, we have $e\cdot u_\lambda\xrightarrow[]{\lambda} e$. 
\end{cor}
\begin{proof}
\textit{A fortiori} from the above lemma we see that $E(x)\cdot\mathcal{A}(x,x)$ is always dense in $E(x)$, and we deduce the corollary using \Cref{limsuptrick}.
\end{proof}

We can now show that every element of a Hilbert module is in the image of the action of some $\A$-morphism. This is done by combining \Cref{unitcor} with the following classical result: 
\begin{thm}[{Cohen-Hewitt Theorem, see \cite[p. 108]{CJL}}]\label{CH}
Suppose $A$ is a Banach algebra with approximate unit $(u_\lambda)$, and $M$ is a right Banach $A$-module (that is, an $A$-module with a Banach norm making the action of $A$ continuous). Suppose $m\in M$ is such that $m\cdot u_\lambda\xrightarrow[]{\lambda}m$. Then there exist $n\in M,a\in A$ such that $m=n\cdot a$.
\end{thm}

\begin{cor}\label{CohHewCstar} 
If $\mathcal{A}$ is a $C^*$-category and $E$ is a right Hilbert $\mathcal{A}$-module, then for any $x\in\Ob\mathcal{A},e\in E(x)$ there exists $f\in E(x),a\in\mathcal{A}(x,x)$ such that $e=f\cdot a$.
\end{cor}




Next, we aim to define the $C^*$-category of right Hilbert $\mathcal{A}$-modules. To do this we must define morphisms, an involution, and a norm. 

\begin{defn}\label{opdefn}
If $E,F:\mathcal{A}^{\emph{op}}\rightarrow\mathsf{Vect}_{\mathbb{C}}$ are two right Hilbert $\A$-modules, then an \emph{operator} $T:E\rightarrow F$ is simply a natural transformation $T:E\Rightarrow F$ of the underlying functors.  We say $T$ is \emph{bounded} if the set $\{\|T(e)\|:x\in\textsf{Ob}\mathcal{A}, e\in E(x), \|e\|=1\}\subseteq\mathbb{R}$ is bounded, in which case we denote its supremum by $\|T\|$.
\end{defn}
We also ask that our operators are \textit{adjointable}:

\begin{defn}
An operator $T^*:F\rightarrow E$ of right Hilbert $\mathcal{A}$-modules is an \emph{adjoint} for the bounded operator $T:E\rightarrow F$ when we have $$\langle Te,f\rangle_F=\langle e,T^*f \rangle_E\in\mathcal{A}(x,y) \textrm{ for all } x,y\in\Ob\mathcal{A}, e\in E(x),f\in F(y).$$ We will denote the space of bounded adjointable operators from $E$ to $F$  by $\mathcal{L}(E,F)$, and shorten $\mathcal{L}(E,E)$ to $\mathcal{L}(E)$.
\end{defn}

\begin{rmk}We will often drop `bounded' and `adjointable' and speak simply of operators, except where we have to prove that an operator is bounded and has an adjoint.\end{rmk}

\begin{prop}
If $\A$ is a locally small $C^*$-category, the category $\Hilb\mathcal{A}$ of right Hilbert $\mathcal{A}$-modules and bounded, adjointable operators is a locally small and unital $C^*$-category with the norm given in \Cref{opdefn} and involution given by adjoint operators. 
\end{prop}
\begin{proof}
    All of this is in \cite{MitchenerCcat}, save for the local smallness, which we prove here. Suppose $E$ and $F$ are right Hilbert $\A$-modules. If $S$ is a generating set for $E$, it is clear from \Cref{generatedsetisclosure} that any bounded adjointable map $T\in\mathcal{L}(E,F)$ is determined by its image on $S$. But each element $s\in E(x_s)$ can only be sent to an element in the set $F(x_s)$, and hence we see there's an injective map from $\mathcal{L}(E,F)$ to the set $\sqcup_{s\in S} F(x_s)$.
\end{proof}

The $C^*$-category $\Hilb\mathcal{A}$ is closed under finite direct sums. 
\begin{defn}\label{findirsumhilb}
If $\mathcal{A}$ is a $C^*$-category and $E_1,\dots,E_n$ is a finite list of right Hilbert $\mathcal{A}$-modules, their \emph{direct sum} $\oplus_{i=1}^nE_i$ is given by 
\begin{itemize}
    \item $(\oplus_{i=1}^nE_i)(x):=\oplus_{i=1}^nE_i(x)$
    \item $\langle (e_1,\dots,e_n),(f_1,\dots,f_n)\rangle_{\oplus_{i=1}^nE_i}:=\langle e_1,f_1\rangle_{E_1}+\dots+\langle e_n,f_n\rangle_{E_n}$
\item Structure maps $\iota_i:E_i\rightarrow\oplus_{i=1}^nE_i$ given by inserting an element into the $i^{\textrm{th}}$ coordinate, with adjoint given by projecting.
\end{itemize}

\end{defn}
We leave it to the reader to verify that $\oplus_{i=1}^nE_i$ is in fact a Hilbert module, noting that when $S_i$ is a generating set for $E_i$, then $\oplus_{i=1}^nE_i$ is generated by the set $\sqcup_{i=1}^{n}\iota_i(S_i)$, and also that the structure maps exhibit it as the direct sum of the $E_i$.

\begin{prop}
$\Hilb\A$ is idempotent complete.
\end{prop}
\begin{proof}
    Consider a projection $P\in\mathcal{L}(E,E)$, that is, a bounded adjointable map such that $P^*=P$ and $P^2=P$. Then it is easily verified that  $\textrm{ker}P$ and $\textrm{im}P$, taken objectwise, are Hilbert submodules of $E$ and that $E\cong\textrm{ker}P\oplus\textrm{im}P $.
\end{proof}

\begin{defn}
    A right Hilbert $\A$-module $E$ is said to be \emph{finitely generated free} if there is a list of $\A$-objects $x_1,\dots,x_n$ such that $\oplus_{i=1}^n h_{x_i}\cong E$, where these are the representable $\A$-modules from \Cref{repdefn}. We say $E$ is \emph{finitely generated projective} if it is a direct summand of a finitely generated free module.  
\end{defn}

We close with a characterization of the unitary isomorphisms in $\Hilb\mathcal{A}$. 


\begin{prop}\label{unitarycrit}
If $E$ and $F$ are right Hilbert $\mathcal{A}$-modules, and $T:E\Rightarrow F$ is a natural transformation from $E$ to $F$, ($E$ and $F$ considered purely as functors) then $T$ is a unitary isomorphism of right Hilbert $\mathcal{A}$-modules if and only if $T$ is surjective at every $x\in\Ob\mathcal{A}$ and $T$ preserves all inner products.
\end{prop}
\begin{proof}
Note that if $T$ preserves inner products it must be an isometry on each single object, so as $T$ is also surjective on each object it must have an object-wise inverse $T^{-1}$, which is in addition natural since an array of isomorphisms is natural if and only if its inverse is. But then for each $e\in E(x), f\in F(y)$ we have $$\langle e,T^{-1}(f)\rangle_E=\langle T(e),TT^{-1}(y)\rangle_F=\langle T(e),f\rangle_F$$ showing that $T^{-1}=T^*$. Hence $T$ is a unitary isomorphism of $\mathcal{A}$-modules. The converse is obvious.
\end{proof}
\subsection{Compact operators on Hilbert modules}
Hilbert spaces have a notion of \textit{compact operators} between them; these are given by the norm-closure of the subspace of all operators with finite-dimensional range. In this subsection we define compact operators between two Hilbert modules over a $C^*$-category and characterize the compact-relative topology on the category of bounded operators.
\begin{defn}
If $E,F$ are right Hilbert $\mathcal{A}$-modules and $x\in\Ob\mathcal{A}$, then for module elements $e\in E(x)$ and $f\in F(x)$, the \emph{single-rank} operator $\theta_x^{f,e}$ is defined by setting for all $y\in\Ob\mathcal{A}, e'\in E(y):$ $$\theta_x^{f,e}(e')= f\cdot\langle e,e'\rangle_E\in F(y).$$ 
\end{defn}
It immediately follows that $\theta^{-,-}$ is linear in the first variable, conjugate linear in the second variable, and that
$$\textrm{if } e\in E(x), f\in F(y) \textrm{ and } a\in\mathcal{A}(x,y), \textrm{ then } \theta_y^{f,e\cdot a^*}=\theta_x^{f\cdot a,e}.$$ We will omit the indexing object $x$ where it is implied or not relevant. It follows from the definition that $\theta^{f,e}$ is bounded by $\|f\|\|e\|$ and has adjoint $\theta^{e,f}$.

\begin{defn}
A \emph{finite-rank operator} is a finite sum of single-rank operators. A \emph{compact operator} is  a norm-limit of finite-rank operators.
\end{defn}
Note that our definition of compact operators is (at least on the face of it) distinct from that in \cite[Definition 3.2]{MitchenerKK}, though it agrees with the definition in \cite[Definition 2.3.6]{Ferrier}.

\begin{prop}
The compact operators form an essential ideal in the $C^*$-category $\Hilb\mathcal{A}$.
\end{prop}
\begin{proof}
Note that if $T$ is any other operator, then we have $T\circ\theta_x^{f,e}=\theta_x^{T(f),e}$ and $\theta_x^{f,e}\circ T=\theta_x^{f,T^*e}$ so it is easy to see that the norm-closure of finite sums of such elements forms a $C^*$-ideal.


We now show this ideal is essential: suppose $T$ is a bounded operator $E\rightarrow F$ such that  $T\circ\theta_x^{e,d}=0$ for all $D\in \Hilb\mathcal{A},x\in\Ob(\mathcal{A}),d\in D(x),e\in E(x)$. Let $D=E$, then we have $T\circ\theta^{e,e'}_x(e'')=T(e)\cdot\langle e',e''\rangle=T(e\cdot\langle e',e''\rangle)=0$ for all $e,e'\in E(x), e''\in E(w),w\in\Ob(\mathcal{A})$. Setting now $w=x$, we see by \Cref{density} that $T=0$ on all of $E(x)$.
\end{proof}

\begin{defn}We denote the category of Hilbert modules and compact operators by $\KHilb\mathcal{A}$. When $\mathcal{A}$ is implicit we write $\mathcal{K}(E,F)$ rather than $\KHilb\mathcal{A}(E,F)$ for the compact operators from $E$ to $F$, and we further shorten $\mathcal{K}(E,E)$ to $\mathcal{K}(E)$.\end{defn}

\begin{lemma}\label{strongapprox}
If $U_\lambda$ is an approximate unit for the $C^*$-algebra $\mathcal{K}(E)$, then for any $x\in\Ob\mathcal{A}$ and $e\in E(x)$, we have $U_\lambda(e)\xrightarrow[]{\lambda}e$.
\end{lemma}
\begin{proof}
Notice for any $e_1,e_2,e_3\in E(x)$, that $$\begin{array}{rl}\|U_\lambda(e_1\cdot\langle e_2,e_3\rangle)-e_1\cdot\langle e_2, e_3\rangle\|&=\|(U_\lambda \circ\theta^{e_1,e_2}-\theta^{e_1,e_2})(e_3)\|\\&\leq\|U_\lambda \circ\theta^{e_1,e_2}-\theta^{e_1,e_2}\|\|e_3\|\xrightarrow[]{\lambda}0\end{array}$$
In other words, $U_\lambda(e_1\cdot\langle e_2,e_3\rangle)\xrightarrow[]{\lambda}e_1\cdot\langle e_2,e_3\rangle $. So since \Cref{density} tells us the elements of this form are dense in $E(x)$ and $\|U_\lambda\|\leq1$  we see by \Cref{limsupid} that $U_\lambda(e)\xrightarrow[]{\lambda}e$ for all $e\in E(x)$.
\end{proof}

We can use compact operators to define a sort of Yoneda embedding for $C^*$-categories:
\begin{lemma}\label{imageyoneda}
The assignment $x\mapsto h_x$ defined in \Cref{repdefn} extends to an isometric $C^*$-functor $\iota_\mathcal{A}:\mathcal{A}\rightarrow\Hilb\mathcal{A}$, where for any $a\in\mathcal{A}(x,y)$ we let $$\iota(a):h_x\rightarrow h_y$$ be the operator with components $\iota(a)_z:\mathcal{A}(z,x)\rightarrow\mathcal{A}(z,y)$ given by postcomposition with $a$. Furthermore, this functor has image $\iota(\mathcal{A}(x,y))=\mathcal{K}(h_x,h_y)$.
\end{lemma}

\begin{proof}
This functor is clearly linear, faithful and intertwines the involutions on the two categories, so it is an isometric $C^*$-functor. To show $\iota(\mathcal{A}(x,y))\supseteq\mathcal{K}(h_x,h_y)$, note that for $a\in h_x(z),b\in h_y(z)$, the rank one operator $\theta^{b,a}:h_x\rightarrow h_y$ is the one given by postcomposition with $ba^*$. To show that $\iota(\mathcal{A}(x,y))\subseteq\mathcal{K}(h_x,h_y)$, factorize any morphism $a\in\mathcal{A}(x,y)$ as in \Cref{factorlemma}: then we see that there exists $d\in\mathcal{A}(x,y)$ such that $\iota(a)=\theta^{(a^*a)^{1/4},d}$, so in fact each operator in the image of $\iota$ is single-rank.
\end{proof}

\begin{cor}\label{unitalyoneda}
If $\A$ is unital, the functor $\iota_{\A}$ is also full, and furthermore $\iota_{\A}$ extends to an equivalence $\A_{\oplus}^\natural\cong\fgProj\A$, where $\fgProj\A$ is the category of finitely generated projective right Hilbert $\A$-modules and bounded operators.
\end{cor}
\begin{proof}
    If $\A(x,x)$ is unital a standard Yoneda argument shows that any transformation $T:h_x\rightarrow h_y$ is given by postcomposing with the morphism $T_x(\id_x)$, so $T$ is in fact in the image of $\iota_{\A}$. 

    Note for the second equivalence firstly that $\iota_{\A}$ is an equivalence between $\A$ and the category of representable right Hilbert $\A$-modules. Secondly, as $\Hilb\A$ is idempotent complete, $\fgProj\A$ is by definition the closure of representable right Hilbert $\A$-modules under finite direct sums and direct summands. The stated equivalence follows.
\end{proof}

We can in fact obtain a Yoneda lemma for $C^*$-categories. Before we do this we recall a generic lemma about maps on Banach spaces.

\begin{lemma}\label{densextend}
If $M$ and $N$ are Banach spaces, $M'$ is a dense subspace of $M$ and $\phi':M'\rightarrow N$ is a  bounded linear map with $\|\phi'\|= L$, then $\phi'$ extends uniquely to a map $\phi:M\rightarrow N$ of Banach spaces with $\|\phi\|=L$. If $\phi'$ is an isometry then so is $\phi$, and if $\phi'$ in addition has dense image, then $\phi$ is an isometric isomorphism of Banach spaces.
\end{lemma}
\begin{proof}
The fact that $\phi$ extends to a map on $M$ with the same norm is an instance of continuous linear extension, see e.g. \cite[Theorem I.7]{ReedSimon}. To show that $\phi$ is an isometry when $\phi'$ is, simply write any $m\in M$ as a limit of a sequence of elements $m_n\in M'$, and then note $$ \|\phi(m)\|=\|\phi(\lim_nm_n)\|=\lim_n\|\phi(m_n)\|=\lim_n\|\phi'(m_n)\|=\lim_n\|m_n\|=\|m\|.$$
Finally, we show that if $\phi'$ has dense image, then $\phi$ must be surjective. Take for an arbitrary $p\in N$ a sequence $(m_n)$ in $M'$ such that $\phi'(m_n)\rightarrow p$. Then by the isometry property it follows that $m_n$ is Cauchy, and letting $m$ be the limit of $(m_n)$ we easily see that $\phi(m)=p$. Hence $\phi$ is an isometric isomorphism of Banach spaces.
\end{proof}



This puts us in a position to prove a Yoneda lemma.
\begin{prop}[Yoneda lemma for Hilbert modules]\label{Yoneda}
For each $x\in\Ob\mathcal{A}$ and right Hilbert $\mathcal{A}$-module $F$, there is an isometric isomorphism $$\eta:\mathcal{K}(h_x,F)\rightarrow F(x) .$$ 
\end{prop}
\begin{proof}
We construct the map $\eta:\mathcal{K}(h_x,F)\rightarrow F(x)$ by defining it first on single-rank operators. For $f\in F(y)$ and $a\in\mathcal{A}(y,x)$, note that for each $b\in\mathcal{A}(z,x)$ we have $$\theta^{f,a}(b)=f\cdot\langle a,b\rangle=f\cdot a^*b=(f\cdot a^*)\cdot b.$$
Hence we can define $\eta(\theta^{f,a}):=f\cdot a^*\in F(x)$ and extend linearly to finite-rank operators. To show this partial map is an isometry, note that picking $x=y$ and letting $b$ vary over an approximate unit for $\mathcal{A}(x,x)$, we see by \Cref{unitcor} that $$\|\sum_{i=1}^n\theta^{f_i,a_i}(u_\lambda)\|=\|(\sum_{i=1}^nf_i\cdot a_i^*)\cdot u_\lambda\|\xrightarrow[]{\lambda}\|\sum_{i=1}^nf_i\cdot a_i^*\|$$ but conversely it follows from \Cref{continuity} that $$\|\sum_{i=1}^n\theta^{f_i,a_i}(b)\|=\|(\sum_{i=1}^nf_i\cdot a_i^*)\cdot b\|\leq\|(\sum_{i=1}^nf_i\cdot a_i^*)\|\| b\| $$ so we see $\|\eta(\sum_{i=1}^n\theta^{f_i,a_i})\|=\|\sum_{i=1}^n\theta^{f_i,a_i}\|$. But then by \Cref{densextend} there is a unique isometric extension of $\eta$ to all of $\mathcal{K}(h_x,F)$.

To show $\eta$ is surjective, recall that by \Cref{CohHewCstar} for any $f\in F(x)$ there exist $a\in\A(x,x)$ and $f'\in F(x)$ such that $f=f'\cdot a$: hence $f=\eta(\theta^{f',a})$.

\end{proof}

The map $\eta$ is natural in $F$ and $x$, but rather than proving this directly we find the inverse to the map and show that \textit{it} is natural:

\begin{lemma}\label{contrayoneda}
For each $x\in\Ob\mathcal{A}$ and right Hilbert $\mathcal{A}$-module $F$, the inverse $\epsilon:F(x)\rightarrow\mathcal{K}(h_x,F)$ to the map $\eta$ defined above is given by setting $\epsilon:f\mapsto \epsilon_f$ where $\epsilon_f(a)=f\cdot a$ for all $a\in h_x(y)$. The adjoint of $\epsilon_f$ is given for each $y\in\Ob\mathcal{A}$ and $f'\in F(y)$ by $\epsilon_f^*(f')=\langle f,f'\rangle$. 

This map is natural in both $F$ and $x$: to be precise, for any $\phi\in\mathcal{L}(F,F')$, we have $\phi\circ\epsilon_f=\epsilon_{\phi(f)}$, and for any $b\in\mathcal{A}(x',x)$, we have $\epsilon_m\circ\iota(b)=\epsilon_{m\cdot b}$
\end{lemma}
\begin{proof}
Each map $\epsilon_f$ is evidently an $\mathcal{A}$-module map and is bounded by  \Cref{continuity}; furthermore, it is simple to verify that its adjoint is as stated.

Note next that for any $b\in h_x(z),\theta^{f,a}\in\mathcal{K}(h_x,F)$ we have $$(\epsilon\circ\eta)(\theta^{f,a})(b)=(f\cdot a^*)\cdot b=f\cdot(a^*b)=f\cdot\langle a,b\rangle_{h_x}=\theta^{f,a}(b).$$
Hence by the density of the finite-rank operators we see $\epsilon$ is the unique inverse for the map $\eta$.

We finish by showing the naturality: note for any $a\in h_x(y)$ that $$\epsilon_{\phi(f)}(a)=\phi(f)\cdot a^*=\phi(f\cdot a^*)=\phi(\epsilon_f(a)). $$
For the second part, simply note for any $b\in\mathcal{A}(x',x),a\in h_{x'}(y)$ we have $$(\epsilon_f\circ\iota(b))(a)=(\epsilon_f)(ba)=f\cdot(ba)=(f\cdot b)(a)=\epsilon_{f\cdot b}(a).$$\end{proof}

This correspondence helps us understand single-rank operators better:

\begin{lemma}\label{singlerank}
If $E$ and $F$ are right Hilbert $\mathcal{A}$-modules and $x\in\Ob\mathcal{A}$, the single-rank operators $\theta^{f,e}_x$ are exactly those operators that factor as the composite $$E\rightarrow h_x\rightarrow F$$ of two compact operators.
\end{lemma}
\begin{proof}
Note simply that for any $e'\in E(y)$, we have $$\theta^{f,e}_x(e')=f\cdot\langle e,e'\rangle=\epsilon_f\circ\epsilon^*_e(e').$$ Conversely, by \Cref{Yoneda} all compact sequences $E\rightarrow h_x\rightarrow F$ are described by an operator $\epsilon_f\epsilon_e^*=\theta_x^{f,e}$ which is single-rank by definition.
\end{proof}

It follows that the ideal $\KHilb\mathcal{A}$ is, in a precise sense, the one \textit{generated by} the image of $\iota_\mathcal{A}$.

\begin{prop}\label{yonedagenerate}
The ideal $\KHilb\mathcal{A}$ is the $C^*$-ideal generated by the image of $\iota$, that is: $\mathcal{K}(E,F)$ is the norm-closure of the linear span of all morphisms which factor as $E\rightarrow h_x\xrightarrow[]{\iota(a)}h_y\rightarrow F$ for some $x,y\in\Ob\mathcal{A},a\in\mathcal{A}(x,y)$.
\end{prop}
\begin{proof}
Note that every morphism that factors as described must be compact, as compact morphisms are an ideal and by \Cref{imageyoneda} all morphisms in the image of $\iota$ are compact; hence the ideal generated by the image of $\iota$ is contained in the compacts.

For the opposite inclusion, consider $\theta^{f,e}_x:E\rightarrow F$ for $e\in E(x),f\in F(x)$, and factorize $f=f'\cdot a$ for $a\in\mathcal{A}(x,x)$ using \Cref{CohHewCstar}. Then $$\theta^{f,e}=\epsilon_f\circ\epsilon^*_e=\epsilon_{f'}\circ\iota(a)\circ\epsilon^*_e.$$ Hence any single-rank operator can be factorized through a morphism in the image of $\iota$, and hence we see any compact operator is a norm-limit of sums of such operators.
\end{proof}

We close by describing a topology of `pointwise convergence' on $\Hilb\mathcal{A}$ that will come to play a significant role.
\begin{defn}\label{strongtopdefn}
The \emph{strong topology} on $\Hilb\mathcal{A}$ is the topology such that a net of operators $T_\lambda:E\rightarrow F$ converges strongly to $T$ when for every object $x\in\Ob\mathcal{A}$ and elements $e\in E(x), f\in F(x)$, we have $\|T_\lambda(e)-T(e)\|\xrightarrow[]{\lambda}0$ and $\|T_\lambda^*(f)-T^*(f)\|\xrightarrow[]{\lambda}0$. 
\end{defn}
As with our ultrastrong topology, this nomenclature coincides with that in the literature, except for the omission of an asterisk: if $\mathcal{A}=\mathbb{C}$ and $E=F$, our strong topology on $\mathcal{L}(E)$ is commonly called the strong* topology.

\begin{lemma}\label{2topologies}
The strong topology coincides with the $\KHilb\mathcal{A}$-relative topology on norm-bounded subsets.
\end{lemma}

We omit the proof as it is completely analogous to the case for $C^*$-algebras: see \cite[Proposition 8.1]{LanceHilbert}. To see that these topologies do not necessarily coincide on unbounded subsets, see \cite[p.76]{LanceHilbert} for an example of an unbounded net of operators on a Hilbert space (i.e. a Hilbert module over $\mathcal{A}=\mathbb{C}$), which goes to zero in the strong topology but not in the topology relative to the compact operators.

\subsection{Bimodules and the equivalence $\mathcal{M}(\KHilb\mathcal{A})\cong\Hilb\mathcal{A}$}

 We now seek to categorify the following classical result in two ways: if $A$ is a $C^*$-algebra and $E$ is a right Hilbert $A$-module, then $\mathcal{M}(\mathcal{K}(E))=\mathcal{L}(E)$ (see e.g. \cite[Theorem 2.4]{LanceHilbert}). The first categorification is that we replace $A$ by a $C^*$-category $\mathcal{A}$, the second is that we replace the algebras $\mathcal{K}(E)$ and $\mathcal{L}(E)$ by the $C^*$-categories $\KHilb\mathcal{A}$ and $\Hilb\mathcal{A}$.
\begin{defn}\label{nondegenbimoddefn}
A $C^*$-functor $F:\mathcal{A}\rightarrow\Hilb\mathcal{B}$ is also called a \emph{right Hilbert} $\mathcal{A}-\mathcal{B}$ \emph{bimodule}. $F$ is said to be \emph{non-degenerate} if the linear span of $\bigcup_{y\in\Ob\mathcal{A}}F\mathcal{A}(y,x)(F(y))$ is dense in $ F(x)$  for all $x\in\Ob\mathcal{A}$; that is, dense in $F(x)(z)$ for each $z\in\Ob\mathcal{B}$.
\end{defn}
\begin{rmk}
The astute reader may be concerned that we have already defined what it means for a functor to be non-degenerate in \Cref{nondegendefn}, but by the end of this subsection we will see that the definitions coincide in a precise sense.
\end{rmk}

\begin{lemma}
The inclusion $\KHilb\mathcal{A}\xhookrightarrow{}\Hilb\mathcal{A}$ is non-degenerate.
\end{lemma}
\begin{proof}
Recall from \Cref{density} that for any right Hilbert $\mathcal{A}$-module $E$ and every $x\in\Ob\mathcal{A}$, we have that $E(x)\cdot\langle E(x),E(x)\rangle$ is dense in $E(x)$. Note that this is a subspace of $\KHilb\mathcal{A}(E,E)(E(x))$, which \textit{a fortiori} must then also be dense, satisfying \Cref{nondegenbimoddefn}. 
\end{proof}

We provide an alternate characterization of non-degeneracy, analogous to criteria 3 and 4 in \Cref{nondegendefn}.
\begin{lemma}\label{strongnondegenchar}
$F$ is non-degenerate if and only if for any $x\in\Ob\A$, approximate unit $(u_\lambda)$ for $\mathcal{A}(x,x)$,  $z\in\Ob\mathcal{B}$, and $e\in F(x)(z)$, we have  $F(u_\lambda)(e)\xrightarrow[]{\lambda} e$.
\end{lemma}
\begin{proof}
For the rightward implication, note that for any $a\in\mathcal{A}(y,x)$, we have that $u_\lambda a\xrightarrow[]{\lambda} a$. So by the continuity of $F$, for any element  $F(a)(d)\in F\mathcal{A}(y,x)(F(y)(z))$ we have that $$F(u_\lambda)(F(a)(d))=F(u_\lambda a)(d)\xrightarrow[]{\lambda}F(a)(d)$$ and clearly the same holds for finite sums of elements of this form.

Hence by approximating $e$ by such finite sums and applying \Cref{nondegenstrongthm}, we see that $F(u_\lambda)(e)\rightarrow e$.

For the converse, note that if $F(u_\lambda)e\xrightarrow[]{\lambda}e$ for $e\in\mathcal{}$then the set $$\{F(u_\lambda)(e):\lambda\in\Lambda\}\subseteq F(\mathcal{A}(x,x))(F(x)(z))$$ has $e$ as a limit point, so the density requirement is satisfied.
\end{proof}

\begin{lemma}\label{extensionlemma}
If $F:\mathcal{A}\rightarrow\Hilb\mathcal{B}$ is a non-degenerate bimodule, and $\mathcal{C}$ is any $C^*$-category containing $\mathcal{A}$ as a $C^*$-ideal, then $F$ extends uniquely to a $C^*$-functor $\bar{F}:\mathcal{C}\rightarrow\Hilb\mathcal{B}$, which is faithful whenever both $F$ is faithful \emph{and} $\mathcal{A}$ is essential in $\mathcal{C}$.
\end{lemma}
\begin{proof}
Note that $\mathcal{A}$ and $\mathcal{C}$ have the same objects by the remark below \Cref{subcatideal}. Take any morphism $c\in\mathcal{C}(x,x')$. By nondegeneracy, at any object $z\in\Ob\mathcal{B}$, the span of $\cup_{y\in\Ob\mathcal{A}}F\mathcal{A}(y,x)F(y)(z)$ is dense in $F(x)(z)$. On this dense subset we define the transformation $$\bar{F}(c):\begin{array}{l}\sum_{y\in\Ob\mathcal{A}}F\mathcal{A}(y,x)F(y)(z)\rightarrow F(x')(z)\\\Sigma_{i=1}^nF(a_i)e_i\mapsto\Sigma_{i=1}^n F(ca_i)e_i.\end{array}$$ Note the expression is well-defined since $ca_i\in\A$. We firstly need to show that this is a well-defined function on the given sum of subspaces, i.e. that if $\Sigma_{i=1}^nF(a_i)e_i=\Sigma_{i=1}^mF(b_i)f_i$, then  $\Sigma_{i=1}^n F(ca_i)e_i=\Sigma_{i=1}^m F(cb_i)f_i$. So let  $u_\lambda$ be an approximate unit for $\mathcal{A}(x,x)$, then by \Cref{approxunit} we have $u_\lambda a_i\xrightarrow[]{\lambda} a_i$, so $cu_\lambda a_i\xrightarrow[]{\lambda} ca_i$ and by the continuity of $F$ we have 

$$\begin{array}{rrl}
    
\Sigma_{i=1}^n F(ca_i)e_i&=\lim_\lambda\sum_{i=1}^n F(cu_\lambda a_i)e_i&=\lim_\lambda F(cu_\lambda)\sum_{i=1}^n  F(a_i)e_i\\&&=\lim_\lambda F(cu_\lambda)\sum_{i=1}^m  F(b_i)f_i\\&&=\lim_\lambda\sum_{i=1}^m F(cu_\lambda b_i)f_i\\&&=\Sigma_{i=1}^m F(cb_i)f_i \end{array}$$

Hence $\overline{F}(c)$ is well-defined, and applying norms to the first three terms, we deduce it's bounded by $\|c\|$: hence by \Cref{densextend} we can extend it to a bounded map $\overline{F}(c):F(x)(z)\rightarrow F(x')(z)$, which is easily seen to be have adjoint $\bar{F}(c^*)$. It's elementary to check that $\bar{F}$ is complex linear and preserves composition. It follows that $\bar{F}$ defines a $C^*$-functor $\mathcal{C}\rightarrow\Hilb\mathcal{A}$. Finally if $F$ is faithful and $\mathcal{A}$ is essential in $\mathcal{C}$, we have $\ker(\bar{F})\cap \mathcal{A}=\ker F=0$, so $\ker(\bar{F})=0$ and $\bar{F}$ is faithful.
\end{proof}

\begin{lemma}\label{idealizerlemma}

Suppose $F:\mathcal{A}\rightarrow\Hilb\mathcal{B}$ is a faithful non-degenerate bimodule. There is a $C^*$-subcategory $\mathcal{D}$ of $\Hilb\mathcal{B}$, termed the \emph{idealizer under} $F$ \emph{of} $\mathcal{A}$,  with objects being those in the image of $F$ and hom spaces 
$$\mathcal{D}(F(x),F(y)):=\Bigg\{ d\in\mathcal{L}(F(x),F(y))):\begin{array}{l}
    d\circ F\mathcal{A}(z,x)\subseteq F\mathcal{A}(z,y)\emph{ and} \\ F\mathcal{A}(y,z)\circ d\subseteq F\mathcal{A}(x,z) \\ \emph{ for all } z\in\Ob\mathcal{A}\
 
\end{array}\Bigg\}.$$ Furthermore, $F$ extends uniquely to a $C^*$-equivalence $\bar{F}:\mathcal{MA}\xrightarrow{\cong}\mathcal{D}.$
\end{lemma}
\begin{proof}
Notice $\mathcal{D}$ is a $C^*$-category which contains (the isometric image of) $\mathcal{A}$ as an ideal. Now if $\mathcal{C}$ is any other $C^*$-category containing $\mathcal{A}$ as an ideal, by \Cref{extensionlemma} we get an embedding $\mathcal{C}\xhookrightarrow{}\Hilb\mathcal{B}$ whose image lands in $\mathcal{D}$ by construction. So $\mathcal{D}$ has the defining universal property of $\mathcal{MA}$ from \Cref{multiplierprop}.
\end{proof}
We deduce our promised result.
\begin{prop}\label{MKA}
For any $C^*$-category $\mathcal{A}$, we have $\mathcal{M}(\KHilb\mathcal{A})\cong\Hilb\mathcal{A}$.
\end{prop}
\begin{proof}
We know the inclusion $\KHilb\mathcal{A}\xhookrightarrow{}\Hilb\mathcal{A}$ is nondegenerate, and since $\KHilb\mathcal{A}$ is an ideal in $\Hilb\mathcal{A}$, we see the idealizer of $\KHilb\mathcal{A}$ is all of $\Hilb\mathcal{A}$, so by \Cref{idealizerlemma}, we get the desired isomorphism. 
\end{proof}

The above proposition was also proven in \cite[Corollary 3.4.5]{Ferrier}. We hence get the promised disambiguation of non-degeneration.
\begin{cor}
A functor $F:\mathcal{A}\rightarrow\Hilb\mathcal{B}$ is non-degenerate in the above sense if and only if,  considered as a functor $\mathcal{A}\rightarrow\mathcal{M}(\KHilb\mathcal{B})$, it's non-degenerate in the sense of \Cref{nondegendefn}.
\end{cor}
\begin{proof}
The content of \Cref{strongnondegenchar} is that $F:\mathcal{A}\rightarrow\mathsf{Hilb}-\mathcal{B}$ is non-degenerate if and only for every approximate unit $u_\lambda$ of every endomorphism algebra $\mathcal{A}(x,x)$, the net $F(u_\lambda)$ approaches $\id\in\mathcal{L}(Fx,Fx)$ in the strong topology. Note that by \Cref{2topologies} this is equivalent to $F(u_\lambda)$ approaching $\id$ in the $\KHilb\mathcal{B}$-relative topology, and in turn by \Cref{nondegenstrongthm} this is equivalent to asking that $F$, considered as a functor $F:\mathcal{A}\rightarrow\mathcal{M}(\KHilb\mathcal{B})$ is non-degenerate in the sense of \Cref{nondegendefn}.
\end{proof}


We close with two propositions on multiplier direct sums of Hilbert modules.
\begin{prop}
The multiplier direct sum in $\KHilb\mathcal{A}$ of an arbitrary collection $(E_i)_{i\in I}$ of right Hilbert $\mathcal{A}$-modules is given on each object $x\in\Ob\mathcal{A}$ by $$ (\bigoplus_{i\in I}E_i)(x)=\{(e_i)\in\prod_{i\in I}E_i(x):\sum_{i\in I}\langle e_i,e_i\rangle \emph{ converges in }\mathcal{A}(x,x)\},$$ 
 with inner products defined by $\langle(e_i),(f_i)\rangle=\sum_{i\in I}\langle e_i,f_i\rangle $, and bounded adjointable structure maps $\iota_i\in\mathcal{L}(E_i,\bigoplus_{i\in I}E_i)$ sending elements to sequences with zeroes in all but one position.
\end{prop}
\begin{proof}[Proof Sketch] To prove the `smallness' of this direct sum, simply proceed as in the finite case: that is, let $S_i$ be a generating set for each $E_i$, then we easily show that $\oplus_{i=1}^nE_i$ is generated by the set $\sqcup_{i=1}^{n}\iota_i(S_i)$.

Moving on to the multiplier direct sum property, it is straightforward to verify that the net  $\sum_{i\in J}\iota_i\iota_i^*$ for finite subsets $J\subseteq I$ converges strongly to the identity, hence it also converges ultrastrongly by \Cref{2topologies}. This is the defining property of a multiplier direct sum.\end{proof}

\begin{prop}\label{compactdirect}
    If $(E_i)_{i\in I}$ and $(F_j)_{j\in J}$ are collections of right Hilbert $\mathcal{A}$-modules with structure maps $\iota_i: E_i\rightarrow \bigoplus_{i\in I}E_i$ and $\iota_j: F_j\rightarrow \bigoplus_{j\in J}F_j$, then the isometries  $$\lambda_{ij}:\mathcal{K}(E_i,F_j)\rightarrow \mathcal{K}(\bigoplus_{i\in I}E_i,\bigoplus_{j\in J}F_j):T\mapsto \iota_j  T \iota_i^*$$ and their left inverses $$\mu_{ij}: \mathcal{K}(\bigoplus_{i\in I}E_i,\bigoplus_{j\in J}F_j)\rightarrow \mathcal{K}(E_i,F_j):S\mapsto \iota^*_j  T \iota_i$$ exhibit the Banach space $\mathcal{K}(\bigoplus_{i\in I}E_i,\bigoplus_{j\in J}F_j)$ as a direct sum of its subspaces $\mathcal{K}(E_i,F_j)$. In other words, for every operator $T\in\mathcal{K}(\bigoplus_{i\in I}E_i,\bigoplus_{j\in J}F_j)$, the net $\sum_{S\subseteq I\times J \text{finite}}\lambda_{ij}(\mu_{ij}(T))$ converges \emph{in the norm} to $T$.
\end{prop}

\begin{proof}[Proof Sketch] The single-rank operators $\theta^{e,f}$ such that $e$ and $f$ have finitely many non-zero coordinates are dense in all single-rank operators from $\bigoplus_{i\in I}E_i$ to $\bigoplus_{j\in J}F_j$. It is easily shown that the stated convergence result holds for these. Hence an application of \Cref{limsupid} proves the result for arbitrary single-rank operators, and hence also for finite-rank operators. Another application of the same corollary proves the result for arbitrary compacts.\end{proof} 



\subsection[Extending modules to the additive hull and approximate projectivity]{Extending Hilbert modules to $\A_{\oplus}$ and approximate projectivity}
In this subsection we show that all Hilbert modules over a $C^*$-category $\mathcal{A}$ extend to Hilbert modules over the additive hull $\mathcal{A}_{\oplus}$. We use for this the following lemma of Mitchener's: 


\begin{lemma}[{\cite[Lemma 3.13]{MitchenerKK}}]\label{positivitycat}
If $E$ is a Hilbert module over a C*-category $\mathcal{A}$, then an operator $T:E\rightarrow E$ is a positive element of the C*-algebra $\mathcal{L}(E,E)$ if and only if for every $x\in\Ob\mathcal{A}$ and $e\in E(x)$, the inner product $\langle e,Te\rangle\in\mathcal{A}(x,x)$ is positive.
\end{lemma}
To begin building the extension of $E$ to $\A_{\oplus}$, we isolate the positivity axiom, which is of independent interest elsewhere:
\begin{lemma}[{c.f. \cite[Lemma 4.2]{LanceHilbert}}]\label{matrixpos}
If $E$ is a right Hilbert module over $\mathcal{A}$ and we have a sequence of elements $e_1\in E(x_1),\dots,e_n\in E(x_n)$, then the $n\times n$ matrix $\mu_e$ with $(i,j)$-th entry $\langle e_i,e_j\rangle$ is a positive element of the $C^*$-algebra $M_{x_1,..,x_n}(\mathcal{A})$
\end{lemma}
\begin{proof}

This is one direction of \cite[Proposition 2.2.19]{henry2015complete}.
\end{proof}

To extend an arbitrary Hilbert module from $\mathcal{A}$ to $\mathcal{A}_{\oplus}$, note firstly that by \Cref{addextendfunc}, any functor $E:\mathcal{A}^{\text{op}}\rightarrow\mathsf{Vect}_\mathbb{C}$ lifts to a functor  $E_{\oplus}:(\mathcal{A}_{\oplus})^{\text{op}}\rightarrow\mathsf{Vect}_\mathbb{C}$.  We need to show that if $E$ is a right Hilbert module, then $E_{\oplus}$ can be given an inner product making it into a right Hilbert $\mathcal{A}_{\oplus}$-module.

Our promised extension lemma now follows:
\begin{prop}\label{addextend}
Let $E_{\oplus}:(\mathcal{A}_{\oplus})^{\emph{op}}\rightarrow\mathsf{Vect}_\mathbb{C}$ be the natural extension mentioned above. Then $E_{\oplus}$ has the structure of a Hilbert module with inner products $$\langle-,-\rangle: E_{\oplus}(\{y_1,..,y_k\})\times E_{\oplus}(\{x_1,..,x_n\})\rightarrow\mathcal{A}_{\oplus}(\{x_1,..,x_n\},\{y_1,..,y_k\})$$
defined by $$\langle(f_1,..,f_k),(e_1,..,e_n)\rangle_{E_{\oplus}}:=[\langle f_j,e_i\rangle_E]_{i,j}.$$
\end{prop}


We omit the proof as all conditions in \Cref{hilbmoddefn} are easily proven, except the positivity which is taken care of by \Cref{matrixpos}. For the smallness condition, the set generating $E_{\oplus}$ can be taken to be the exact same one generating $E$.

Conversely, it is clear that a Hilbert module over $\mathcal{A}_{\oplus}$ is determined by its spaces and inner products on objects of $\mathcal{A}$, and clearly a transformation $E_{\oplus}\rightarrow F_{\oplus}$ is determined by its values on the objects of $\mathcal{A}$, in other words by its restriction $E\rightarrow F$. All this gives us the following proposition, which we will later reprove by other means:

\begin{prop}\label{extensionequivalence}
The extension operation in \Cref{addextend} and restriction operation defined above constitute a unitary equivalence  $\Hilb\mathcal{A}\cong\Hilb\mathcal{A}_{\oplus}$.
\end{prop}

We also note that this equivalence preserves compacts:

\begin{lemma}\label{compactadd}
Suppose $E,F$ are right Hilbert modules over $\mathcal{A}_{\oplus}$. For all lists $\mathbf{x}=\{x_1,\dots,x_n\}\in\mathsf{Ob}(\mathcal{A}_{\oplus})$ and tuples of elements $(e_i)\in E(\mathbf{x}),(f_i)\in F(\mathbf{x})$, we have $\theta_{\mathbf{x}}^{(f_i),(e_i)}=\sum_{i=1}^{n}\theta_{x_i}^{f_i,e_i}$.
\end{lemma}
We omit the proof as it follows directly from the definitions.

We end this section with a powerful result which says that any Hilbert module $E$ has an approximate identity consisting of self-adjoint operators that factor through finitely generated free modules. The results here are a direct adaptation and generalization of those in \cite[Section 3]{BlecherModules}.
\begin{lemma}\label{compactapprox}
If $E$ is a right Hilbert $\mathcal{A}$-module, then the $C^*$-algebra $\mathcal{K}(E)$ has an approximate unit $\{u_\lambda:\lambda\in\Lambda\}$ consisting of operators $u_\lambda=\sum_{e\in I(\lambda)}\theta_{x_e}^{e,e}$ where each $I(\lambda)$ is a finite list of module elements $e\in E(x_e)$, where $x_e\in\mathsf{Ob}\mathcal{A}$. 
\end{lemma}
\begin{proof}
Note that the right ideal of finite-rank operators is dense in $\mathcal{K}(E)$, so by \cite[Theorem 2.1]{BrownIdeal}, there's an increasing approximate unit consisting of elements $u_{\lambda}=\sum_{r\in R_{\lambda}}r^*r$, where each $R_\lambda$ is a finite list of finite-rank operators. Take one such operator $r=\sum_{i=1}^n\theta_{x_i}^{f_i,e_i}$, then we have  $$r^*r=\sum_{1\leq i,j\leq n}\theta_{x_i}^{e_i,f_i}\theta_{x_j}^{f_j,e_j}=\sum_{1\leq i,j\leq n}\theta_{x_i}^{e_i,\theta_{x_j}^{e_j, f_j}(f_i)}=\sum_{1\leq i,j\leq n}\theta_{x_i}^{e_i,e_j\cdot\langle f_j,f_i\rangle}.$$

Note by \Cref{matrixpos} the matrix $\mu_f\in M_{x_1,\dots,x_n}(\mathcal{A})$ with $(i,j)$-th entry $\langle f_i,f_j\rangle$ is positive, so it has a square root $N$ with entries $n_{ij}$ such that $N^*N=\mu_f$. Hence extending $E$ over $\mathcal{A}_{\oplus}$ and using \Cref{compactadd} we get 

$$r^*r=\theta_{\oplus_i x_i}^{(e_i),N^*N\cdot(e_i)}=\theta_{\oplus_i x_i}^{N\cdot(e_i),N\cdot(e_i)}=\sum^n_{i=1}\theta_{x_i}^{\sum_j e_j\cdot n_{ij},\sum_j e_j\cdot n_{ij}},$$

giving an approximate unit of the described form.
\end{proof}
We can then deduce the following generalization of one direction in \cite[Theorem 3.1]{BlecherModules} .
\begin{thm}\label{approxff}
If $E$ is a right Hilbert $\mathcal{A}$-module, then there is a collection of  operators $\phi_\lambda:\bigoplus_{x\in S_{\lambda}}h_x\rightarrow E$, where each $S_\lambda,\lambda\in\Lambda$ is a finite list of $\mathcal{A}$-objects, such that the net of operators $\phi_\lambda\phi_\lambda^*$ converges strongly to $\id_E$.
\end{thm}
\begin{proof}
We proceed as in \cite{BlecherModules}. Note that by the previous lemma, $\mathcal{K}(E)$ has an approximate unit $\{u_\lambda:\lambda\in\Lambda\}$ consisting of operators $u_\lambda=\sum_{x\in S_\lambda}\theta_{x}^{e_x,e_x}$. It then follows from \Cref{contrayoneda} that the map $\phi_\lambda:\bigoplus_{x\in S_{\lambda}}h_x\rightarrow E$ defined on $h_{x}(y)=\mathcal{A}(y,x)$ by $a\mapsto e_x\cdot a$ is bounded and has adjoint $\phi_\lambda^*$ with components $\phi_\lambda^*:\begin{array}{l}E(y)\rightarrow h_{x}(y)\\e'\mapsto\langle e,e'\rangle\end{array}$, and one easily verifies $\phi_\lambda\phi_\lambda^*(e)= u_\lambda(e)$. Hence \Cref{strongapprox} gives us the required convergence result.
\end{proof}

\begin{rmk}Note that we do not have that $S_\lambda\subset S_{\lambda'}$ for $\lambda<\lambda'$, and even if we did, it would not follow that $\phi_{\lambda'}$ restricts to $\phi_\lambda$. Hence \Cref{approxff} does not allow us to embed $E$ as a direct summand of a single infinite-dimensional free module: to do so would give an untenable generalization of the Kasparov stabilization theorem, contradicting the results in e.g. \cite{NoFrames}.\end{rmk}


As an aside we deduce that a Hilbert module with compact identity must be finitely generated projective:

\begin{prop}\label{compactidfgp}
    If the Hilbert module $E$ has $\id_E\in\mathcal{K}(E)$, then $E$ must be  finitely generated and projective. If $\A$ is unital, the converse holds, i.e. any finitely generated projective module has compact unit.
\end{prop}
\begin{proof}
    As noted before an approximate unit in a unital algebra converges in the norm to the identity. Therefore if $\mathcal{K}(E)$ is unital, then by \Cref{compactapprox} the module $E$ must have a finite set of elements $e_1\in E(x_1),\dots,e_n\in E(x_n)$ such that $\|\id_E-\sum_{i=1}^{n}\theta_{x_i}^{e_i,e_i}\|<1$. But then $\sum_{i=1}^{n}\theta_{x_i}^{e_i,e_i}$ is invertible and decomposing as in \Cref{approxff} we see that in fact there is a compact operator $\phi:\bigoplus_{i=1}^nh_{x_i}\rightarrow E$ such that $\phi\phi^*:E\rightarrow E$ is an isomorphism. Suppose it has inverse $\psi$, then one easily verifies that $\phi^*\psi\phi:\bigoplus_{i=1}^nh_{x_i}\rightarrow \bigoplus_{i=1}^nh_{x_i}$ is a projection with image isomorphic to $E$.

    To prove the partial converse: suppose $\A$ is unital, then for any $x\in\Ob\A$ we have $\mathcal{L}(h_x,h_x)\cong\MA(x,x)\cong\A(x,x)\cong\mathcal{K}(h_x,h_x)$ so we see that for any list $x_1\dots x_n\in\Ob\A$ the module $\bigoplus_{i=1}^nh_{x_i}$ has a compact unit given by a diagonal matrix of identities $\id_{x_i}$. So if $E$ is a direct summand of $\bigoplus_{i=1}^nh_{x_i}$, we immediately see we can factor the identity of $E$ through that of $\bigoplus_{i=1}^nh_{x_i}$, meaning $\id_E$ is compact.
\end{proof}


\section{Tensor products of bimodules}
In this section we prove two of the main results in this article: an Eilenberg-Watts theorem that characterizes functors between Hilbert module categories as being given by tensor products, and a Morita theorem that describes which of these bimodules give equivalences upon tensoring.
\subsection{Tensoring bimodules}
When we consider a ring $A$ as a one-object additive category, an $A-B$ bimodule can also be described as an additive functor $A\rightarrow\mathsf{Mod}-B$. We recall from the previous section the following definition, which generalizes this:

\begin{defn}
If $\mathcal{A}$ and $\mathcal{B}$ are $C^*$-categories, a \emph{right Hilbert $\mathcal{A}-\mathcal{B}$ bimodule} is a $C^*$-functor $F:\mathcal{A}\rightarrow\Hilb\mathcal{B}$. A right Hilbert bimodule is called \textit{non-degenerate} when the linear span of the subspace $\bigcup_{y\in\Ob\mathcal{A}}F\mathcal{A}(y,x)(F(y))\subseteq F(x)$ is dense for all $x\in\Ob\mathcal{A}$.
\end{defn}
For any $a\in\mathcal{A}(x,y), z\in\Ob\mathcal{B}$ and $e\in F(x)(z)$, we write $a\cdot e$ for $F(a)(e)$ when the action is not ambiguous. 

Notice that there is a $C^*$-category $(\Hilb\mathcal{B})^\mathcal{A}$ of right Hilbert $\mathcal{A}-\mathcal{B}$ bimodules by \Cref{functorcat}. In this category, a morphism $T\in(\Hilb\mathcal{B})^\mathcal{A}(E,E')$ is an adjointable natural transformation of functors with a uniform bound $\|T\|$ given by $\sup_{x\in\Ob\mathcal{A}}\|T_x\|_{\mathcal{L}(E(x),E'(x))}$. There is also a \textit{strong topology} on $(\Hilb\mathcal{B})^\mathcal{A}(E,E')$ in which $T_\lambda\rightarrow T$ if and only if $(T_\lambda)_x\rightarrow T_x$ strongly for each $x\in\Ob\mathcal{A}$. 

We can easily generalize \Cref{unitarycrit} to capture unitary isomorphisms of bimodules:

\begin{prop}\label{unitarybimodcrit}
If $E$ and $F$ are right Hilbert $\mathcal{A}-\mathcal{B}$ bimodules, and $T:E\Rightarrow F$ is a natural transformation from $E$ to $F$ ($E$ and $F$ considered purely as functors $\mathcal{A}\rightarrow\Hilb\mathcal{B}$), then $T$ is a unitary isomorphism of right Hilbert $\mathcal{A}$-modules if and only if for every all objects $x\in\Ob\mathcal{A}$, the map of $\mathcal{B}$-modules $T(x):E(x)\rightarrow F(x)$ is surjective at every $y\in\Ob\mathcal{B}$ and preserves all inner products.
\end{prop}
\begin{proof}
    It is obvious that a transformation of functors on $C^*$-categories is a unitary isomorphism in the functor $C^*$-category if and only if it is a unitary isomorphism at any object. So this follows from \Cref{unitarycrit}.
\end{proof}

We continue with our first example of a non-degenerate bimodule.
\begin{lemma}\label{representation}
The Yoneda functor $\iota_\mathcal{A}:\mathcal{A}\xhookrightarrow{}\Hilb\mathcal{A}$ is a non-degenerate $\mathcal{A}-\mathcal{A}$ bimodule.
\end{lemma}
\begin{proof}
Recall from \Cref{strongnondegenchar} that a functor $F:\mathcal{A}\rightarrow\Hilb\mathcal{B}$ is non-degenerate if and only if for every approximate unit $(u_\lambda)$ for $\mathcal{A}(x,x)$ and every $e\in F(x)$ we have $u_\lambda\cdot e\xrightarrow[]{\lambda} e$. But in the case of the Yoneda functor this follows from \Cref{unitcor}.
\end{proof}
\begin{xmpl}
Every right Hilbert $\mathcal{A}$-module $E$ can be given the structure of a non-degenerate right Hilbert $\mathbb{C}-\mathcal{A}$ bimodule by the unique unital $C^*$-functor $\mathbb{C}\rightarrow\Hilb\mathcal{A}$ that sends the single object of the former category to $E$.
\end{xmpl}

We now move on to a discussion of tensor products of bimodules. In the section below, $\mathcal{A},\mathcal{B},\mathcal{C}$ and $\mathcal{D}$ are locally small $C^*$-categories.
\begin{defn}\label{tensordefn}
    If $E:\mathcal{A}\rightarrow\Hilb\mathcal{B}$ and $F:\mathcal{B}\rightarrow\Hilb\mathcal{C}$ are right Hilbert bimodules, their \emph{uncompleted tensor product} $E\otimes_\mathcal{B}F:\mathcal{A}\rightarrow\mathsf{Fun}({\mathcal{C}^\textrm{op}},\mathsf{Vect}_\mathbb{C})$ is defined at any two objects $x\in\Ob\mathcal{A},z\in\Ob\mathcal{C}$ by taking in $\mathsf{Vect}_\mathbb{C} $ the coend $$ E\otimes_\mathcal{B}F(x)(z):=\bigoplus_{y\in\Ob\mathcal{B}}E(x)(y)\otimes_\mathbb{C} F(y)(z)\bigg/\sim$$
where the relation $\sim$ is generated by identifying any pair of simple tensor elements $e\otimes f\in E(x)(y)\otimes_\mathbb{C}F(y)(z)$ and $e'\otimes f'\in E(x)(y')\otimes_\mathbb{C}F(y')(z)$ such that there exists an element $b\in\mathcal{B}(y,y')$ with  $e'\cdot b=e$ and $b\cdot f=f'$. The \emph{tensor product} $E\bar{\otimes}_\mathcal{B}F:\mathcal{A}\rightarrow\Hilb\mathcal{C}$ is then obtained using the inner product defined on any pair of simple tensors $e\otimes f\in E(x)(y)\otimes F(y)(z),e'\otimes f'\in E(x)(y')\otimes F(y')(z')$ as $$\langle e\otimes f, e'\otimes f'\rangle:=\langle f,\langle e,e'\rangle_E\cdot f'\rangle_F$$ and completing under the norm this gives.  The action of $\mathcal{A}$ on these spaces is defined on simple tensors by $a\cdot(e\otimes f):=a\cdot e\otimes f$ and the $\mathcal{C}$-action is given by $(e\otimes f)\cdot c:=e\otimes (f\cdot c)$; these actions are easily checked to be linear and norm-decreasing so they extend to all of $E\bar{\otimes}_\mathcal{B}F$.
\end{defn}
\begin{prop}
    
$E\bar{\otimes}_\mathcal{B}F:\mathcal{A}\rightarrow\Hilb\mathcal{C}$ is well-defined and satisfies the axioms of a right Hilbert $\A-\C$ bimodule.
\end{prop}
\begin{proof} This is done for small $C^*$-categories in \cite[Proposition 3.10]{MitchenerKK}, so we prove here only that the construction extends to (small) Hilbert modules over locally small $C^*$-categories.

We need to show firstly that the coend defining $E\bar{\otimes}_\mathcal{B}F(x)(z)$ for a given pair $x\in\Ob\A,z\in\Ob\C$ is in fact a small vector space, since \emph{a priori} it is defined as a quotient of a direct sum over $\Ob\B$. Let $S$ be a set of elements $s\in E(x)(y_s)$ generating $E(x)\in\Hilb\B$. Suppose that $e=\Sigma_{i=1}^n\mu_is_i+ \Sigma_{j=1}^m s_j\cdot b_j$, where for all $i$ and $j$ we have $s_i\in S\cap E(x)(y),\mu_i\in\mathbb{C},s_j\in S$, and $b_j\in\B(y,y_{s_j})$. Then for any $f\in F(y)(z)$, using the relationship $\sim$ on the second sum defining $e$, we see in fact that $$e\otimes f\in \bigoplus_{y_s\in\Ob\mathcal{B}:s\in S}E(x)(y_s)\otimes_\mathbb{C} F(y_s)(z)\bigg/\sim$$ noting that this coend is indexed by a \emph{set} and is hence a small vector space. But of course by definition of $S$ generating $E(x)$, elements $e$ of the form described are dense in $E(x)(y)$. Now approximating an arbitrary $e\in E(x)(y)$ by a sequence of finite sums $e_n$ as above, and we see by definition of the norm on $E\bar{\otimes}_\mathcal{B}F(x)(z)$ that $e_n\otimes f\xrightarrow[]{}e\otimes f$ as $n\rightarrow\infty$. Hence the small coend above is dense in the space of finite sums of simple tensors, and hence dense in its completion $E\bar{\otimes}_\mathcal{B}F(x)(z)$. So  $E\bar{\otimes}_\mathcal{B}F(x)(z)$ must  be a small vector space, since we can map it injectively to the set of Cauchy sequences on a set.

We also need to verify for each $x\in\Ob\A$ that the a priori large right Hilbert $\C$-module $E\bar{\otimes}_\mathcal{B}F(x)(-)$ is generated by a set. To do this, let $S$ be a set of elements $s\in E(x)(y_s)$ generating $E(x)$, and for each $s$ let $T_s$ be a set of elements generating $F(y_s)\in\Hilb\C$. Then like above, we can use the relation $\sim$ and the definition of the norm to show that the set $\{s\otimes t:s\in S,t\in T_s\}$ generates $E\bar{\otimes}_\mathcal{B}F(x)(-)$.
\end{proof}

A basic result on tensor products of bimodules is that they are associative up to unitary isomorphism:

\begin{lemma}\label{tensorassociative}
    If $E,F,$ and $G$ are respectively $\mathcal{A}-\mathcal{B}, \mathcal{B}-\mathcal{C}$ and $\mathcal{C}-\mathcal{D}$ right Hilbert bimodules, there is a canonical unitary isomorphism $E\bar{\otimes}_\mathcal{B}(F\bar{\otimes}_\mathcal{C}G)\cong (E\bar{\otimes}_\mathcal{B}F)\bar{\otimes}_\mathcal{C} G$.
\end{lemma}
\begin{proof}
    This isomorphism is given on simple tensors by $e\otimes(f\otimes g)\mapsto (e\otimes f)\otimes g$. It is elementary to verify that it is in fact a unitary isomorphism which is natural in all three modules.
\end{proof}

In addition, the Yoneda bimodule functions as a tensor identity for non-degenerate bimodules:

\begin{lemma}\label{yonedaisid}
    If $E$ is a non-degenerate right Hilbert $\mathcal{A}-\mathcal{B}$ bimodule there are canonical unitary isomorphisms $U_\ell:(\iota_\mathcal{A})\bar{\otimes}_\mathcal{A}E\xrightarrow[]{\cong} E$ and $U_r:E\bar{\otimes}_\mathcal{B}(\iota_\mathcal{B})\xrightarrow[]{\cong} E$ of right Hilbert $\mathcal{A}-\mathcal{B}$ bimodules.
\end{lemma}
\begin{proof}
    The first isomorphism is given on simple tensors by $a\otimes e\mapsto a\cdot e$, and the second by $e\otimes b\mapsto e\cdot b$. It is elementary to verify that these maps respect the equivalence relation defining the uncompleted tensor product, and that they preserve inner products, meaning they are isometric and we can extend them to isometries on the whole module by \Cref{densextend}. Note that if $(u_\lambda)$ and $(v_\mu)$ are approximate units for the relevant endomorphism algebra in $\mathcal{A}$ and $\mathcal{B}$ respectively, then $u_\lambda\cdot e\xrightarrow[]{\lambda} e$ by non-degeneracy and $e\cdot v_\mu\xrightarrow[]{\mu} e $  by \Cref{unitcor}. Hence by applying \Cref{CH}, we see both $U_\ell$ and $U_r$ are surjective. Hence by \Cref{unitarybimodcrit} these are unitary maps of bimodules.

    It is elementary to verify that the inverses of these maps are in fact their adjoints, and that they are natural in $E$.
\end{proof}

We deduce that the tensor product functor is an extension (through the Yoneda embedding) of a bimodule from representable modules to arbitrary modules:

\begin{cor}\label{tensorextend}
If $E:\mathcal{A}\rightarrow\Hilb\mathcal{B}$ is a non-degenerate $\mathcal{A}-\mathcal{B}$ bimodule, then there is for each $x\in\Ob\mathcal{A}$ a unitary $\mathcal{B}$-module isomorphism  $\rho_x:h_x\bar{\otimes}_\mathcal{A}E\rightarrow E(x)$, given on simple tensors $a\otimes e\in \mathcal{A}(x',x)\otimes E(x')(y)$ by $\rho_x(a\otimes e)=a\cdot e\in E(x)(y)$. This isomorphism is moreover natural in $x$ and hence gives a natural unitary isomorphism of bimodules $E\cong(-\bar{\otimes}_\mathcal{A}E)\circ\iota_\mathcal{A}$.

\end{cor}
\begin{proof}

The map $\rho_x$ is obtained by taking the component at $x$ of the map $U_r$ from \Cref{yonedaisid}. Naturality of $\rho_x$ in $x$ is obvious, and naturality of its adjoint/inverse follows from this. 
\end{proof}

Before we move onto our next proposition we need an intermediate lemma on positive operators.

\begin{lemma}\label{positiveineq}
If $S$ is a positive operator on a right Hilbert $\mathcal{D}$-module $H$, then for any $x\in\Ob\mathcal{D}$, $h\in H(x)$ we have $\|S\|\langle h,h\rangle\geq\langle h,Sh\rangle$ in the $C^*$-algebra $\mathcal{D}(x,x)$. Furthermore, if $S\geq Q\geq0$, we have $\|\langle h,Sh\rangle\|\geq \|\langle h,Qh\rangle\|$. If $S$ is any operator, $\langle Sh,Sh\rangle\leq\|S\|^2\langle h,h\rangle$.
\end{lemma}
\begin{proof}
Note  by functional calculus in $\mathcal{L}(H)$ we have $\|S\|\id_H\geq S\geq0$ and hence $\|S\|\id_H-S\geq0$, hence by \Cref{positivitycat} we get $$\|S\|\langle h,h\rangle-\langle h, Sh\rangle=\langle h, (\|S\|\id_H-S)h\rangle\geq0.$$ 

If $S\geq Q\geq 0$ then by applying \Cref{positivitycat} to $S-Q$ we see $\langle h,Sh\rangle\geq\langle h,Qh\rangle\geq0$, so $\|\langle h,Sh\rangle\|\geq\langle \|h,Qh\rangle\|$.

The final result is obtained by applying the first result to the positive operator $S^*S$.
\end{proof}

This allows us to prove the functoriality of the tensor product construction.
\begin{prop}For any right Hilbert $\mathcal{B}-\mathcal{C}$ bimodule $F$, a bounded adjointable transformation of $\mathcal{A}-\mathcal{B}$ bimodules $T:E\rightarrow E'$ induces a bounded adjointable morphism of $\mathcal{A}-\mathcal{C}$ bimodules $T\bar{\otimes}_\mathcal{B}\id:E\bar{\otimes}_\mathcal{B}F\rightarrow E'\bar{\otimes}_\mathcal{B}F$ by acting on the first coordinate.
\end{prop}
\begin{proof}
We first define a map $T\otimes_\mathcal{B}\id:E\otimes_\mathcal{B} F\rightarrow E'\otimes_\mathcal{B}F$ on the uncompleted Hilbert modules by $(T\otimes_\mathcal{B}\id)(e\otimes f):=T(e)\otimes f$. It is straightforward to show that $(T\otimes_\mathcal{B}\id)^*=(T^*)\otimes_\mathcal{B}\id$ functions as an adjoint on the uncompleted module. To extend the transformation to the completed module, we must show that $T\otimes_\mathcal{B}\id$ is bounded on sums of simple tensors: more specifically we will prove that $\|T\otimes_\mathcal{B}\id\|\leq\|T\|$ on this subspace. 

We need to show, for an arbitrary element $$g=\sum_{i=1}^n e_i\otimes f_i\in E\otimes_\mathcal{B}F(x)(z)$$ where $e_i\otimes f_i\in E(x)(y_i)\otimes F(y_i)(z)$, that $\|T\|\|g\|\geq\|(T\otimes\id)(g)\|$. But by \Cref{addextend} and \Cref{extensionequivalence} we can assume without loss of generality that $\mathcal{B}$ is closed under finite direct sums. So writing $\textbf{f}:=(f_i)\in ( F(\oplus_i y_i))(z)$ and $\textbf{e}:=E_{\oplus}(x)(\oplus_i y_i)$, we have $$\begin{array}{rl}
   \|T\|^2\|g\|^2  & =\|T\|^2\|\sum_{i,j}\langle f_i,F(\langle e_i,e_j\rangle)(f_j)\rangle\|\\&=\|\langle\textbf{f},F(\langle \|T\|\textbf{e},\|T\|\textbf{e}\rangle)(\textbf{f})\rangle\|
    \\ &\geq\|\langle\textbf{f},F(\langle T\textbf{e},T\textbf{e}\rangle)(\textbf{f})\rangle\| \\&=\|(T\otimes\id)(g)\|^2
\end{array}$$
where the inequality follows from \Cref{positiveineq}. 

By \Cref{densextend} we can extend $T\otimes_\mathcal{B}\id$ to a map on the completed module $T\bar{\otimes}_\mathcal{B}\id:E\bar{\otimes}_\mathcal{B}F\rightarrow E'\bar{\otimes}_\mathcal{B}F$, also bounded by $\|T\|$ and with adjoint $T^*\bar{\otimes}_\mathcal{B}\id$.
\end{proof}
We are now going to show that moreover, the tensoring operation above satisfies a strong continuity requirement.

\begin{lemma}\label{tensorcts}
For any right Hilbert $\mathcal{B}-\mathcal{C}$ bimodule $F$, the right tensor product with $F$, together with the assignment $T\mapsto T\bar{\otimes}_\mathcal{B}\id$ defined above, gives a unital $C^*$-functor $-\bar{\otimes}_\mathcal{B}F:(\Hilb\mathcal{B})^{\mathcal{A}}\rightarrow(\Hilb\mathcal{C})^{\mathcal{A}}$ which is strongly continuous on bounded subsets of operators.
\end{lemma}
\begin{proof}
It is easily verified from the definition of $T\bar{\otimes}_\mathcal{B}\id$ that this assignment it defines a $C^*$-functor. 

So we have only left to show that if $T_\lambda$ is a bounded net of operators in $\mathcal{L}(E,E')$ strongly converging to $0$, then the net $(T_\lambda\bar{\otimes}\id)$ of bimodule operators $(\Hilb\mathcal{C})^\mathcal{A}$ strongly converges to $0$ too. Note that as the operators are uniformly bounded, by \Cref{limsuptrick} it is enough to show convergence on finite sums of simple tensors, hence we can show it just on simple tensors:

$$\begin{array}{rl}\|(T_\lambda\otimes\id)(e\otimes f)\|^2 &= \|T_\lambda(e)\otimes f\|^2
\\&=\|\langle f,F(\langle T_\lambda(e),T_\lambda(e)\rangle)(f)\rangle\|
\\&\leq\|F(\langle T_\lambda(e),T_\lambda(e)\rangle)\|\|\langle f,f\rangle\|
\\&\leq\|T_\lambda(e)\|^2\|f\|^2
\end{array}$$
where the first inequality follows from the last statement of \Cref{positiveineq}, and the second from \Cref{starfuncisbounded}. But by hypothesis the last term goes to zero. Hence $(T_\lambda\bar{\otimes}\id)$ goes to 0 strongly.
\end{proof}

\begin{rmk}From now on we will write `strong' instead of `strongly continuous on bounded subsets', for brevity.\end{rmk} Note that by \Cref{2topologies}, this agrees with the terminology in \cite{AntounVoigt}. We follow with a lemma that says the functoriality of tensor products also holds on the left.
\begin{lemma}\label{lefttensorfunctorial}If $E$ is a right Hilbert $\mathcal{A}-\mathcal{B}$ bimodule, the left tensor product with $E$ induces a strong unital  $C^*$-functor $E\bar{\otimes}_\mathcal{B}-:\Hilb\mathcal{C}^\mathcal{B}\rightarrow\Hilb\mathcal{C}^\mathcal{A}$.

\end{lemma}

We omit the proof as it proceeds similarly to the above two results. It turns out that if $E$ is non-degenerate, the above functor has image entirely in non-degenerate bimodules:

\begin{lemma}\label{nondegencompo}
    If $E:\mathcal{A}\rightarrow\Hilb\mathcal{B}$ is a non-degenerate right Hilbert bimodule and $F:\mathcal{B}\rightarrow\Hilb\mathcal{C}$ is a right Hilbert bimodule, then the bimodule $E\Bar{\otimes}_\mathcal{B}F:\mathcal{A}\rightarrow\Hilb\mathcal{C}$ is also non-degenerate.
\end{lemma}
\begin{proof}
    By the characterization in \Cref{nondegenstrongthm}, it suffices to show that for any element $g\in E\Bar{\otimes}_\mathcal{B}F(x)(z)$, and approximate unit $(u_\lambda)$ for $\mathcal{A}(x,x)$, that $u_\lambda\cdot g\xrightarrow[]{\lambda}g$ in norm. Consider first a simple tensor $e\otimes f\in E(x)(y)\otimes F(y)(z)$: note that $$\begin{array}{rll}\|u_\lambda\cdot(e\otimes f)-e\otimes f\|=\|( u_\lambda\cdot e-e)\otimes f\|=\|\langle f,F(\langle u_\lambda\cdot e-e,u_\lambda\cdot e-e\rangle)(f)\rangle\|\end{array}$$ and this last term tends to zero since $u_\lambda\cdot e\xrightarrow[]{\lambda} e$ in the norm and $F$ is bounded. 
    
    The same result follows for finite sums of simple tensors, and by \Cref{limsupid} we obtain the case for general $g$, since every $g$ is by definition a norm-limit of finite sums of simple tensors.
\end{proof}

Letting now $\mathcal{A}=\mathbb{C}$ and noting that a non-degenerate right Hilbert bimodule whose left action comes from $\mathbb{C}$ is simply a right Hilbert module, we easily derive the following proposition from \Cref{lefttensorfunctorial} and \Cref{nondegencompo}:

\begin{prop}\label{tensorfunctor} Given a right Hilbert $\mathcal{B}-\mathcal{C}$ bimodule $E$, the tensor product above defines a strong unital $C^*$-functor $-\bar{\otimes}_\mathcal{B}E:\Hilb\mathcal{B}\rightarrow\Hilb\mathcal{C}$.
\end{prop}

We derive a corollary that characterizes which bimodules preserve compact operators upon tensoring.
\begin{cor}\label{compactpreserve}
For a given right Hilbert $\A-\B$ bimodule $E:\mathcal{A}\rightarrow\Hilb\mathcal{B}$, the tensor product functor $-\bar{\otimes}_\mathcal{A}E:\Hilb\mathcal{A}\rightarrow\Hilb\mathcal{B}$ preserves compact morphisms if and only if $E$ has image entirely in compact morphisms.
\end{cor}
\begin{proof}
Since by \Cref{tensorextend} we have $E\cong(-\bar{\otimes}_\mathcal{A}E)\circ\iota_\mathcal{A}$ and by \Cref{imageyoneda} the image of $\iota_{\mathcal{A}}$ consists of compact operators, the `only if' direction is clear.

For the converse, note that by \Cref{yonedagenerate} the ideal $\KHilb\mathcal{A}$ is generated by the image of $\iota_\mathcal{A}$, so if $E$ has image in $\KHilb\mathcal{B}$, we have that $-\bar{\otimes}_\mathcal{A}E$ sends all of $\KHilb\mathcal{A}$ to compact morphisms since it is an extension of $E$ through $\iota_\mathcal{A}$.
\end{proof}


The point of the following section is to show a converse to \Cref{tensorfunctor}: that is, to prove that up to unitary isomorphism, tensor products by right Hilbert $\mathcal{B}-\mathcal{C}$ bimodules in fact make up \textit{all strong unital functors} from $\Hilb\mathcal{B}$ to $\Hilb\mathcal{C}$.
\subsection{The Eilenberg-Watts theorem}
In this section we prove the $C^*$-categorical Eilenberg-Watts theorem, which establishes a 1-1 correspondence between bimodules and strong functors on categories of Hilbert modules:
\begin{thm}[The Eilenberg-Watts Theorem for $C^*$-categories]\label{EW}
  Suppose that $\mathcal{A}$ and $\mathcal{B}$ are $C^*$-categories, and that  $F:\Hilb\mathcal{A}\rightarrow\Hilb\mathcal{B}$ is a strong unital $C^*$-functor. Let $E=F\circ\iota_{\A}:\mathcal{A}\rightarrow\Hilb\mathcal{B}$ be the bimodule obtained by precomposing $F$ with the Yoneda embedding $\iota_\mathcal{A}:\mathcal{A}\rightarrow\Hilb\mathcal{A}$. Then there exists a unitary isomorphism $F\cong-\bar{\otimes}_\mathcal{A}E$ of functors.
\end{thm}
We begin by defining the map and proving its naturality.

\begin{defn}\label{EWMapDefn}
For $F$ and $E$ as above and $M\in\Hilb\mathcal{A}$, we define the \emph{Eilenberg-Watts map} $\psi_M:M\bar{\otimes}_\mathcal{A}E\rightarrow F(M)$ on simple tensors $m\otimes e\in M(x)\otimes E(x)(z)$ by $\psi_M(m\otimes e):=F(\epsilon_m)(e)\in F(M)(z)$, where $\epsilon_m\in\mathcal{K}(h_x,M)$ is defined as in \Cref{contrayoneda}.
\end{defn}

\begin{lemma}
The Eilenberg-Watts map as defined above is well-defined and extends to a map $\psi_M:M\bar{\otimes}_\mathcal{A}E\rightarrow F(M)$ on the completed module which is natural in $M$ and in $F$, and which moreover preserves all inner products.
\end{lemma}
\begin{proof}

We show first that $\psi$ is well-defined on the uncompleted module $M\otimes_\mathcal{A} E$: recall that $m\otimes e\in M(x)\otimes E(x)(z)$ and $m'\otimes e'\in M(x')\otimes E(x')(z)$ are identified in $M\otimes_\mathcal{A}E(z)$ whenever there exists an element $a\in\mathcal{A}(x,x')$ such that we have  $m'\cdot a=m$ and $a\cdot e=e'$. But then 

$$\begin{array}{rll}\psi_x(m'\otimes e')&=F(\epsilon_{m'})(e')&=F(\epsilon_{m'})(a\cdot e)\\&=F(\epsilon_{m'}\circ \iota(a))(e)&=F(\epsilon_{m'\cdot a})(e)\\&=F(\epsilon_m)(e)&=\psi_x(m\otimes e).\end{array}$$

To see that $\psi$ is a map of $\mathcal{B}$-modules, take $m\otimes e\in M(x)\otimes E(x)(z)$ and a morphism $b\in\mathcal{B}(z',z)$. Then note $$\psi((m\otimes e)\cdot b)=\psi(m\otimes(e\cdot b))=F(\epsilon_m)(e\cdot b)=F(\epsilon_m)(e)\cdot b=\psi(m\otimes e)\cdot b$$ as $F(\epsilon_m)$ is a $\B$-module map.

To extend $\psi$ and show that it preserves inner products,  take simple tensors $m\otimes e\in M(x)\otimes E(x)(z)$ and $ m'\otimes e' \in M(y)\otimes E(y)(z)$, and note that from \Cref{contrayoneda} we get that $\epsilon_{m'}^*\epsilon_m=\iota(\langle m',m\rangle)$. So we have 
$$\begin{array}{rll}\langle \psi(m\otimes e),\psi(m'\otimes e')\rangle&=\langle F(\epsilon_m)(e),F(\epsilon_{m'})(e')\rangle&=\langle e, F(\epsilon_{m'}^*\epsilon_m)(e)\rangle\\&=\langle e, E(\langle m,m'\rangle)(e)\rangle&=\langle m\otimes e,m'\otimes e'\rangle.\end{array}$$
Hence we see $\psi$ preserves all inner products on simple tensors, so  by \Cref{densextend} extends to an isometry on $M\bar{\otimes}_{\A} E$ which must also preserve inner products.

Recall from \Cref{contrayoneda} that for a bounded adjointable operator $\phi:M\rightarrow N$ we have $\phi\circ\epsilon_m=\epsilon_{\phi(m)}$. To prove naturality of $\psi_M$ in $M$, we must show that $F(\phi)\circ\psi_M=\psi_N\circ(\phi\otimes\id)$. But note for $m\otimes e\in M\otimes E$ that $$\begin{array}{rll}(F(\phi)\circ\psi_M)(m\otimes e)&=F(\phi)F(\epsilon_m)(e)&=F(\phi\epsilon_m)(e)\\&=F(\epsilon_{\phi(m)})(e)&=\psi_N(\phi(m)\otimes e)\\&&=\psi_N((\phi\otimes\id)(m\otimes e))\end{array}$$ proving naturality on simple tensors, and hence again on the whole module.
\end{proof}

Having defined the natural transformation $\psi$, we show in a few steps that it is a natural unitary isomorphism, beginning with the case for representable modules.

\begin{prop}\label{EWonRep}
For $F:\Hilb\mathcal{A}\rightarrow\Hilb\mathcal{B}$ and $\psi:-\bar{\otimes}_\mathcal{A}(F\circ\iota)\rightarrow F$ as above, $\psi$ is a unitary isomorphism on the subcategory of representable $\mathcal{A}$-modules and their finite direct sums. 
\end{prop}
\begin{proof}
We have shown so far that $\psi$ preserves all inner products, so by \Cref{unitarycrit} we only have left to show that for each $x\in\Ob\mathcal{A}$, the $\mathcal{B}$-module map $\psi_{h_x}:h_x\otimes_\mathcal{A}(F\circ\iota)\rightarrow F(h_x)$ is surjective at each $y\in\Ob\mathcal{B}$. But this is easily seen to be an instance of \Cref{tensorextend}.



Finally, $F$, being a unital $C^*$-functor, must preserve finite direct sums, and we easily deduce $\psi$ is a unitary isomorphism on these, too.
\end{proof}
We are now in a position to prove our main theorem using the finite free approximation result at the end of the previous section.

\begin{thm}[{c.f. \cite[Theorem 5.4]{BlecherModules}}]\label{EWthm}
For every right Hilbert $\mathcal{A}$-module $M$, the map of right Hilbert $\mathcal{B}$-modules $\psi_M:M\otimes(F\circ\iota)\rightarrow F(M)$ is a unitary isomorphism.
\end{thm}
\begin{proof}
As above, we only have left to show that $\psi_M$ is surjective at every $y\in\Ob\mathcal{B}$. Recall from \Cref{approxff} that there exists a set of operators $\phi_\lambda:\bigoplus_{x\in S_{\lambda}}h_x\rightarrow M$, where each $S_\lambda:\lambda\in\Lambda$ is a finite list of $\mathcal{B}$-objects, such that the net of operators $\phi_\lambda\phi_\lambda^*$ converges strongly to $\id_M$. Writing $S$ for $S_\lambda$ we have a diagram

\begin{center}
\begin{tikzcd}
M\bar{\otimes}_\mathcal{A}E \arrow[d, "\phi_\lambda^*\otimes\id"', shift right] \arrow[r, "\psi_M"]          & F(M) \arrow[d, "F(\phi^*_\lambda)"', shift right]             \\
(\oplus_Sh_s)\bar{\otimes}_\mathcal{A} E \arrow[u, "\phi_\lambda\otimes\id"', shift right] \arrow[r, "\psi_{\oplus_Sh_s}"] & F(\oplus_Sh_s) \arrow[u, "F(\phi_\lambda)"', shift right]
\end{tikzcd}
\end{center}
where the two squares with vertical arrows pointing in the same direction both commute by the naturality of $\psi$. We aim to show the top horizontal map is surjective at any $y\in\Ob\mathcal{B}$, so take an arbitrary $f\in F(M)(y)$. We know by \Cref{EWonRep} that $\psi_{\oplus_Sh_s}$ is surjective at every object, so we have $F(\phi^*_\lambda)(f)=\psi_{\oplus_Sh_s}(t)$ for some $t\in(\oplus_Sh_s)\bar{\otimes}_\mathcal{A} E(y)$. 

So we have $$F(\phi_\lambda\phi_\lambda^*)(f)=F(\phi_\lambda)(\psi_{\oplus_Sh_s}(t))=\psi_M((\phi_\lambda\otimes\id)(t))\in\text{im }\psi_M.$$ But $F$ is strong$^*$ and unital and $\phi_\lambda\phi_\lambda^*$ is a bounded net, so $F(\phi_\lambda\phi^*_\lambda)\xrightarrow[]{\lambda}\id_{F(M)}$ strongly, and we get $F(\phi_\lambda\phi_\lambda^*)(f)\xrightarrow[]{\lambda} f$. So $\psi_M$ has dense image, so by \Cref{densextend} it is surjective. 
\end{proof}

There is a related result about module categories that uses the approximate projectivity of Hilbert modules in a similar way:

\begin{lemma}\label{repdetermine}
    Suppose $F,F':\Hilb\mathcal{A}\rightarrow\Hilb\mathcal{B}$ are strong unital $C^*$-functors, and let $\eta,\zeta:F\rightarrow F'$ be two (bounded, adjointable) natural transformations. If $\eta$ and $\zeta$ agree on representable modules, then $\eta=\zeta$. 
\end{lemma}
\begin{proof}
Since $\eta-\zeta:F\Rightarrow F'$ is another bounded adjointable natural transformation, it suffices to show that if $\xi:F\Rightarrow F'$ is zero on representable modules, then $\xi=0$.    

Recall once more from \Cref{approxff} that there exists a set $\{\phi_\lambda:\lambda\in\Lambda\}$ of bounded adjointable operators $\phi_\lambda:\bigoplus_{x\in S_{\lambda}}h_x\rightarrow M$ (where $S_\lambda$ is a finite list of $\mathcal{B}$-objects for each $\lambda\in\Lambda$) such that the net of operators $\phi_\lambda\phi_\lambda^*$ converges strongly to $\id_M$. We hence have for each $S:=S_\lambda$ a diagram

\begin{center}
\begin{tikzcd}
F(M) \arrow[d, "F(\phi^*_\lambda)"', shift right] \arrow[r, "\xi_M"]          & F'(M) \arrow[d, "F'(\phi^*_\lambda)"', shift right]             \\
F(\oplus_Sh_s) \arrow[u, "F(\phi_\lambda)"', shift right] \arrow[r, "\xi_{\oplus_Sh_s}"] & F'(\oplus_Sh_s) \arrow[u, "F'(\phi_\lambda)"', shift right]
\end{tikzcd}
\end{center}
We want to show for an arbitrary $f\in F(M)(x),x\in\Ob\mathcal{B}$ that $\xi_M(f)=0$. Since $F(\phi_\lambda)F(\phi_\lambda^*)\xrightarrow[]{\lambda}\id_{F(M)}$ strongly, we have  $F(\phi_\lambda)F(\phi_\lambda^*)(f)\rightarrow f$. At the same time, since $\xi_{\oplus_Sh_s}=0$, by naturality of $\xi$ we get $$(\xi_M\circ F(\phi_\lambda)\circ F(\phi_\lambda^*))(f)=(F'(\phi_\lambda)\circ\xi_{\oplus_Sh_s}\circ F(\phi_\lambda^*))(f)=0.$$ Hence we have a net of elements converging in the norm to $f$ which are in the kernel of the bounded map $\xi_M$, so $\xi_M(f)=0$. 
\end{proof}

Another way of formulating this result is: if $\eta:-\bar{\otimes}_\mathcal{A}E\rightarrow\bar{\otimes}_\mathcal{A}E'$ is any transformation of tensor product functors, it is in fact given on each module by tensoring with a fixed transformation $\epsilon$, which can be obtained by whiskering $\eta$ through $\iota_\mathcal{A}$.
\subsection{The Morita theory of $C^*$-categories}

Up until now we have discussed bimodules which have inner products valued in the $C^*$-category acting on the right. We will now discuss bimodules with products in both categories, and see that a subclass of these `bi-Hilbert' bimodules are exactly those bimodules that give Morita equivalences, that is equivalences of module categories. 

The thesis \cite{Ferrier} investigated the same question for small $C^*$-categories, with one key step being a result (\cite[Theorem 6.3.1]{Ferrier}) saying that each small $C^*$-category is Morita equivalent to a $C^*$-algebra. We do not assume smallness so we take a somewhat different approach, adapting directly the results from \cite[Chapter 3]{RaeWill} and \cite[Chapter 1]{Echterhoff} that characterize which bimodules give Morita equivalences in the $C^*$-algebra case. The aforementioned category-algebra equivalence becomes an example in this framework.

Recall from \Cref{minimalunit} that for a $C^*$-category $\A$ we denote by $\A^+$ the \emph{minimal unitization} of $\A$.

\begin{defn}\label{Bihilbdefn}
If $\mathcal{A}$ and $\mathcal{B}$ are $C^*$-categories, a \emph{bi-Hilbert} $\mathcal{A}-\mathcal{B}$ \emph{bimodule} is a functor $E:\mathcal{A}^+\times\mathcal{B}^{\mathrm{op}+}\rightarrow\mathsf{Vect}_\mathbb{C}$ which is equipped for all $x\in\Ob\mathcal{A},y\in\Ob\mathcal{B}$ with $\mathcal{A}$-valued products $$_\mathcal{A}\langle-,-\rangle:E(x')(y)\times E(x)(y)\rightarrow\mathcal{A}(x,x')$$ and $\mathcal{B}$-valued products $$\langle-,-\rangle_\mathcal{B}:E(x)(y')\times E(x)(y)\rightarrow \mathcal{B}(y,y')$$ which make $E$ both an $\mathcal{A}-\mathcal{B}$ right Hilbert bimodule and a $\mathcal{B}^{\mathrm{op}}-\mathcal{A}^{\mathrm{op}}$ right Hilbert bimodule\footnotemark. Unpacking this a little, we require that
\begin{itemize}
    \item $E(x)(-)$ is a right Hilbert $\mathcal{B}$-module for each $x\in\Ob\mathcal{A}$, with product $\langle -,-\rangle_\mathcal{B}$ and right action by morphisms of the form $(\id_x,b)$. 
    \item $E(-)(y)$ is a left Hilbert $\mathcal{A}$-module for each $y\in\Ob\mathcal{B}$, with product $_\mathcal{A}\langle-,-\rangle$ and left action by morphisms of the form $(a,\id_y)$
    \item The action of $a\in\mathcal{A}$ on the $\mathcal{B}$-modules is bounded by $\|a\|$ and adjoint to the action of $a^*$; similarly for the action of $b\in\mathcal{B}$ on the $\mathcal{A}$-modules. Writing this out in equations, we get $$\begin{array}{rlcrl}
         \langle e,a\cdot f\rangle_\mathcal{B}&=\langle a^*\cdot e,f\rangle_\mathcal{B} &\mathrm{and}& \|\langle a\cdot e,a\cdot e\rangle_\mathcal{B}\|&\leq \|a\|^2\|\langle e,e\rangle_\mathcal{B}\|  \\ _\mathcal{A}\langle e, f\cdot b\rangle&={}_\mathcal{A}\langle e\cdot b^*,f\rangle &\mathrm{and}& \|{}_\mathcal{A}\langle  e\cdot b, e\cdot b\rangle\|&\leq \|b\|^2\|{}_\mathcal{A}\langle e,e\rangle\|.
    \end{array} $$
    
\end{itemize}
\end{defn}
We write $E^1:\mathcal{A}\rightarrow\Hilb\mathcal{B}$ and $E^2:\mathcal{B}^{\mathrm{op}}\rightarrow\Hilb\mathcal{A}^\mathrm{op}$ for the two bimodules when it is necessary to consider them separately, but will generally speak simply of $E$. These two bimodules are automatically non-degenerate:
\footnotetext{To see how this works, notice that functors from non-unital  into unital $\mathbb{C}$-categories always extend to the minimal unitization (as defined in \Cref{minimalunit}), so a bi-Hilbert $\mathcal{A}$-$\mathcal{B}$-bimodule is an object of $\mathrm{Func(}\mathcal{A}^+,\mathrm{Func}(\mathcal{B}^{\mathrm{op}+},\mathsf{Vect}_\mathbb{C})\cong\mathrm{Func(}\mathcal{A}^+\times\mathcal{B}^{\mathrm{op}+},\mathsf{Vect}_\mathbb{C})\cong\mathrm{Func}(\mathcal{B}^{\mathrm{op}+},\mathrm{Func}(\mathcal{A}^+,\mathsf{Vect}_\mathbb{C}))$. We use the minimal unitizations to ensure  the $\mathcal{A}$ and $\mathcal{B}$-actions are defined in isolation.}

\begin{lemma}\label{bihilbnondegen}
If $E$ is a bi-Hilbert bimodule, the bimodules $E^1:\mathcal{A}\rightarrow\Hilb\mathcal{B}$ and $E^2:\mathcal{B}^{\mathrm{op}}\rightarrow\Hilb(\mathcal{A}^\mathrm{op})$ are non-degenerate.
\end{lemma}
\begin{proof}
Take any $x\in\Ob\mathcal{A},y\in\Ob\mathcal{B},e\in E(x)(y)$,  and note that we can identify $E(x)(y)=E^1(x)(y)=E^2(y)(x)$. By \Cref{nondegenstrongthm} it suffices to show that for any approximate unit $(u_\lambda)$ for $\mathcal{A}(x,x)$ that $E^1(u_\lambda)(e)\xrightarrow[]{\lambda}e$. But note that $E^1(u_\lambda)(e)=e\cdot u_\lambda\in E^2(y)(x)$, and that $(u_\lambda)$ is certainly an approximate unit of $\mathcal{A}^\mathrm{op}(x,x)$. So by \Cref{unitcor} we have that $e\cdot u_\lambda\xrightarrow[]{\lambda}e$, proving the lemma. The non-degeneracy of $E^2$ follows similarly.
\end{proof}

\begin{defn}\label{fulldefn}
For any right Hilbert bimodule $E:\mathcal{A}\rightarrow\Hilb\mathcal{B}$, we let $\langle E,E\rangle$ denote the subcategory of $\mathcal{B}$ given by the linear span of the inner products $\langle e,f\rangle\in\mathcal{B}(y,y')$ for  all objects $x\in\Ob\mathcal{A},y,y'\in\Ob\mathcal{B}$ and module elements $e\in E(x)(y'),f\in E(x)(y)$. Noting that this subcategory is closed under the involution and under multiplication by outside elements, we denote by $\overline{\langle E,E\rangle}$ the ideal given by its norm closure and we say that $E$ is \emph{essential} if $\overline{\langle E,E\rangle}$ is an essential ideal and that $E$ is \emph{full} if $\overline{\langle E,E\rangle}=\mathcal{B}$. 
\end{defn}

\begin{xmpl}\label{yonedafull}
The Yoneda bimodule $\iota_\mathcal{A}:\mathcal{A}\xhookrightarrow[]{}\Hilb\mathcal{A}$ is full: note that $\langle\iota_\mathcal{A},\iota_\mathcal{A}\rangle$ consists of all morphisms of $\mathcal{A}$ which have some factorization, but by \Cref{factorlemma} this is all of $\mathcal{A}$. 
\end{xmpl}
We are interested in bi-Hilbert bimodules that satisfy an additional axiom on the inner products:
\begin{defn}\label{imprimdefn}
A bi-Hilbert $\mathcal{A}-\mathcal{B}$ bimodule $E$ is termed \emph{a partial imprimitivity} $\mathcal{A}-\mathcal{B}$ \emph{bimodule} when for all elements $e\in E(x)(y), f\in E(x')(y)$ and $g\in E(x')(y')$ we have $$_\mathcal{A}\langle e,f\rangle\cdot g=e\cdot\langle f,g\rangle_\mathcal{B},$$ noting this is an equation in $E(x)(y')$. It is called\footnotemark\vspace{0pt} \emph{right essential} if the bimodule $E^1:\mathcal{A}\rightarrow\Hilb\mathcal{B}$ is essential, \emph{right full} if $E^1$ is full, \emph{left-essential} if $E^2:\mathcal{B}^{\mathrm{op}}\rightarrow\Hilb\mathcal{A}^\mathrm{op}$ is essential and \emph{left-full} if $E^2$ is full. A partial imprimitivity bimodule $E$ is called an $\mathcal{A}-\mathcal{B}$ \emph{imprimitivity bimodule} if both $E^1$ and $E^2$ are full. 
\end{defn}
\footnotetext{A note on the terminology here: in \cite{Echterhoff}, bimodules which are left full in our terminology  are termed `right partial', and vice versa. We modify their terminology here since strictly it implies that modules have full products on both sides when they are both left and right partial, a rather unintuitive statement.}
\begin{xmpl}
The Yoneda bimodule $\iota_\mathcal{A}:\mathcal{A}\xhookrightarrow[]{}\Hilb\mathcal{A}$ can be given the structure of a bi-Hilbert bimodule by setting for $a\in h_x(y),b\in h_z(y)$ the product $_\mathcal{A}\langle a,b\rangle:=a\circ b^*$, and we see immediately that this is a partial imprimitivity bimodule since $_\mathcal{A}\langle a,b\rangle\cdot c=a\circ b^*\circ c=a\cdot\langle b,c\rangle_\mathcal{A}$. As explained in the previous example, it is evident that both products are full, so $\iota_\mathcal{A}$ is in fact an $\mathcal{A}-\mathcal{A}$ imprimitivity bimodule.
\end{xmpl}

\begin{lemma}\label{imprimnormequiv}
If $E$ is a partial imprimitivity $\mathcal{A}-\mathcal{B}$ bimodule, the two norms on each space $E(x)(y)$ given by the $\mathcal{A}$- and $\mathcal{B}$-valued product are equal. 
\end{lemma}
\begin{proof}
Note that 

$$\begin{array}{rll}
     \|\langle e,e\rangle_\mathcal{B}\|^2&=\|\langle e,e\rangle_\mathcal{B}\langle e,e\rangle_\mathcal{B}\|
     &=\|\langle e,e\cdot \langle e,e\rangle_\mathcal{B}\rangle_\mathcal{B}\| \\&=\|\langle e,{}_\mathcal{A}\langle e,e\rangle\cdot e\rangle_\mathcal{B} &\leq\|\langle e,e\rangle_\mathcal{B}\|\|_\mathcal{A}\langle e,e\rangle\|
     
\end{array} $$

So for $e\neq0$ we can divide to get $\|\langle e,e\rangle_\mathcal{B}\|\leq\|_\mathcal{A}\langle e,e\rangle\|$, and it follows similarly that $\|_\mathcal{A}\langle e,e\rangle\|\leq\|\langle e,e\rangle_\mathcal{B}\|$, proving the lemma.
\end{proof}
The imprimitivity equation gives us significant control over both actions on the bimodule:
\begin{lemma}\label{imprimcompact}
If $E$ is a partial imprimitivity $\mathcal{A}-\mathcal{B}$ bimodule, we have that  $$E^1(_{\A}\langle E^2,E^2\rangle(x,x'))\subseteq\mathcal{L}(E^1(x),E^1(x'))$$ and

$$E^2(\langle E^1,E^1\rangle_{\B}(y,y'))\subseteq\mathcal{L}(E^2(y),E^2(y'))$$ both consist of exactly the finite-rank operators.

\end{lemma}

\begin{proof}
The identity in \Cref{imprimdefn} states that $E^1(_\mathcal{A}\langle e,f\rangle)(g)=\theta^{e,f}(g)$, so as $E^1$ is additive and continuous we get the first result. 

The analogous result on $E^2$ follows similarly as the identity in \Cref{imprimdefn} gives $E^2(\langle f,g\rangle_\mathcal{B})(e)=\theta^{g,f}(e)$.
\end{proof}

Given additional information about the products, we can deduce that the action is entirely compact:

\begin{lemma}\label{imprimaction}
    If $E$ is a left-essential partial imprimitivity $\mathcal{A}-\mathcal{B}$ bimodule, then in fact  $E^1$ is an isometry on hom-sets and satisfies $$E^1(_{\A}\overline{\langle E^2,E^2\rangle}(x,x'))=\mathcal{K}(E^1(x),E^1(x')),$$ and if $E$ is left full, $E^1$ gives a surjection of $\mathcal{A}(x,x')$ onto $\mathcal{K}(E^1(x),E^1(x'))$. 

Similarly, if $E$ is right-essential, $E^2$ acts isometrically and we have $$E^2(\overline{\langle E^1,E^1\rangle_\mathcal{B}}(y',y))=\mathcal{K}(E^1(y),E^1(y')),$$ and if $E$ is right-full, $E^2$ in addition surjects $\mathcal{B}(y',y))$ onto $\mathcal{K}(E^1(y),E^1(y'))$.
\end{lemma}
\begin{proof}
We show that if $E$ is left essential, $E^1$ acts isometrically on hom-sets. Suppose $E^1(a)=0$, then for all $e,f$ we have $a\circ {}_\mathcal{A}\langle e,f\rangle={}_\mathcal{A}\langle E^1(a)(e),f\rangle=0$, so as $\langle E^2,E^2\rangle$ is essential, we in fact have $a=0$. Hence $E^1$ is faithful and by \Cref{starfuncisbounded} in fact acts isometrically, and in particular preserves closures of sets. If $E^2 $ is in fact full, clearly the final result on $E^1$ follows.

The dual results follow similarly.
\end{proof}
We'd like to build up to a converse to the above lemma, that specifies exactly when a right Hilbert bimodule is one `half' of an imprimitivity bimodule.

\begin{lemma}\label{DyHilbMod}
    If $\D$ is a subcategory of $\HilbB$ containing all the compact operators between its objects, then for every $y\in\Ob\B$ there is a large left Hilbert $\D$-module $$F_y:\D\rightarrow\mathsf{Vect}_{\mathbb{C}}$$ given by \begin{itemize} \item $F_y(D):=D(y)$ for all $D\in\Ob\D$, 
    \item $F_y(T):=T_y:D(y)\rightarrow D'(y)$ for all $T\in\mathcal{L}(D,D')$, and
    \item $_{\D}\langle d,d'\rangle:=\theta_y^{d,d'}\in\mathcal{K}(D',D)$
    \end{itemize}
\end{lemma}
\begin{proof}
Note firstly that $\theta^{e,f}=(\theta^{f,e})^*$ and that $T\circ\theta^{e,f}=\theta^{T(e),f}$. To see for any $E\in\Ob\mathcal{D},e\in E(y),$ that $_\mathcal{D}\langle e,e\rangle=\theta^{e,e}_y\in\mathcal{KD}(E,E)$ is positive, note that for all $f\in E(x)$ we have $\langle f,\theta_y^{e,e}(f)\rangle=\langle f,e\rangle\langle e,f\rangle\geq0$, so by \Cref{positivitycat} we get positivity. Finally if $\theta^{e,e}_y=0$, then $\langle e,e\rangle^2=\langle e,\theta^{e,e}(e)\rangle=0 $ so $e=0$. 

Every space in the bimodule is complete in this norm as it is equivalent to the $\mathcal{B}$-norm; note that the proof of \Cref{imprimnormequiv}
uses only the imprimitivity equation and the compatibility of the actions. The action of $\mathcal{B}$ is bounded in the $\mathcal{D}$-norm since by \Cref{CSLemma} it is bounded in the $\mathcal{B}$-norm and the two norms are equal, and finally it is adjointable since $\theta^{e,f\cdot b}=\theta^{e\cdot b^*,f}$. This completes the proof.
\end{proof}

\begin{defn}
    A subcategory $\D\subseteq\HilbB$ containing all compact operators between its objects is called \emph{left small} if for all objects $y\in\Ob\B$, the large left Hilbert $\D$-module $F_y$ is small, that is, generated by a set of elements. A bimodule $E:\A\rightarrow\HilbB$ is called left small if its image is.
\end{defn}

\begin{xmpl}
    If $\D\subseteq\HilbB$ is small then $\D$ is left small by \Cref{smallhilboversmallcat}.
\end{xmpl}
\begin{lemma}\label{containsrepssmall}
    If $\D\subseteq\HilbB$ contains every representable right Hilbert $\B$-module, then it is left small.
\end{lemma}
\begin{proof} 
Recall that by \Cref{CohHewCstar}, for all $D\in\Ob\D,y\in\Ob\B$ and $d\in D(y)$ there exist $d'\in D(y)$ and $b\in\B(y,y)$ such that $d=d'\cdot b$. But then $d=\epsilon_{d'}(b)$, where $\epsilon_{d'}\in\mathcal{K}(h_y,D)$ is the map defined in \Cref{contrayoneda}. Therefore we see that for each $y\in\Ob\B$, the set $\B(y,y)\subseteq h_y(y)$ generates the $\D$-module $F_y$.
\end{proof}
We can promote the above left-valued product to a left essential imprimitivity bimodule: 
\begin{prop}\label{giveimprimstruc}Suppose $\mathcal{D}\subseteq\HilbB$ is a left small subcategory containing all compact operators between its objects. Then the inclusion $$\Delta_1:\mathcal{D}\xhookrightarrow{}\HilbB$$ can be given the structure of a left-essential $\mathcal{D}$-$\mathcal{B}$ imprimitivity bimodule $\Delta$ with other half $$\Delta_2:\B^{\mathrm{op}}\rightarrow\mathsf{Hilb}\textsf{-}\D^{\mathrm{op}}$$ defined by setting $\Delta_2(y)=F_y$, where the latter is defined in \Cref{DyHilbMod}. Furthermore, $\Delta$ is an imprimitivity bimodule if and only if $\mathcal{D}$ contains precisely the compact operators and the bimodule $\Delta_1:\mathcal{D}\xhookrightarrow[]{}\HilbB$ is full. \end{prop}
\begin{proof}
The left action of $\mathcal{B}$ on the modules $F_y$ is obvious, and the imprimitivity equation follows immediately from the definition of $_{\mathcal{D}}\langle-,-\rangle$.

Left essentiality follows \textit{a fortiori} from the fact that $\KHilbB$ is essential in $\HilbB$, and left fullness is clearly equivalent to asking that $\mathcal{D}$ contains only compact operators. So $\Delta$ is an imprimitivity bimodule with described left product if it is also right full.\end{proof}

\begin{cor}\label{KHilbImprim} Every $C^*$-category $\B$ is linked to its own category of right Hilbert modules and compact operators $\KHilb\B$ by an imprimitivity bimodule. \end{cor}\begin{proof} In the above lemma, simply let $\D=\KHilb\B$: this is clearly full by considering for example the inner products on representable modules, and by \Cref{containsrepssmall} it is left small. \end{proof}

\Cref{giveimprimstruc} and \Cref{imprimaction} together allow us to give a complete characterization of which bimodules come from imprimitivity bimodules:
\begin{prop}\label{imprimchar}
A bimodule $E^1:\mathcal{A}\rightarrow\HilbB$ can be given the structure of a left-full partial imprimitivity bimodule if and only if $E^1$ is left small and is a faithful functor that surjects onto the compact operators between the modules in its image. In this case the $\mathcal{A}$-valued product must equal $_{\mathcal{A}}\langle e,f\rangle=\theta^{e,f}$. This is an imprimitivity bimodule if and only if $E^1$ is full.
\end{prop}

Hence we see that for a given right Hilbert bimodule, being an imprimitivity bimodule or not is a \textit{property} rather than a structure.

The reason to study imprimitivity bimodules is that as in the $C^*$-algebra case, they turn out to be exactly the bimodules that induce strong unitary equivalences of Hilbert module categories.

\begin{defn}
If $E:\mathcal{A}^+\times\mathcal{B}^{\mathrm{op}+}\rightarrow\mathsf{Vect}_\mathbb{C}$ is a bi-Hilbert $\mathcal{A}-\mathcal{B}$ bimodule, its \emph{conjugate} $\widetilde{E}:\mathcal{B}^+\times\mathcal{A}^{\mathrm{op}+}\rightarrow\mathsf{Vect}_\mathbb{C}$ is the bi-Hilbert $\mathcal{B}-\mathcal{A}$ bimodule defined as follows:
\begin{itemize}
    \item For $x\in\Ob\mathcal{A}$ and $y\in\Ob\mathcal{B}$ we set $\widetilde{E}(y)(x)=E(x)(y)^*$, where $^*$ denotes the conjugate of a complex vector space.
    \item For any morphisms $a\in\mathcal{A}^+(x_1,x_2)$ and $b\in\mathcal{B}^+(y_1,y_2)$, and $e\in E(x_2,y_2)$, we define the action\footnotemark\hspace{0pt} by $\widetilde{E}(b,a)(\widetilde{e})=\widetilde{E(a^*,b^*)(e)}$, where e.g. $\widetilde{e}$ signifies the element of $\widetilde{E}(x_2,y_2)$  corresponding to $e\in E(x_2,y_2)$. 
    \item $\bar{E}$ has products $_\mathcal{B}\langle-,-\rangle$ and $\langle-,-\rangle_\mathcal{A}$ defined by setting ${}_\mathcal{B}\langle \widetilde{e},\widetilde{f}\rangle=\langle f,e\rangle_\mathcal{B}$ and $\langle \widetilde{e},\widetilde{f}\rangle_\mathcal{A}={}_\mathcal{A}\langle f,e\rangle$ where e.g. $\widetilde{e}\in\widetilde{E}(y)(x)$ denotes the element of the conjugate space corresponding to $e\in E(x)(y)$.
\end{itemize}
\end{defn}
\footnotetext{Note that $\widetilde{E}$ is a $\mathbb{C}$-linear functor since $\widetilde{E}(\lambda(a,b))(\widetilde{e}):=\widetilde{(E(\overline{\lambda}a^*,\overline{\lambda}b^*)(e))}=\widetilde{(\overline{\lambda}E(a^*,b^*)(e))}=\lambda(\widetilde{E(a^*,b^*)(e)})=\lambda\widetilde{E}(a,b)(\widetilde{e})$. } We denote by $\widetilde{E}^1:\mathcal{B}\rightarrow\Hilb\mathcal{A}$ and $\widetilde{E}^2:\mathcal{A}^\mathrm{op}\rightarrow\Hilb\mathcal{B}^\mathrm{op}$ the two associated bimodules. It is immediate that many properties of $\widetilde{E}$ carry over from $E$, namely:
\begin{itemize}
    \item $\widetilde{E}$ is a partial imprimitivity bimodule if and only if $E$ is.
    \item $\widetilde{E}$ is left essential/full, respectively right essential/full if and only if $E$ is right essential/full, respectively left essential/full.
\end{itemize}

Imprimitivity bimodules were termed `equivalence bimodules' in \cite{Ferrier}: we stick here to imprimitivity bimodules to emphasize the analogy with the $C^*$-algebra case.

We begin by proving that an imprimitivity bimodule is invertible under the tensor product.
\begin{lemma}\label{imprimisinvertible}
If $E$ is a partial imprimitivity bi-Hilbert $\mathcal{A}-\mathcal{B}$ bimodule, there are isometric bimodule maps $\phi:\widetilde{E}^1\bar{\otimes}_\mathcal{A} E^2\rightarrow\iota_\mathcal{B}$ and $\psi:E^2\bar{\otimes}_\mathcal{B} \widetilde{E}^1\rightarrow\iota_\mathcal{A}$. If $E$ is right-full, $\phi$ is a unitary isomorphism, and if $E$ is left-full, $\psi$ is a unitary isomorphism.
\end{lemma}
\begin{proof}
We construct $\phi$ first. For all $y,y'\in\Ob\mathcal{B}$ we define the bimodule map $\phi:\widetilde{E}^1\otimes_\mathcal{B} E^2(y)(y')\rightarrow\iota_\mathcal{B}(y)(y'):=\mathcal{B}(y',y)$ on  tensors $\widetilde{e}\otimes f\in \widetilde{E}^1(y)(x)\otimes E^2(x)(y')$ by $\phi(\widetilde{e}\otimes f):=\langle e,f\rangle_\mathcal{B}$. It's easily verified that this is a bilinear map which is natural with respect to the actions of $\mathcal{A}$ and $\mathcal{B}$, and that it respects the equivalence relation defining $\widetilde{E}^1\otimes_\mathcal{B} E^2(y)(y')$. To show that it extends to an isometry on the completion, we show it respects the inner product on simple tensors:
$$\begin{array}{rll}
\langle \phi(\widetilde{e}\otimes f),\phi(\widetilde{e'}\otimes f)\rangle_\mathcal{B}     &=\langle e,f\rangle_\mathcal{B}^*\langle e',f'\rangle_\mathcal{B} &=\langle f,e\rangle_\mathcal{B}\langle e',f'\rangle_\mathcal{B}  \\
     =\langle f,e\cdot\langle e',f'\rangle_\mathcal{B}\rangle_\mathcal{B} &=\langle f,{}_\mathcal{A}\langle e,e'\rangle\cdot f'\rangle_\mathcal{B}
      &=\langle \widetilde{e}\otimes f,\widetilde{e'}\otimes f'\rangle_\mathcal{B}
\end{array}$$
Hence by \Cref{densextend} $\phi$ extends to an isometric bimodule map on the completed module $\widetilde{E}^1\bar{\otimes}_\mathcal{B} E^2(y)(y')$. This map is evidently surjective when $E$ is right-full, so by \Cref{unitarybimodcrit} it is a unitary equivalence.

The map $\psi$ is defined on simple tensors by $\psi(e\otimes\widetilde{f})={}_\mathcal{A}\langle e,f\rangle$. The proof that $\psi$ preserves inner products is completely analogous to the above, and clearly it is surjective (and hence a unitary equivalence) when $E$ is left-full. 
\end{proof}
We follow with a converse of this lemma, which has the hardest proof in this section.
\begin{lemma}\label{inversesareimprim}
Let $E:\mathcal{A}\rightarrow\Hilb\mathcal{B}$ and $F:\mathcal{B}\rightarrow\Hilb\mathcal{A}$ be non-degenerate right Hilbert bimodules, and suppose there exist bimodule isomorphisms
$\psi:E\bar{\otimes}_\mathcal{B}F\rightarrow\iota_\mathcal{A}$ and 
$\phi:F\bar{\otimes}_\mathcal{A}E\rightarrow\iota_\mathcal{B}$. Then both $E$ and $F$ are full bimodules, and are faithful functors that surject onto the compact operators between the modules in their image.
\end{lemma}
\begin{proof}
By symmetry it is enough to prove the stated properties of $E$. Note that we can assume that all isomorphisms are unitary by \Cref{isoisunitary}. 

Note also that $\langle\iota_\mathcal{B},\iota_\mathcal{B}\rangle=\langle F\bar{\otimes}_{\mathcal{A}}E,F\bar{\otimes}_{\mathcal{A}}E\rangle\subseteq\langle E,E\rangle$ by definition of the inner product on $F\bar{\otimes}_{\mathcal{A}}E$, so $E$ is full. Similarly, $F$ is full since $\iota_\mathcal{A}$ is.

To show $E$ is an isometry, note that we have unitary isomorphisms of functors $$(-\bar{\otimes}_\mathcal{B}F)\bar{\otimes}_\mathcal{A}E\cong-\bar{\otimes}_\mathcal{B}(F\bar{\otimes}_\mathcal{A}E)\cong-\bar{\otimes}_\mathcal{A}\iota_\mathcal{A}\cong\id_{\Hilb\mathcal{A}}.$$ Here we have used \Cref{tensorassociative} for the first isomorphism, the hypothesis for the second and \Cref{yonedaisid} for the third. We hence conclude that $-\bar\otimes_\mathcal{A}E$ must be faithful, but by \Cref{tensorextend}, $-\bar\otimes_\mathcal{A}E$ is an extension of $E$ from $\mathcal{A}$ to $\Hilb\mathcal{A}$. So certainly $E$ is faithful, and hence an isometry by \Cref{starfuncisbounded}.

We end with the most challenging part: showing that the image of $\mathcal{A}(x,x')$ under $E$ is exactly $\mathcal{K}(E(x),E(x'))$. For every $x\in\Ob\mathcal{A},y\in\Ob\mathcal{B}$ we define the map $$\begin{array}{rl}\delta:&F(y)(x)\rightarrow\mathcal{L}(E(x),h_y)\\&\delta(f)(e):=\phi(f\otimes e)\end{array}$$ for each $y'\in\Ob\mathcal{B}$ and $e\in E(x)(y')$. It is clear that $\delta(f)$ is bounded by $\|f\|$, and since $\phi$ is a unitary isomorphism at every $y\in\Ob\mathcal{B}$ we see that $\delta(f)$ is adjointable if and only if the map $\phi_y^*\delta(f):E(x)\rightarrow F\bar{\otimes}_\mathcal{A}E(y)$ is adjointable, where we have $\phi_y^*\delta(f)(e):=f\otimes e$. But $\phi_y^*\delta(f)$ has an obvious bounded adjoint sending $f'\otimes e'$ to $\langle f,f'\rangle_\mathcal{B}\cdot e'$. So $\delta$ has an adjoint $\delta^*$ found by $(\phi_y^*\delta)^*\phi_y$. 

Note that if we fix $y$, not only does the domain of $\delta$ have an inner product $$\langle-,-\rangle_F:F(y)(x')\times F(y)(x)\rightarrow\mathcal{A}(x,x'),$$ but its codomain has an inner product $$\langle-,-\rangle_{\textit{O}}:\mathcal{L}(E(x'),h_y)\times \mathcal{L}(E(x),h_y)\rightarrow \mathcal{L}(E(x),E(x'))\text{ defined by }\langle S,T\rangle_{\textit{O}}=S^*\circ T.$$

We are going to prove the following two claims about the map $\delta$:
\begin{itemize}
    \item For each $f\in F(y)(x),f'\in F(y)(x')$ we have $\langle\delta(f'),\delta(f)\rangle_\textit{O}=E(\langle f',f\rangle_F)$
    \item The image of $\delta$ at every $x\in\Ob\mathcal{A},y\in\Ob\mathcal{B}$ is $\mathcal{K}(E(x),h_y)$.

\end{itemize}
Provided we can prove these claims, the desired result will then roll out as follows: note that as $F$ is full, the span of $\bigcup_{y\in\Ob\mathcal{B}}\langle F(y)(x'),F(y)(x)\rangle_F$ is dense in $\mathcal{A}(x,x')$. Note also by \Cref{singlerank} that the span of $$\bigcup_{y\in\Ob\mathcal{B}}\langle \mathcal{K}(E(x'),h_y),\mathcal{K}(E(x),h_y)\rangle_{\textit{O}}$$ is a dense subset of $\mathcal{K}(E(x),E(x'))$ (in the operator norm, which is the one induced by $\langle-,-\rangle_O$). But $E$ is a faithful $C^*$-functor so in particular acts by linear isometries on hom-sets, it preserves spans and closures, and the second claim in fact gives $E(\mathcal{A}(x,x'))=\mathcal{K}(E(x),E(x'))$.

To prove the first claim, let $f\in F(y)(x),f'\in F(y)(x')$, and take elements $e\in E(x)(y'),e'\in E(x')(y'')$. Then note that we have $$\begin{array}{rl}
  \langle e',\delta(f')^*\delta(f)e\rangle   &=\langle \delta(f')(e'),\delta(f)(e)\rangle  \\
  &= \langle \phi(f'\otimes e'),\phi(f\otimes e)\rangle   \\
  &=\langle f'\otimes e',f\otimes e\rangle 
    \\&=\langle e',E(\langle f',f\rangle)(e)\rangle 
\end{array} $$
Hence if we set $T=\delta(f')^*\delta(f)-E(\langle f',f\rangle)$ we see $\langle e',T(e)\rangle=0$ for all $e,e'$, and setting $e'=T(e)$ we conclude $T=0$, proving the first claim. 

To prove the second claim, note that we have for every $a\in\mathcal{A}(x',x),f\in F(y)(x)$ that $\delta(f\cdot a)=\delta(f)\circ E(a)$, and for every $b\in\mathcal{B}(y,y')$ that $\delta(F(b)(f))=\iota_\mathcal{B}(b)\circ\delta(f)$. From now on, we will denote all these actions as simply $-\cdot a$ and $b\cdot-$, which $\delta$ preserves. Now note that 
$$\begin{array}{rll}
     \delta(F(y)(x)) &\overset{(1)}{=}\delta(\overline{\Span_{y'\in\Ob\mathcal{B}}\mathcal{B}(y,y')\cdot F(y')(x)}) & \\
     &\overset{(2)}{=}\overline{\Span_{y'\in\Ob\mathcal{B}}\iota_\mathcal{B}(y')(y)\circ\delta(F(y')(x))} &\\
       &\overset{(3)}{=}\overline{\Span_{y'\in\Ob\mathcal{B}}(\iota_\mathcal{B}(y)(y'))^*\circ\delta(F(y')(x))} &\\
     &\overset{(4)}{=}\overline{\Span_{y'\in\Ob\mathcal{B}}(\phi(F\bar{\otimes}_\mathcal{A}E)(y)(y'))^*\circ\delta(F(y')(x))} &\\
     &\overset{(5)}{=}\overline{\Span_{y'\in\Ob\mathcal{B},x'\in\Ob\mathcal{A}}\mathcal{K}(E(x'),h_{y})\circ\delta(F(y')(x'))^*\circ\delta(F(y')(x))}&\\
     &\overset{(6)}{=}\overline{\Span_{y'\in\Ob\mathcal{B},x'\in\Ob\mathcal{A}}\mathcal{K}(E(x'),h_{y})\circ E(\langle F(y')(x'),F(y')(x)\rangle)}&\\
     &\overset{(7)}{=}\overline{\Span_{x'\in\Ob\mathcal{A}}\mathcal{K}(E(x'),h_{y})\circ E(\mathcal{A}(x',x))}&\\
     &\overset{(8)}{=}\mathcal{K}(E(x)(h_y))&
\end{array}$$
We justify each of these equalities as follows:
\begin{enumerate}
    \item By non-degeneracy of $F$.
    \item Since $E$ acts isometrically, the first claim we proved shows that $\delta$ is a norm isometry, so it preserves closures, and we also know $\delta$ is natural with respect to the left action of $\mathcal{B}$.

    \item Since the involution of operators is an isometry.
    \item By the surjectivity of $\phi$.
    \item Notice that if $\epsilon_e:h_y\rightarrow E(x')$ is the compact operator corresponding to $e\in E(x')(y)$, then by definition of $\delta$, for all $f\in F(y')(x')$ we have $$\phi(e\otimes f)=\iota_\mathcal{B}(\delta(f)(e))=\delta(f)\circ\epsilon_e:h_y\rightarrow h_{y'},$$ since for all $b\in h_y$, we have $$(\delta(f)\circ\epsilon_e)(b)=\phi(f\otimes e\cdot b)=\phi(f\otimes e)\circ b.$$ Hence $\phi(F\bar{\otimes}_\mathcal{A}E)(y')(y)=\overline{\Span_{x'\in\Ob\mathcal{A}}\delta(F(y')(x'))\circ\mathcal{K}(h_y,E(x'))}$, and applying the involution we deduce this equality.
    \item By the first of the two claims.
    \item By the fullness of $F$.
    \item By the non-degeneracy of $E$.
\end{enumerate}
This concludes the  proof.
\end{proof}
\begin{cor}[{c.f. \cite[Proposition 8.1.3]{Ferrier}}]\label{equivimprim}
Let $E:\mathcal{A}\rightarrow\Hilb\mathcal{B}$ and $F:\mathcal{B}\rightarrow\Hilb\mathcal{A}$ be non-degenerate right Hilbert bimodules, and suppose there exist bimodule isomorphisms
$\psi:E\bar{\otimes}_\mathcal{B}F\rightarrow\iota_\mathcal{A}$ and 
$\phi:F\bar{\otimes}_\mathcal{A}E\rightarrow\iota_\mathcal{B}$.
Then $E$ and $F$ are imprimitivity bimodules and we in fact have $F\cong\widetilde{E}$.
\end{cor}
\begin{proof}
Notice that at any $x\in\Ob\mathcal{A},y\in\Ob\mathcal{B}$, the map $\delta$ in the proof of \Cref{inversesareimprim} gives an isometry from $F(y)(x)$ to $\mathcal{K}(E(x),h_y)$, which is of course conjugate isomorphic to $\mathcal{K}(h_y,E(x))\cong E(x)(y)$. Hence we see that $E$ is left small since $F$ lands in small $\A$-modules and vice versa. 

The imprimitivity structure on $E$ is deduced by combining \Cref{inversesareimprim} with \Cref{imprimchar}. 
\end{proof}

Notice that this directly generalizes the case of $C^*$-algebras, which is explored in e.g. \cite[Chapter 3]{RaeWill}. We are now ready to summarize the results of this subsection into one theorem: 
\begin{thm}\label{moritathm}
Suppose $\A$ and $\B$ are $C^*$-categories. A strong unital functor $F:\Hilb\mathcal{A}\rightarrow\Hilb\mathcal{B}$ is an equivalence if and only if its restriction $E:=F\circ\iota_\mathcal{A}$ is left small, a full bimodule, a faithful functor, and surjects onto the compact operators between the modules in its image.
\end{thm}
\begin{proof}
Suppose $F$ is an equivalence. Firstly, the Eilenberg-Watts theorem guarantees that $F\cong -\bar{\otimes}_\mathcal{A}E$. But then by \Cref{equivimprim} we see that $E$ is an imprimitivity $\mathcal{A}-\mathcal{B}$ module. So by \Cref{imprimchar}, we see that $E$ is an isometry onto the compacts between modules in its image.

The converse is even simpler; a bimodule with the specified properties is always an imprimitivity bimodule and by \Cref{imprimisinvertible}, we see that $-\bar{\otimes}_\mathcal{A} E$ is a strong unital unitary equivalence.
\end{proof}

We will hereafter describe such an equivalence $F$ as a \emph{Morita equivalence} between $\mathcal{A}$ and $\mathcal{B}$.

An interesting corollary to this states that such an equivalence must necessarily preserve compact operators. Recall from \Cref{multequiv} that a multiplier equivalence between $\A$ and $\B$ is a pair of non-degenerate functors $F:\A\rightarrow\MB$ and $G:\B\rightarrow\MA$ whose lifts $\bar{F},\bar{G}$ to the respective multiplier categories give an equivalence.

\begin{cor}\label{compactdetermine}
If a unital functor $G:\Hilb\mathcal{A}\rightarrow\Hilb\mathcal{B}$ is a strong equivalence, then it restricts to a multiplier equivalence $G|_{\KHilb\mathcal{A}}:\KHilb\mathcal{A}\rightarrow\KHilb\mathcal{B}$. 

Conversely, any multiplier equivalence $\Gamma:\KHilb\mathcal{A}\rightarrow\KHilb\mathcal{B}$ extends to a strong unital equivalence $G:\Hilb\mathcal{A}\rightarrow\Hilb\mathcal{B}$. 
\end{cor}
\begin{proof}
For the rightward implication, note that \Cref{moritathm}  tells us that $G$ is given by tensoring with an $\mathcal{A}-\mathcal{B}$ bimodule $E$ which acts entirely by compact operators. So by \Cref{compactpreserve} we see that $G$ preserves compact operators.

For the converse, note that the extension $\bar{\Gamma}:\mathcal{M}(\KHilb\mathcal{A})\rightarrow\mathcal{M}(\KHilb\mathcal{B})$ provided by \Cref{nondegenstrongthm} is necessarily unital and ultrastrongly continuous on bounded subsets. So then by \Cref{2topologies}, $\Bar{\Gamma}$ is strong and unital.
\end{proof}
We are long overdue some examples of Morita equivalences of $C^*$-categories, other than the Yoneda bimodules which are self-Morita equivalences by \Cref{yonedaisid}. Luckily, the results in this subsection set us up well to provide these. Specifically, by \Cref{imprimisinvertible}, if we can provide an example of an $\mathcal{A}-\mathcal{B}$ imprimitivity bimodule, then we know there's a unital, strong, unitary equivalence between $\Hilb\mathcal{A}$ and $\Hilb\mathcal{B}$.  We will begin with an example that already played a role in Section 2:

\begin{lemma}
    If $\mathcal{A}$ is a $C^*$-category, there is an $\mathcal{A}_{\oplus}-\mathcal{A}$ imprimitivity bimodule given by sending $x_1\oplus\cdots\oplus x_n$ to $h_{x_1}\oplus\cdots\oplus h_{x_n}$.
\end{lemma}
\begin{proof}
    It follows from \Cref{compactdirect} and \Cref{imageyoneda} that $$\mathcal{A}_{\oplus}(x_1\oplus\dots\oplus x_n,y_1\oplus \dots\oplus y_k)\cong\mathcal{K}(h_{x_1}\oplus\cdots\oplus h_{x_n},h_{y_1}\oplus\cdots\oplus h_{y_k})                                  $$ naturally, so the described assignment is functorial and isometric onto the compact operators. This bimodule is left small by \Cref{containsrepssmall}. Hence as this bimodule is certainly full, by \Cref{moritathm} we see the described functor is a Morita equivalence.
\end{proof}

We show next that two unital $C^*$-categories are Morita equivalent if and only if they are Morita equivalent as additive categories:
\begin{prop}\label{unitalmorita}
    Two unital $C^*$-categories $\A$, $\B$ are Morita equivalent if and only if there is a unitary equivalence $\A^{\natural}_{\oplus}\cong\B^{\natural}_{\oplus}$.
\end{prop}
\begin{proof}
Recall from \Cref{unitalyoneda} that for a unital $C^*$-category $\A$ there's an equivalence between $\A_{\oplus}^{\natural}$ (the closure of $\A$ under projections and direct sums) and the category $\fgProj\A$ of finitely generated projective $\A$-modules.

Suppose $\A$ and $\B$ are unital and $E$ is an $\A-\B$ imprimitivity bimodule: then $E^1:\A\rightarrow\Hilb\B$ is non-degenerate by \Cref{bihilbnondegen}, so it must in fact be unital, but by \Cref{imprimaction} it must also have image in compact operators, so by \Cref{compactidfgp} we see that for each $x\in\Ob\A$, the module $E^1(x)$ is finitely generated and projective. Hence the functor $-\bar\otimes_{\A} E^1$ sends representable modules to finitely generated projectives, and hence sends all finitely generated projectives to finitely generated projectives. We can make a similar argument on the inverse functor $\bar\otimes_{\B} E^2$ and then we obtain by  the equivalence $\A^{\natural}_{\oplus}\cong\B^{\natural}_{\oplus}$.

On the other hand, suppose there's an equivalence $F:\fgProj\A\xrightarrow{\cong}\fgProj\B$: let $E^1$ be the right Hilbert $\A-\B$ bimodule given by postcomposing $F$ with the inclusion $\fgProj\B\xhookrightarrow{}\Hilb\B$ and precomposing with the Yoneda embedding $\A\xhookrightarrow{}\fgProj\A$, and $E^2$ be the $\B-\A$ bimodule similarly obtained from the inverse of $F$. Then it is evident that $E^1\bar{\otimes}_{\A}E^2\cong\iota_{\A}$ and $E^2\bar{\otimes}_{\B}E^1\cong\iota_{\B}$. \end{proof}

This means our  terminology coincides with that in \cite{IvoGoncalo}, where Morita equivalence of (unital) $C^*$-categories was \emph{defined} by the second criterion above.

Recall from \Cref{matrixalg} that $\Mat\mathcal{A}$ is the direct limit of the endomorphism algebras of \textit{non-repeating} sums. Using the results we have at our disposal it is now very easy to prove that the algebra $\Mat\mathcal{A}$ is Morita equivalent to $\mathcal{A}$.

\begin{prop}[{c.f. \cite[Theorem 6.3.1]{Ferrier}}]\label{catalgequiv}
If $\mathcal{A}$ is a small $C^*$-category, there is a  $(\Mat\mathcal{A})-\mathcal{A}$ imprimitivity bimodule given by letting the former act on the $\mathcal{A}$-module $\bigoplus_{x\in\Ob\mathcal{A}}h_x$. 

More generally, if $\A$ is a locally small $C^*$-category with a small full subcategory $\B$ such that every morphism in $\A$ is a norm-limit of finite sums of morphisms factoring through some object of $\B$, then $\bigoplus_{x\in\Ob\mathcal{B}}h_x$ is a $(\Mat\B)-\A$ imprimitivity bimodule. 

\end{prop}
\begin{proof}
  Note firstly that both bimodules are left small since they have a single-object $C^*$-category acting on the left.

   For the first case, note that  $\Mat\mathcal{A}\cong\bigoplus_{x,y\in\Ob\mathcal{A}}\mathcal{A}(x,y)\cong\bigoplus_{x,y\in\Ob\mathcal{A}}\mathcal{K}(h_x,h_y) $ (with the obvious multiplication on the latter two algebras), which by \Cref{compactdirect} is isomorphic to $\mathcal{K}(\bigoplus_{x\in\Ob\mathcal{A}}h_x)$. The module $\bigoplus_{x\in\Ob\mathcal{A}}h_x$ is clearly full, so tensoring with this bimodule gives a Morita equivalence by \Cref{moritathm}.

   The second result follows similarly as the requirement on $\B$ equates to asking that $\bigoplus_{x\in\Ob\mathcal{B}}h_x$ is full. 
   \end{proof}



 These results are a significant generalization of those obtained in \cite{JoachimKHom}, where it was proved that (translated to our terminology) $\Mat\mathcal{A}$ and $\mathcal{A}$ are Morita equivalent in the case where $\mathcal{A}$ is unital and has a countable set of objects.  

\section{(Bi)categories of $C^*$-categories and $C^*$-algebras}
In this section, we use the results in this article to prove results on a variety of categories whose objects are $C^*$-categories and $C^*$-algebras. We start by looking at several bicategories of $C^*$-categories.

\subsection{The bicategories $\Hcat$ and $\Bimod$}

We will assume basic results on bicategories and refer the reader to \cite{JohnsonYau} for a primer. 

Our first bicategory has evidently strictly associative and unital compositions, so is in fact a 2-category:
\begin{defn}
    We let $\Hcat$ be the 2-category defined as follows:
    \begin{itemize}
    \item Its objects are locally small $C^*$-categories.
    \item The 1-morphisms from $\mathcal{A}$ to $\mathcal{B}$ are strong unital $C^*$-functors from $\Hilb\mathcal{A}$ to $\Hilb\mathcal{B}$.
    \item  The 2-morphisms are bounded adjointable natural transformations between these functors. 
    \item $\Hcat$ has 1- and 2-composition laws are defined in the obvious way.
    \end{itemize}
\end{defn} 
Our second bicategory is an honest bicategory:
\begin{prop}
    There is a bicategory $\Bimod$ defined as follows: \begin{itemize} \item Its objects are locally small $C^*$-categories.
    \item  The 1-morphisms from $\mathcal{A}$ to $\mathcal{B}$ are non-degenerate right Hilbert $\mathcal{A}-\mathcal{B}$ bimodules, and horizontal identities are the Yoneda bimodules.
    \item The 2-morphisms are bounded adjointable natural transformations.
    \item Composition of 1-morphisms is given by tensor products of bimodules\footnotemark\vspace{0pt}. 
    \item Composition of 2-morphisms is given by composition of transformations.
    \item Associators provided by \Cref{tensorassociative} and unitors by \Cref{yonedaisid}. \end{itemize}
    The equivalences in this bicategory are given by imprimitivity bimodules.
\end{prop}
\begin{proof}
   It is easily verified that the associators satisfy the pentagon identity and that the unitors satisfy the triangle identity: we leave the details to the reader.

   The last statement follows directly from \Cref{imprimisinvertible} and \Cref{inversesareimprim}.
\end{proof}
\footnotetext{Recall from \Cref{nondegencompo} that a tensor product where the left bimodule is non-degenerate is again non-degenerate.}

This bicategory is analogous to similar ones defined over rings and operator algebras, see \cite{Brouwer}. 

The results in this article let us show without much effort that these two bicategories are equivalent:

\begin{thm}\label{bicatequiv}
There is a biequivalence $\Psi:\Bimod\rightarrow\Hcat$ given by: 

\begin{itemize}
    \item The identity on objects.
    \item Sending every bimodule $E:\mathcal{A}\rightarrow\Hilb\mathcal{B}$ to the functor 
    
    $-\bar{\otimes}_\mathcal{A} E:\Hilb\mathcal{A}\rightarrow\Hilb\mathcal{B}$.
    \item Sending a bounded adjointable transformation $T:E\rightarrow E'$ of right Hilbert $\mathcal{A}-\mathcal{B}$-bimodules to the bounded adjointable natural transformation 
    
    $\Psi(T):-\bar{\otimes}_\mathcal{A} E\rightarrow -\bar{\otimes}_\mathcal{A} E'$ whose existence is given by \Cref{lefttensorfunctorial}.
    \item Natural equivalences $\gamma:\Psi(E\bar{\otimes}_{\B}F)\xrightarrow[]{\cong}\Psi(F)\circ\Psi(E)$ given for every right Hilbert $\A$-module $M$, $\A-\B$ bimodule $E$, and $\B-\mathcal{C}$ bimodule $F$ simply by the associator isomorphism $M\bar{\otimes}_{\A}(E\bar{\otimes}_{\B}F)\cong (M\bar{\otimes}_{\A}E)\bar{\otimes}_{\B}F$ from \Cref{tensorassociative}.
    \item Natural equivalences $I:\Psi(\iota_{\A})\cong\id_{\A}$ given for every right Hilbert $\A$-module $M$ by the unitor isomorphism $M\bar{\otimes}_{\A} (\iota_{\A})\cong M$ from \Cref{yonedaisid}.
\end{itemize}
\end{thm}
\begin{proof}
It is straightforward to verify that this data gives a pseudofunctor: we leave this to the reader.

     To show that $\Psi$ is a biequivalence, we employ the criterion in e.g. \cite[Theorem 7.4.1]{JohnsonYau} which says $\Psi$ is a biequivalence if and only if it is essentially surjective on objects, and gives an equivalence  $\Psi:\Bimod(\mathcal{A},\mathcal{B})\rightarrow\Hcat(\mathcal{A},\mathcal{B})$ of 1-categories for any $\mathcal{A}$ and $\mathcal{B}$. The essential surjectivity is immediate.

     To show that $\Psi:\Bimod(\mathcal{A},\mathcal{B})\rightarrow\Hcat(\mathcal{A},\mathcal{B})$ is an equivalence of 1-categories, we show that it is essentially surjective and fully faithful. 

     For the first claim, suppose $F:\Hilb\mathcal{A}\rightarrow\Hilb\mathcal{B}$ is a strong unital functor; then by \Cref{EWthm} we know that there is a natural isomorphism $F\cong-\bar\otimes_{\mathcal{A}}(F\circ\iota_{\mathcal{A}})$, hence $F$ is in the essential image under $\Psi$ of $\Bimod(\mathcal{A},\mathcal{B})$.

     To show the faithfulness, we note simply that $\Psi(T):-\bar{\otimes}_\mathcal{A} E\rightarrow -\bar{\otimes}_\mathcal{A} E'$ whiskers through the inclusion $\iota_\mathcal{A}:\mathcal{A}\rightarrow\Hilb\mathcal{A}$ to $T:E\rightarrow E'$, so if two transformations $T,T':E\rightarrow E'$ have the same image under $\Psi$, they must already be the same, showing $\Psi$ is faithful.

     To show fullness, let $\tau:-\bar{\otimes}_\mathcal{A} E\rightarrow -\bar{\otimes}_\mathcal{A} E'$ be a bounded adjointable natural transformation. Let $T:E\rightarrow E'$ be the whiskering of $\tau$ through $\iota_\mathcal{A}$. It is clear by definition that $\Psi(T)$ and $\tau$ agree on representable $\mathcal{A}$-modules, but then by \Cref{repdetermine} they must agree everywhere. Hence $\Psi$ is full.
\end{proof}

This theorem tells us that the 2-category $\Hcat$ is a model for the \emph{strictification} of the bicategory $\Bimod$. We also record that the Morita equivalences in \Cref{catalgequiv} provide us with further biequivalences.

\begin{prop}\label{4bicatequiv}
    The following bicategories are biequivalent:
    \begin{itemize}
        \item The full sub-bicategory of $\Bimod$ whose objects are small $C^*$-categories.
        \item The full sub-2-category of $\Hcat$ whose objects are small $C^*$-categories.
        \item The full sub-bicategory of $\Bimod$ whose objects are the one-object categories, i.e. $C^*$-algebras.
        \item The full sub-2-category of $\Hcat$ whose objects are the $C^*$-algebras.
    \end{itemize}
\end{prop}
\begin{proof}
    The equivalence of the first two bicategories follows from the preceding theorem: note that $\Psi$ is the identity on objects.

    The inclusion of the third into the first category is an equivalence since by \Cref{catalgequiv} it is essentially surjective on objects, as is the inclusion from the fourth into the second; hence again by \cite[Theorem 7.4.1]{JohnsonYau} these inclusions are biequivalences.
\end{proof}
 Variants of several categories in this proposition were studied in other works: the third category, often labelled the `correspondence bicategory', was first defined in  \cite{LandsmanBicat} and further studied in \cite{BussMeyerZhu}. The equivalence between the first and third bicategory was proven in \cite[Section 7]{Ferrier}.

 The $C^*$-algebraic and $C^*$-categorical outlooks may both offer advantages: for example, in the algebraic setting it is easier to speak of the \textit{dimension} of a bimodule, whereas in the categorical setting it may be easier to form (bi)(co)limits. The results in this section offer a convenient way to move between these perspectives.
 
We characterize briefly an interesting sub-2-category of $\Hcat$ that relates to other works and will become relevant in the next subsection:

\begin{defn}
For any locally small $C^*$-category $\A$, let $\Hilbrep\A$ be the category of representable right Hilbert $\A$-modules, i.e. the essential image of the Yoneda embedding $\iota_{\A}:\A\rightarrow\Hilb\A$.

Let $\Hcatrep$ be the subcategory of $\Hcat$ with the same objects but whose 1-morphisms are strong unital functors $\Hilb\A\rightarrow\Hilb\B$ that preserve representable modules, i.e. they send $\Hilbrep\A$ to $\Hilbrep\B$.

Finally, let $\CCatndg$ be the 2-category whose objects are locally small categories, whose 1-morphisms are non-degenerate functors $\A\rightarrow\MB$ and whose 2-morphisms are natural transformations of functors.
\end{defn}
\begin{prop}
   There is a 2-equivalence $\Xi:\Hcatrep\rightarrow\CCatndg$ given by \begin{itemize}
       \item The identity on objects
       \item Sending a strict unital functor $F:\Hilb\A\rightarrow\Hilb\B$ to the functor $\overline{\iota_{\B}}^{-1}\circ F\circ\iota_{\A}$, where $\overline{\iota_{\B}}^{-1}$ is an inverse functor to the `multiplier Yoneda embedding' $\overline{\iota_B}:\MB\rightarrow\Hilbrep\B$.
       \item Sending a natural transformation $\eta:F\Rightarrow F'$ of functors to the natural transformation $\Xi(\eta):\overline{\iota_{\B}}^{-1}\circ F\circ\iota_{\A}\Rightarrow\overline{\iota_{\B}}^{-1}\circ F'\circ\iota_{\A}$ given by whiskering through the two functors $\iota_{\A}$ and $\overline{\iota_{\B}}^{-1}$.
   \end{itemize}
\end{prop}
\begin{proof}
As before it is elementary to show that $\Xi$ is a pseudofunctor, and as it is the identity on objects, we just need to show that its components  $\Xi(\A,\B):\Hcatrep(\A,\B)\rightarrow\CCatndg(\A,\B)$ are equivalences.

As in the proof of \Cref{bicatequiv}, we obtain faithfulness by \Cref{repdetermine}: a natural transformation of functors on Hilbert module categories is determined by its values on the representables.

For fullness, considering a natural transformation $\eta:F\Rightarrow F'$ of functors $\A\rightarrow\MB$ as a transformation $\overline{\eta}$ of bimodules $\A\rightarrow\Hilb\B$. Then by \Cref{lefttensorfunctorial} this gives a natural transformation $\id\bar{\otimes}_{\B}\overline{\eta}$ of tensor functors $\Hilb\A\rightarrow\Hilb\B$, and one easily sees that $\Xi(\id\bar{\otimes}_{\B}\overline{\eta})=\eta$.

Essential surjectivity follows simply by noting that a non-degenerate functor $\A\rightarrow\MB$ gives a non-degenerate bimodule by postcomposing with $\overline{\iota_{\B}}:\MB\rightarrow\Hilb\B$, and tensoring a representable $\A$-module with this bimodule then gives a representable $\B$-module by \Cref{tensorextend}.
\end{proof}
Several full subcategories of the bicategory $\CCatndg$ were studied in \cite{AntounVoigt}: in particular, they prove that the full sub-bicategory of $\CCatndg$ whose objects are small $C^*$-categories containing all countable direct sums, which they call $\text{C}^*\text{Lin}$, is closed under bicolimits.


\subsection{Localizing a 1-category of $C^*$-categories at the Morita equivalences}

In this subsection we construct a category of $C^*$-categories where Morita equivalences have been inverted, i.e. made into isomorphisms: we will moreover do this in a \emph{universal} way, so that the end result constitutes a localization of 1-categories. To be specific, we will invert the Yoneda embeddings, so that isomorphisms between $\Hilb\A$ and $\Hilb\B$ give isomorphisms between $\A$ and $\B$. We establish a \emph{reflexive localization}, which gives excellent control over the hom-sets of a localization, as done in recent works like \cite{bunke2023kk}.

In the lemmas below, all categories and all functors that are not $C^*$-categories, resp. $C^*$-functors, are unital unless otherwise stated.  Recall first the definition of a localization. 
\begin{defn}
Let $\C$ be a category and $W$ be a class of morphisms in $\C$. The \emph{localization of $\C$ at $W$}, if it exists, is a category $\C[W^{-1}]$ admitting a functor $Q:\C\rightarrow\C[W^{-1}]$ such that for all categories $\D$, the precomposition functor $$ -\circ Q:\textrm{Fun}(\C[W^{-1}],\D)\rightarrow \textrm{Fun}(\C,\D)$$ is fully faithful and has essential image consisting of those functors $F:\C\rightarrow\D$ that send all morphisms in $W$ to isomorphisms.

\end{defn}

Our localizations will be of a special type, namely reflective localizations. These are localizations which are at once adjoints:
\begin{defn}
The localization $Q:\C\rightarrow\C[W^{-1}]$ of $\C$ at $W$ is called a \emph{reflective localization} if it has a fully faithful right adjoint $R:\C[W^{-1}]\rightarrow\C$ 
\end{defn}
Reflective localizations have several convenient equivalent characterizations:
\begin{prop}\label{reflecloc}
For a functor $R:\B\rightarrow\C$ of 1-categories with a left adjoint $Q:\C\rightarrow\B$, the following are equivalent:
\begin{enumerate}
\item $R:\B\rightarrow\C$ is a fully faithful functor.
\item $Q:\C\rightarrow\B$ exhibits $\B$ as the reflective localization of $\C$ at the class of morphisms which get sent to isomorphisms by $Q$.
\item $Q:\C\rightarrow\B$ exhibits $\B$ as the 
reflective localization of $\C$ at the class of units $\eta_c:c\rightarrow RQ(c)$
\item The counit $\epsilon:QR\Rightarrow\id_{\B}$ is an isomorphism of functors.
\end{enumerate}
\end{prop}
\begin{proof}
The equivalence of (1), (2), and (4) is standard, see for example \cite[Proposition 3.1]{nlab:reflective_subcategory}. To see how (3) follows from the other statements, we show that in the presence of statements (1) and (4), the two localizations must have the same universal property.

Note first that one of the adjunction axioms gives the commutativity of  the diagram 
$\begin{tikzcd}
Q(c) \arrow[r, "Q(\eta_c)"] \arrow[rr, "\textrm{id}", bend right] & QRQ(c) \arrow[r, "\epsilon_{Q(c)}"] & Q(c)
\end{tikzcd}$ so as the right morphism is an isomorphism by (4), the left one must be too. Hence if a functor sends all those morphisms to isomorphisms which are sent to isomorphisms under $Q$, it must certainly send all units $\eta_c$ to isomorphisms.

Suppose conversely that a functor $F:\C\rightarrow\D$ sends all units $\eta_c$ to isomorphisms, and let $g\in\C(c,c')$ be a morphism sent to an isomorphism by $Q$. Note that we can whisker the natural transformation $\eta$ through $F$ and obtain a commutative diagram $
\begin{tikzcd}
F(c) \arrow[r, "F(g)"] \arrow[d, "F(\eta_c)"] & F(c') \arrow[d, "F(\eta_{c'})"] \\
FRQ(c) \arrow[r, "FRQ(g)"]                    & FRQ(c')                        
\end{tikzcd}$ where the vertical arrows are isomorphisms by hypothesis, and the bottom arrow is an isomorphism since $Q(g)$ is. Hence $F(g)$ is an isomorphism. 
\end{proof}

We now apply this theory to the following category of $C^*$-categories:

\begin{defn} Let $\mathbf{C^*Cat}$ be the category whose objects are (locally small, not necessarily unital) $C^*$-categories, and whose morphisms from $\A$ to $\B$ are natural (unitary) isomorphism classes of non-degenerate functors $\A\rightarrow\mathcal{MB}$. \end{defn} 

This is of course a `1-truncation' of the bicategory $\CCatndg$ from the previous subsection. 

\begin{defn}
Let $\sKHilb: \mathbf{C^*Cat}\rightarrow \mathbf{C^*Cat}$ be the endofunctor that sends every $C^*$-category to its $C^*$-category of Hilbert modules and compact operators, and sends every non-degenerate functor $F:\A\rightarrow\MB$ to the tensor product with the non-degenerate bimodule given by composing with the Yoneda embedding $\iota_{\B}:\MB\rightarrow\Hilb\B$. Let $\mathbf{KHilb}$ be the full subcategory with objects those in the image of $\sKHilb$.
\end{defn}
\begin{prop}\label{hilbadjoint}
The functor $\sKHilb$ is left adjoint to the inclusion functor $\mathrm{incl}:\mathbf{KHilb}\rightarrow \mathbf{C^*Cat}$.
\end{prop}
\begin{proof}
We need to show that for any two locally small $C^*$-categories $\A\in\Ob\mathbf{C^*Cat}$ and $\B\in\Ob\mathbf{KHilb}$, there is a natural isomorphism of hom-sets $$ \mathbf{C^*Cat}(\A,\B)\cong\mathbf{C^*Cat}(\sKHilb\A,\B).$$ Since $\B$ is in the image of $\sKHilb$ we can write $\B=\sKHilb\C$. But then the first hom-set refers to non-degenerate $\A-\C$ bimodules and the second corresponds to strong unital functors from $\sKHilb\A$ to $\sKHilb\mathcal{C}$, and these are in correspondence by \Cref{EW}: this correspondence is easily seen to be natural in both $\A$ and $\C$.
\end{proof}
We hence derive the main result of this subsection:
\begin{thm}\label{MoritaLocalization}
    The functor $\sKHilb$ is the (reflective) localization of the category $\mathbf{C^*Cat}$ at the class of Yoneda inclusions $\siota$, as well as at the class of non-degenerate functors $\A\rightarrow\mathcal{MB}$ inducing Morita equivalence.
\end{thm}
\begin{proof}
   The functor $\mathrm{incl}$ from \Cref{hilbadjoint} is clearly full and faithful, so the adjunction satisfies criterion (1) in \Cref{reflecloc}. The Yoneda bimodules $\siota$ give the identity on $\sKHilb\A$ upon tensoring, so give us the units of the adjunction. Hence criterion (3) tells us that $\sKHilb$ is the localization at the class of Yoneda embeddings. Criterion (2) gives us the final statement.
\end{proof}

Hence, $\mathbf{KHilb}$ is the first model for the \textit{Morita homotopy theory} of $C^*$-categories, akin to similar results on dg-categories (\cite{ToenMorita}\cite{TabuadaAdditive}), unital $C^*$-categories (\cite{IvoGoncalo}), and $(\infty,1)$-categories (\cite{MoritaHtopyInfty}).

We also derive easily the idempotence of the functor $\sKHilb$, which can alternatively be deduced from \Cref{KHilbImprim}.
\begin{cor}
For any locally small $C^*$-category $\A$, the map $$\iota_{\sKHilb\A}: \sKHilb\A\rightarrow\sKHilb(\sKHilb\A)$$ is an isomorphism.
\end{cor}
\begin{proof}
We prove that in general in the set-up of \Cref{reflecloc}, for all $b\in\Ob\B$ the unit $\eta_{R(b)}:R(b)\rightarrow RQR(b)$ is an isomorphism. But the second adjunction axiom states that the diagram $\begin{tikzcd}
R(b) \arrow[r, "\eta_{R(b)}"] \arrow[rr, "\textrm{id}", bend right] & RQR(b) \arrow[r, "R(\epsilon_b)"] & R(b)
\end{tikzcd}$ commutes. But we know $\epsilon_b$ is an isomorphism so $R(\epsilon_b)$ must also be, and hence $\eta_{R(b)}$ is also an isomorphism.
\end{proof}
It follows that $C^*$-categories which are equivalent to one of the form $\sKHilb\A$ for some locally small $C^*$-category $\A$ are \emph{exactly} those $C^*$-categories whose Yoneda embedding is an equivalence: this is a partial answer to the questions raised about `pre-completeness' in \cite{henry2015complete}.

\printbibliography

\end{document}